\documentclass[reqno,10pt]{amsart}

\usepackage{latexsym}
\usepackage{amsmath,amsfonts,amssymb,epsfig}
\usepackage{amsthm,amscd}
\usepackage{mathrsfs} 
\usepackage{hyperref}
\usepackage{mathtools}
\usepackage{lscape}
\usepackage{array,booktabs}
\usepackage{enumitem}
\usepackage{svg}
\usepackage{multicol}
\usepackage[mathcal]{eucal}
\usepackage{setspace}
\usepackage{wasysym}
\usepackage{color, epsfig,url, graphicx}
\usepackage{graphicx,import}
\usepackage[all]{xy}          % for xy-pic pictures
\xyoption{dvips}              % for xy-pic pictures
\usepackage{enumitem}
\usepackage{tikz-cd}
\usepackage{float}

\makeatletter
\renewcommand\l@subsection{\@tocline{2}{0pt}{2pc}{5pc}{}}
\makeatother

\newtheorem*{rep@theorem}{\rep@title}
\newcommand{\newreptheorem}[2]{%
	\newenvironment{rep#1}[1]{%
		\def\rep@title{#2 \ref{##1}}%
		\begin{rep@theorem}}%
		{\end{rep@theorem}}} 

\theoremstyle{plain}
\newtheorem{thm}{Theorem}[section]
\newreptheorem{thm}{Theorem}
\newtheorem{thmp}{Theorem}[section] 
\newtheorem{corp}[thmp]{Corollary}

\newreptheorem{thmp}{Theorem} 
\newreptheorem{corp}{Corollary} 
\newtheorem{prop}[thm]{Proposition}
\newtheorem{lemma}[thm]{Lemma}
\newtheorem{cor}[thm]{Corollary}

\theoremstyle{definition}

\newtheorem{example}[thm]{Example}
\newtheorem{def/ex}[thm]{Definition/Example}
\theoremstyle{remark}
\newtheorem{rem}[thm]{Remark}

\def\phi{\varphi}

\newcommand{\no}{\noindent}
\newcommand{\bex}{\begin{example}\em}
	\newcommand{\eex}{\end{example}}

\newcommand{\R}{{\mathbb R}}

\newcommand{\Z}{{\mathbb Z}}

\newcommand{\bx}{{\mathbf x}}

\newcommand{\K}{{\mathscr{K}}}

\newcommand{\ev}{\operatorname{ev}}

\newcommand{\supp}{\operatorname{supp}}

\def\cprime{$'$} % for bibtex

\def\Cnf{\text{\rm Conf}}  % configuration spaces
\graphicspath{{figs/}}

\newcommand{\vvcenteredinclude}[2]{\begingroup
	\setbox0=\hbox{\includegraphics[scale=#1]{#2}}%
	\parbox{\wd0}{\box0}\endgroup}

\title[]{From integrals to combinatorial formulas\\ of finite type invariants - a case study}
\date{\today}

 \author{Robyn Brooks}
 
 \address{
 	Univeristy of Utah
 	Salt Lake City, Utah 84112 } 
 \email{Robyn.Brooks@utah.edu}
 
 \author{Rafa{\l} Komendarczyk}
 \address{
 	Tulane University,
 	New Orleans, Louisiana 70118 } 
 \email{rako@tulane.edu}
 
\begin{document}
%\onehalfspacing
\setcounter{section}{0}

\begin{abstract}
We obtain a localized version of the configuration space integral for the Casson knot invariant, where the standard symmetric Gauss form is replaced with a locally supported form. An interesting technical difference between the arguments presented here and the classical arguments is that the vanishing of integrals over hidden and anomalous faces does not require the well known ``involution tricks''. The integral formula easily yields the well-known arrow diagram expression for regular knot diagrams, first presented in the work by Polyak and Viro. Moreover, it yields an arrow diagram count for the multicrossing knot diagrams, such as petal diagrams and gives a new lower bound for the {\em {\"u}bercrossing number}. Previously, the known arrow diagram formulas were applicable only to the regular knot diagrams.
\end{abstract}

\maketitle

\section{Introduction}

Configuration space integrals,  originally introduced by Bott and Taubes in \cite{Bott-Taubes:1994} and motivated by the work of Kontsevich \cite{Kontsevich:1993, Kontsevich:1994}, are a far-reaching generalization of the celebrated Gauss linking number
 integral formula \cite{Gauss:1975} for a $2$--component link.  Given a smooth embedding of $S^1\sqcup S^1$ into $\R^3$, $L=(L_1, L_2):S^1\sqcup S^1\longrightarrow \R^3$, the linking number can be computed as:
\begin{equation}\label{eq:lk-number-integral}
 \operatorname{lk}(L_1,L_2)=\varint\limits_{S^1\times S^1} h^\ast_{2,1} \omega=\frac{1}{2\pi} \varint^{1}_0 \varint^{1}_0 \frac{\langle \dot{L}_1(s), \dot{L}_2(t), L_2(s)-L_1(t)\rangle}{|L_2(s)-L_1(t)|^3} ds\,dt,
\end{equation}
where $h_{12} :S^1\times S^1\longrightarrow S^2$ denotes the Gauss map of $L$: $h_{12}(s,t)=\frac{L_2(s)-L_1(t)}{|L_2(s)-L_1(t)|}$, and $h^\ast_{12} \omega$ the pullback of the 
rotationally symmetric  unit area form  $\omega$ over $S^2$.
Note that the domain $S^1\times S^1$ can be regarded as the ordered configuration space $\Cnf(S^1\sqcup S^1; 1,1)$ 
of a pair of points in $S^1\sqcup S^1$, with the first and second point restricted to the first and second factor in $S^1\sqcup S^1$ respectively.  Recall, given a topological space $X$, the {\em configuration space of $n$ points in $X$} is the ``deleted'' product
\[
\Cnf(X;n)=\{(x_1,\ldots,x_n)\in X^n\ |\ x_i\neq x_j \  \text{for} \ i\neq j\}.
\]

\subsection{Integral for the Casson knot invariant} General configuration space integrals for knots and links were investigated in number of works: \cite{Bott-Taubes:1994, Cattaneo-Cotta-Ramusino-Longoni:2002, Koytcheff-Munson-Volic:2013, Thurston:1999, Volic:2007}, and involve integrals over various pieces of configuration spaces of points {\em on the knot} and {\em off the knot} (these pieces are labeled by trivalent diagrams).
Besides the linking number, one of the simplest such integrals computes the second coefficient of the Conway polynomial $c_2=c_2(K)$, \cite{Conway:1970} of a closed knot $K:I\longrightarrow \R^3$, $I=[0,1]$, $K(0)=K(1)$, which is also known as the {\em Casson knot invariant}, \cite{Polyak-Viro:2001}. In the following, we denote the space of classical pointed knots by $\mathscr{K}_{S^1}$, i.e. the smooth pointed embeddings of $S^1$ in $\R^3$. First rigorously established in the works of Bott and Taubes \cite{Bott-Taubes:1994}, Bar-Natan \cite{Bar-Natan:1991}, and later generalized in \cite{Thurston:1999, Volic:2007}, the integral reads as follows:
\begin{equation}\label{eq:casson-integral-original}
\begin{split}
 c_2(K)  & = \frac{1}{4}\varint\limits_{\Cnf(I;4)} \omega_{13}\wedge\omega_{24}-\frac{1}{3}\varint\limits_{\Cnf(I,K;3,1)} \omega_{1\mathbf{4}}\wedge\omega_{2\mathbf{4}}\wedge\omega_{3\mathbf{4}},
 \end{split}
\end{equation}
where\footnote{Note that these are connected components of the respective open configuration spaces, and we use the same notation $\Cnf$ for them.}
\begin{equation}\label{eq:connected-comp-conf}
\begin{split}
\Cnf(I;4) & \cong\bigl\{(x_1,x_2,x_3,x_4)\in I^4\ |\ 0\leq x_1<x_2<x_3<x_4\leq 1\bigr\},\\
\Cnf(I,K;3,1) & \cong\bigl\{(x_1,x_2,x_3,K,\mathbf{x})\ |\ (x_1,x_2,x_3)\in I^3; 0\leq x_1<x_2<x_3\leq 1; \\ 
& \qquad\qquad\qquad\qquad\quad  \mathbf{x}\in \R^3-\{K(x_1), K(x_2), K(x_3)\}; K\in \mathscr{K}_{S^1} \bigr\},
\end{split}
\end{equation}
and integral terms in \eqref{eq:casson-integral-original} can be explicitly written as 
\begin{multline*}%\label{eq:casson-integral-terms}
 \varint\limits_{\Cnf(I;4)} \omega_{13}\wedge\omega_{24}=\frac{1}{(4\pi)^2} \varint\limits_{\Cnf(I;4)} \prod^2_{i=1}\frac{\bigl\langle K(x_{i+2})-K(x_i),\dot{K}(x_i),\dot{K}(x_{i+2})\bigr\rangle}{|K(x_{i+2})-K(x_i)|^3}  dx_1 dx_2 dx_3 dx_4,\\ \varint\limits_{\Cnf(I,\K_{S^1};3,1)} \omega_{1\mathbf{4}}\wedge\omega_{2\mathbf{4}}\wedge\omega_{3\mathbf{4}}=\frac{1}{(4\pi)^3}\varint\limits_{\Cnf(I,\K_{S^1};3,1)} \operatorname{det}\bigl(E_1(x_1,\mathbf{x}),E_2(x_2,\mathbf{x}),E_3(x_3,\mathbf{x})\bigr) dx_1 dx_2 dx_3 d\mathbf{x},
\end{multline*}
where $E_i(x_i,\mathbf{x})=\tfrac{(K(x_i)-\mathbf{x})\times \dot{K}(x_i)}{|K(x_i)-\mathbf{x}|^3}$, $i=1,2,3$. As in the case of the linking number, the $2$--forms $\omega_{ij}$ in \eqref{eq:casson-integral-original} are obtained as pullbacks of the standard unit area form $\omega$ on $S^2$, along the {\em Gauss maps}: 
\[
 \tilde{h}_{ij}(x_i,x_j)=\frac{K(x_j)-K(x_i)}{|K(x_j)-K(x_i)|},\qquad \tilde{h}_{i\mathbf{4}}(x_i,\mathbf{x})=\frac{\mathbf{x}-K(x_i)}{|\mathbf{x}-K(x_i)|}.
\]
One immediate issue in establishing integrals of type \eqref{eq:casson-integral-original} is addressing the divergence of forms $\omega_{ji}$ as $x_i$ and $x_j$ or $\mathbf{x}$ approach each other. 
This is achieved by replacing the domains of integration $\Cnf(I;4)$ and $\Cnf(I,K;3,1)$ with their compactifications $\Cnf[I;4]$ and $\Cnf[I,K;3,1]$, \cite{Volic:2007,Bott-Taubes:1994}. In order to simplify our considerations, in place of closed knots $\K_{S^1}$, we consider the space of {\em long knots} $\K=\K_{\R}$ \cite{Polyak-Viro:2001, Volic:2007, Koytcheff-Munson-Volic:2013}, i.e. smooth embeddings $K:\R\longrightarrow \R^3$ of $\R$ into $\R^3$, which are standard outside some box, e.g. the image of $K$: $\{K(s)\ |\ s\in \R\}$ is contained in the cube $-1\leq x,y,z\leq 1$ in $\R^3$ for $-1\leq s\leq 1$ and equal to the $x$--axis outside the cube $s\geq 1$ and $s\leq -1$, see Figure \ref{fig:figure-8-plane-diagram}. The integral \eqref{eq:casson-integral-original} can be defined analogously for long knots, by replacing the configuration spaces in \eqref{eq:connected-comp-conf} over $S^1$, with the ones over $\R$ and appropriately adjusting the coefficients in \eqref{eq:casson-integral-original}, which is further reviewed in Section \ref{sec:bott-taubes}.

%%%%%%%%%%%%%%%%%%%%%%%%%%%%%%%%%%%%%%%%%
\begin{figure}[ht]
\begin{minipage}{.48\linewidth}
	\centering
	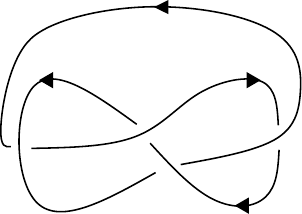
\end{minipage}
\begin{minipage}{.48\linewidth}
	\centering
	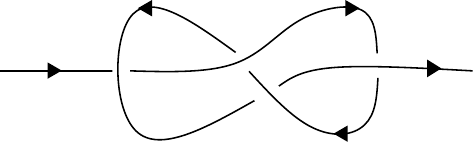
\end{minipage}
	\caption{Plane diagrams for the figure 8 long knot (right) and figure 8 knot (left)}
	\label{fig:figure-8-plane-diagram}
\end{figure}
%%%%%%%%%%%%%%%%%%%%%%%%%%%%%%%%%%%%%%%%%
\subsection{From integrals to combinatorial formulas}\label{sec:integrals-combinatorial} An interesting property of the Gauss linking number integral \eqref{eq:lk-number-integral} is that it easily converts to 
a combinatorial formula for $\operatorname{lk}(L_1,L_2)$. Indeed, \eqref{eq:lk-number-integral} is an integral expression for the degree of the Gauss map $h_{12}$.  
Therefore, whenever $L$ has a regular planar projection\footnote{i.e. with transverse double crossings.} (onto the plane orthogonal to\footnote{we just need any fixed direction vector, w.l.o.g. we choose $N$ (the {\em North})} $N=(0,0,1)$), the local degree formula \cite{Hatcher:2002} applied to the Gauss map parametrized by $h_{12}:(s,u)\xrightarrow{L} (L_1(s),L_2(u))$ and the regular value $N$, yields a signed crossing count formula for $\operatorname{lk}(L_1,L_2)$:
%%%%%%%%%%%%%%%%%%%%%%%%%%%%%
\begin{equation}\label{eq:lk-crossing-count}
\operatorname{lk}(L_1,L_2)=\deg(h_{12})=\sum_{s\in h^{-1}_{12}(N)} \operatorname{sign}(D h_{12}(s))=\sum_{s\in \operatorname{Cross}(L)} \operatorname{sign}(s),
\end{equation}
$\operatorname{Cross}(L)$ is the set of overcrossings, $L_2$ over $L_1$.
%%%%%%%%%%%%%%%%%%%%%%%%%%%%%%%%%%%%%%%%%
\begin{figure}[h!]
	\centering
	\includegraphics[scale=.7]{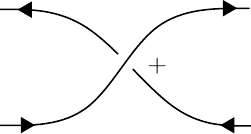}\qquad\qquad \includegraphics[scale=.7]{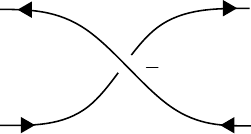}
	\caption{A positive crossing (left) and negative crossing (right).}
	\label{fig:knotcrossings}
\end{figure}
%%%%%%%%%%%%%%%%%%%%%%%%%%%%%%%%%%%%%%%%%
It is a classical argument, c.f. \cite[p. 41]{Bott-Tu:1982}, that the combinatorial formula \eqref{eq:lk-crossing-count} can be obtained directly from the integral \eqref{eq:lk-number-integral},  by replacing the area form $\omega$ on $S^2$ (the generator of $H^2(S^2)$) with a unit area bump form $\eta_N$, supported on an $\varepsilon$--disk $D^2_\varepsilon\subset S^2$, centered at the regular value $N$ of $h_{12}$, i.e.
%%%%%%%%%%%%%%%%%%%%%%%%%%%%%
\begin{equation}\label{eq:eta_N}
\eta_N=\eta^\varepsilon_N=f_{N,\varepsilon}\omega,\qquad \supp(f_{N,\varepsilon})\subset D^2_\varepsilon,\qquad \varint_{S^2} \eta_N=1,
\end{equation}
where $f_{N,\varepsilon}$ is just a real valued bump function supported on the disk $D^2_\varepsilon\subset S^2$ around $N$.
The $2$--form $\eta_N$ is not anti-symmetric i.e. 
\begin{equation}\label{eq:antipodal-eta_N}
 A^\ast \eta_N=-\eta_S,
\end{equation}
under the antipodal map $A:\xi\longrightarrow -\xi$ of $S^2$,
where $S$ is the antipodal point of $N$.
Indeed, substituting $\omega=d\alpha+\eta_N$ Stoke's theorem yields
\begin{equation}\label{eq:lk-eta_N}
\operatorname{lk}(L_1,L_2) =\varint\limits_{S^1\times S^1} h^\ast_{12} \omega=\varint\limits_{S^1\times S^1}  h^\ast_{12} \eta_N =\sum_{s\in h^{-1}_{12}(N)} \varint_{U_s} h^\ast_{12} \eta_N=\sum_{s\in h^{-1}_{12}(N)} \operatorname{sign}(Dh_{12}(s)),
\end{equation}
where $U_s\subset h^{-1}_{12}(D^2_\varepsilon)$ are disjoint neighborhoods of finitely many crossing points $s\in h^{-1}_{12}(N)$ such that $h_{12}|_{U_s}$ is a diffeomorphism (for $\varepsilon$ small enough), by the inverse function theorem. The last identity in the above computation is simply the  
 change of variables for the integral in \eqref{eq:eta_N}, and $\operatorname{sign}(Dh_{12}(s))$ depends  if $h_{12}|_{U_s}$ is orientation preserving or reversing.
 
 A natural question which arises for $c_2(K)$, as well as other finite type invariants, is whether one may derive a combinatorial expressions in the spirit of the derivation \eqref{eq:lk-crossing-count}--\eqref{eq:lk-eta_N}. In the following theorem, we answer this question positively for long knots in the subspace of $\K$:
 \begin{equation}\label{eq:K_N}
  \K_{N}=\{K\in \K\ |\ \dot{K}(s)\nparallel N,\ s\in \R \},
 \end{equation}
 i.e. the space of embeddings, which are transverse to the fixed direction parallel to the vector $N$ in \eqref{eq:eta_N}. 
 Note that paths in  $\K_{N}$ are knot isotopies that allow for the Reidemeister {\bf II} and {\bf III} moves but not the {\bf I} move (in short \textbf{RII} and \textbf{RIII}, but not \textbf{RI}), \cite{Alexander:1928}, \cite{Reidemeister:1927}. Thus, each path connected component of $\K$ is a union of path connected components of $\K_N$, where knots in different components differ only by the Reidemeister {\bf I} move. 
 \begin{thmp}\label{thm:c2-integral}
 	$(i)$ Given a long knot $K$ in $\K_{N}$, we have the integral invariant\footnote{for any $K\in \K_{N}$ the sequence under the limit becomes eventually constant} 
 	\begin{equation}\label{eq:casson-integral-eta}
 	\begin{split}
 	I(K)  & = \lim_{\varepsilon\to 0}\ \Bigl(\varint\limits_{\Cnf[\R;4]} h^\ast_{13}\eta^\varepsilon_N\wedge h^\ast_{42}\eta^\varepsilon_N-\varint\limits_{\Cnf[\R,\K_N;3,1]} h^\ast_{1\mathbf{4}}\eta^\varepsilon_N\wedge h^\ast_{\mathbf{4}2}\eta^\varepsilon_N\wedge h^\ast_{3\mathbf{4}}\eta^\varepsilon_N\Bigr).
	 \end{split}
 	\end{equation} 	
 	$(ii)$ If $K$ admits the regular double--crossing projection onto the plane orthogonal to $N$, we obtain
	\begin{equation}\label{eq:gd_c2_dbl-crossings}
	I(K)=c_2(K)=\langle \vvcenteredinclude{.3}{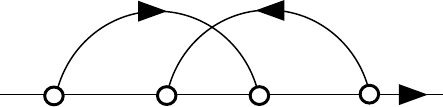},G_K\rangle.
	\end{equation}
	$(iii)$  $I(K)$ descends to the knot invariant on $\K$, i.e.  $I:\pi_0(\K)\longrightarrow \Z$.
 \end{thmp}
The combinatorial expressions in \eqref{eq:gd_c2_dbl-crossings}  are known as the arrow diagram formulas and were first introduced by Polyak and Viro in \cite{Polyak-Viro:1994}, and then further developed in subsequent works \cite{Polyak-Viro:2001}, \cite{Chmutov-Duzhin-Mostovoy:2012}. We further obtain an adaptation of the arrow diagram formulas and Gauss diagrams to the setting of long knots with multiple crossings in the following

\begin{thmp}\label{thm:c2-multicrossing}
	 For the multiple--crossing Gauss diagram $G_K$, we have the following generalization of the formula in \eqref{eq:gd_c2_dbl-crossings} 
	\begin{equation}\label{eq:gd_c2_mult-crossings}
		\begin{split}
			I(K)=c_2(K) & =\langle \vvcenteredinclude{.3}{gd_c2_chord.pdf},G_K\rangle+ \tfrac{1}{2}\bigl(\langle \vvcenteredinclude{.3}{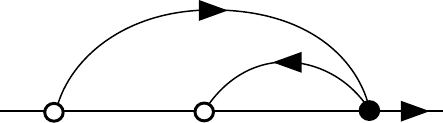},G_K\rangle\\
			&\qquad +\langle \vvcenteredinclude{.3}{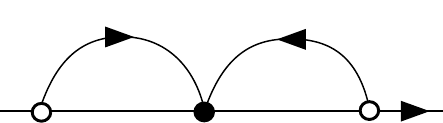},G_K\rangle+\langle \vvcenteredinclude{.3}{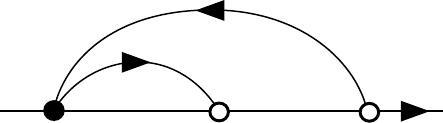},G_K\rangle\bigr).
		\end{split}
	\end{equation}
\end{thmp}
As a consequence of the above result we obtain a lower bound, in terms of $c_2(K)$, for the {\em {\"u}bercrossing number} $\textit{\"u}(K)$, as well as the {\em petal crossing number} $p(K)$, introduced\footnote{Note that $\textit{\"u}(K)\leq p(K)$.} in \cite{Adams:2012}.
\begin{corp}\label{cor:uber-crossing}
	\[
	|c_2(K)|\leq {\textit{\"u}(K) -1 \choose 3}+{ \textit{\"u}(K) -1 \choose 4}.
	\]
\end{corp}
\no We refer the reader to \cite{Adams:2015} for previous results concerning $\textit{\"u}(K)$.

%%%%%%%%%%%%%%%%%%%%%%%%%%%%%%%%%%%%%%%%%%%%
\subsection{Arrow diagrams}\label{sec:arrow-diagrams} A general (multi-crossing) \textit{Gauss diagram} of a long knot $K:\R\longrightarrow\R^3$, denoted $G_K$, is a graph (obtained from a planar projection of $K$) with:
\begin{itemize}
\item[$(i)$] a base strand - an oriented copy of $\R$, and
\item[$(ii)$] the preimages of each crossing connected with a chord.
\end{itemize}

To incorporate the information of over and under strands of the crossing, each chord is oriented from the lower branch to the upper one.  Each chord in $G_K$ is also decorated with the crossing sign. Figure \ref{fig:4_1-gd} shows the Gauss diagram of $4_1$ obtained from the knot diagram shown in Figure \ref{fig:figure-8-plane-diagram}.

\begin{figure}[h]
\centering
%% Creator: Inkscape 1.1-dev (48d18c81, 2020-07-08), www.inkscape.org
%% Accompanies image file 'Figure8-longknot-gd-new.pdf' (pdf, eps, ps)
\begingroup%
  \makeatletter%
  \providecommand\color[2][]{%
    \errmessage{(Inkscape) Color is used for the text in Inkscape, but the package 'color.sty' is not loaded}%
    \renewcommand\color[2][]{}%
  }%
  \providecommand\transparent[1]{%
    \errmessage{(Inkscape) Transparency is used (non-zero) for the text in Inkscape, but the package 'transparent.sty' is not loaded}%
    \renewcommand\transparent[1]{}%
  }%
  \providecommand\rotatebox[2]{#2}%
  \newcommand*\fsize{\dimexpr\f@size pt\relax}%
  \newcommand*\lineheight[1]{\fontsize{\fsize}{#1\fsize}\selectfont}%
  \ifx\svgwidth\undefined%
    \setlength{\unitlength}{288.33682677bp}%
    \ifx\svgscale\undefined%
      \relax%
    \else%
      \setlength{\unitlength}{\unitlength * \real{\svgscale}}%
    \fi%
  \else%
    \setlength{\unitlength}{\svgwidth}%
  \fi%
  \global\let\svgwidth\undefined%
  \global\let\svgscale\undefined%
  \makeatother%
  \begin{picture}(1,0.2032724)%
    \lineheight{1}%
    \setlength\tabcolsep{0pt}%
    \put(-0.06263306,-0.21421016){\color[rgb]{0,0,0}\makebox(0,0)[lt]{\begin{minipage}{0.04860934\unitlength}\raggedright \end{minipage}}}%
    \put(0.19588163,0.04006598){\color[rgb]{0,0,0}\makebox(0,0)[lt]{\begin{minipage}{0.08843824\unitlength}\end{minipage}}}%
    \put(0,0){\includegraphics[width=\unitlength,page=1]{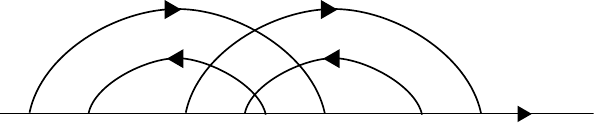}}%
    \put(0.27357837,0.14442178){\color[rgb]{0,0,0}\makebox(0,0)[lt]{\lineheight{1.25}\smash{\begin{tabular}[t]{l}$+$\end{tabular}}}}%
    \put(0.27357837,0.0663891){\color[rgb]{0,0,0}\makebox(0,0)[lt]{\lineheight{1.25}\smash{\begin{tabular}[t]{l}$+$\end{tabular}}}}%
    \put(0.533691,0.14442178){\color[rgb]{0,0,0}\makebox(0,0)[lt]{\lineheight{1.25}\smash{\begin{tabular}[t]{l}$-$\end{tabular}}}}%
    \put(0.533691,0.0663891){\color[rgb]{0,0,0}\makebox(0,0)[lt]{\lineheight{1.25}\smash{\begin{tabular}[t]{l}$-$\end{tabular}}}}%
  \end{picture}%
\endgroup%

\caption{A Gauss diagram representing the long figure 8 long knot.}
\label{fig:4_1-gd}
\end{figure}
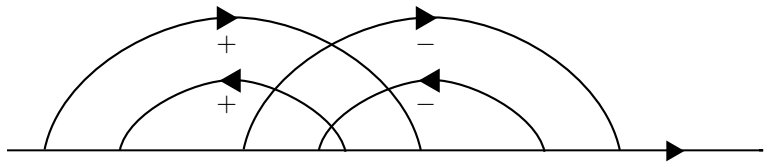

A (long) {\em arrow diagram} $A$ is a general unsigned graph which is a copy of oriented $\R$ (base strand) with an arbitrary collection of {\em oriented chords} also called {\em arrows}, having distinct endpoints marked with $\circ$, common endpoints marked with $\bullet$, or unmarked endpoints: $\perp$. The unmarked endpoints are allowed to be either distinct or colliding. These properties are analogous to the Gauss diagrams  $G_K$,  except no signs of crossings attached to the chords. 
The arrow diagrams and Gauss diagrams can be paired as follows; for  a  Gauss diagram $G=G_K$ of a long knot $K$, an {\em{embedding}} of an  arrow diagram $A$ in $G$ is a graph embedding of $A$ into $G$ mapping the base strand $\mathbb{R}$ of $A$ to the base strand of $G$, and chords to chords, preserving orientations and the above conventions about the markings of the endpoints: $\circ$, $\bullet$, $\perp$.  Further,  the {\em{sign}} of an embedding $\phi : A \longrightarrow G$  is given by\footnote{chords of $A$ will be denoted by the greek letters: $\alpha$, $\beta$,\ldots , and chords of $G$ by lowercase letters: $g$, $h$,\ldots}
%%%%%%%%%%%%%%%%%%%%%%%%%%%%%%%%%%%%%%%%%%
\begin{equation}\label{eq:sign_representation}
\operatorname{sign}(\phi) = \prod_{\alpha \in A} \operatorname{sign}(\phi (\alpha)),
\end{equation}
\no where the $\operatorname{sign}(\phi(\alpha))$ is a sign of the arrow $g=\phi(\alpha)$ in $G$. Further, 
\begin{equation}\label{eq:<A,G>}
\langle A, G\rangle = \sum_{\phi: A \rightarrow G} \operatorname{sign}(\phi),
\end{equation}
\no is a sum taken over all embeddings $\phi: A \longrightarrow G$ of $A$ in $G$. 
A formal sum of arrow diagrams $P=\sum_i c_i A_i$ with integer coefficients is known\footnote{in the standard double crossing case} as an {\em arrow polynomial} \cite{Polyak-Viro:1994}, and $\langle P,G\rangle$ is defined from \eqref{eq:sign_representation} and \eqref{eq:<A,G>} by the linear extension. Also note that using the endpoints convention, we may express the arrow diagrams with unmarked endpoints in terms of the marked endpoint diagrams, for instance 
\begin{equation}\label{eq:ubmarked-to-marked}
\vvcenteredinclude{.3}{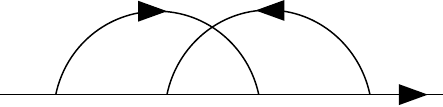}=
 \vvcenteredinclude{.3}{gd_c2_chord.pdf}+\vvcenteredinclude{.3}{gd_c2_chord-left.pdf}
 +\vvcenteredinclude{.3}{gd_c2_chord-mid.pdf}+\vvcenteredinclude{.3}{gd_c2_chord-right.pdf}.
\end{equation}

In a slightly different setting of {\em signed arrow diagrams}, the theorem of Goussarov \cite{Goussarov-Polyak-Viro:2000} shows that any finite type invariant $v$ of knots can be expressed  as $\langle P_v,\,\cdot\,\rangle$ for a suitable choice of a signed arrow polynomial $P_v$. Arrow polynomials of some low degree invariants have been computed in \cite{Ostlund:2004, Polyak-Viro:1994, Polyak-Viro:2001, Willerton:2002}. Apart from low degree examples, the arrow diagram formulae are known for: the coefficients of the Conway \cite{Chmutov-Khoury-Rossi:2009} and the HOMFLY-PT polynomials, \cite{Chmutov-Polyak:2010}, and the Milnor linking numbers \cite{Komendarczyk-Michaelides:2020,Kravchenko-Polyak:2011}, see \cite{Chmutov-Duzhin-Mostovoy:2012} for further information. All formulas obtained in these works assume that the Gauss diagrams of knots are regular double crossing diagrams.

%%%%%%%%%%%%%%%%%%%%%%%%%%%%%%%%%%%%%%%%%%
\begin{center}
	\begin{table}[h!]	
		\caption{Transforming a regular double--crossing projection into a multicrossing projection. All unmarked chords have a positive sign. The chords marked in red correspond to chords that contribute to the $c_2$ count.  Note that in figure $(f)$, the triple crossing between the $1^{st}, 3^{rd}$, and $4^{th}$ strand creates two arrow subdiagrams which cancel each other out, as is depicted in row 2 of Table \ref{tb:triplecrossing}. 
	 }
		\begin{tabular}{|ccc|}
			\hline
			& Long Knot Diagram & Long Gauss diagram\\ \hline
			(a) & \scalebox{.7}{%% Creator: Inkscape 1.1-dev (48d18c81, 2020-07-08), www.inkscape.org
%% Accompanies image file 'long-trefoil-nosigns.pdf' (pdf, eps, ps)
\begingroup%
  \makeatletter%
  \providecommand\color[2][]{%
    \errmessage{(Inkscape) Color is used for the text in Inkscape, but the package 'color.sty' is not loaded}%
    \renewcommand\color[2][]{}%
  }%
  \providecommand\transparent[1]{%
    \errmessage{(Inkscape) Transparency is used (non-zero) for the text in Inkscape, but the package 'transparent.sty' is not loaded}%
    \renewcommand\transparent[1]{}%
  }%
  \providecommand\rotatebox[2]{#2}%
  \newcommand*\fsize{\dimexpr\f@size pt\relax}%
  \newcommand*\lineheight[1]{\fontsize{\fsize}{#1\fsize}\selectfont}%
  \ifx\svgwidth\undefined%
    \setlength{\unitlength}{243.8397361bp}%
    \ifx\svgscale\undefined%
      \relax%
    \else%
      \setlength{\unitlength}{\unitlength * \real{\svgscale}}%
    \fi%
  \else%
    \setlength{\unitlength}{\svgwidth}%
  \fi%
  \global\let\svgwidth\undefined%
  \global\let\svgscale\undefined%
  \makeatother%
  \begin{picture}(1,0.45299999)%
    \lineheight{1}%
    \setlength\tabcolsep{0pt}%
    \put(0,0){\includegraphics[width=\unitlength,page=1]{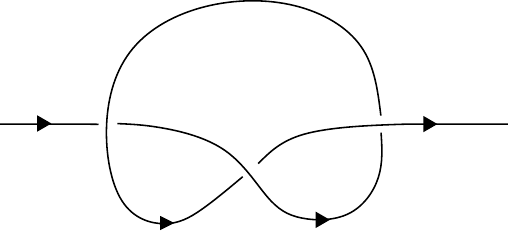}}%
    \put(0.1207542,-0.04198462){\color[rgb]{0,0,0}\makebox(0,0)[lt]{\begin{minipage}{0.05747982\unitlength}\raggedright \end{minipage}}}%
    \put(0.42644394,0.25869307){\color[rgb]{0,0,0}\makebox(0,0)[lt]{\begin{minipage}{0.10457688\unitlength}\end{minipage}}}%
   % \put(0,0){\includegraphics[width=\unitlength,page=2]{long-trefoil-nosigns.pdf}}%
  \end{picture}%
\endgroup%
} & \scalebox{.7}{%% Creator: Inkscape 1.1-dev (48d18c81, 2020-07-08), www.inkscape.org
%% Accompanies image file 'long-trefoil-gd1-nolabels.pdf' (pdf, eps, ps)
\begingroup%
  \makeatletter%
  \providecommand\color[2][]{%
    \errmessage{(Inkscape) Color is used for the text in Inkscape, but the package 'color.sty' is not loaded}%
    \renewcommand\color[2][]{}%
  }%
  \providecommand\transparent[1]{%
    \errmessage{(Inkscape) Transparency is used (non-zero) for the text in Inkscape, but the package 'transparent.sty' is not loaded}%
    \renewcommand\transparent[1]{}%
  }%
  \providecommand\rotatebox[2]{#2}%
  \newcommand*\fsize{\dimexpr\f@size pt\relax}%
  \newcommand*\lineheight[1]{\fontsize{\fsize}{#1\fsize}\selectfont}%
  \ifx\svgwidth\undefined%
    \setlength{\unitlength}{241.83682583bp}%
    \ifx\svgscale\undefined%
      \relax%
    \else%
      \setlength{\unitlength}{\unitlength * \real{\svgscale}}%
    \fi%
  \else%
    \setlength{\unitlength}{\svgwidth}%
  \fi%
  \global\let\svgwidth\undefined%
  \global\let\svgscale\undefined%
  \makeatother%
  \begin{picture}(1,0.20032187)%
    \lineheight{1}%
    \setlength\tabcolsep{0pt}%
    \put(-0.26695444,-0.25539813){\color[rgb]{0,0,0}\makebox(0,0)[lt]{\begin{minipage}{0.05795588\unitlength}\raggedright \end{minipage}}}%
    \put(0.04126704,0.0477698){\color[rgb]{0,0,0}\makebox(0,0)[lt]{\begin{minipage}{0.105443\unitlength}\end{minipage}}}%
    \put(0,0){\includegraphics[width=\unitlength,page=1]{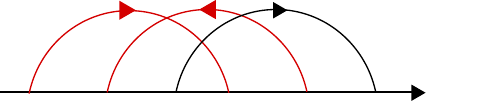}}%
  \end{picture}%
\endgroup%
}\\ 
			(b) & \scalebox{.7}{%% Creator: Inkscape 1.1-dev (48d18c81, 2020-07-08), www.inkscape.org
%% Accompanies image file 'long-trefoil-nosigns2.pdf' (pdf, eps, ps)
\begingroup%
  \makeatletter%
  \providecommand\color[2][]{%
    \errmessage{(Inkscape) Color is used for the text in Inkscape, but the package 'color.sty' is not loaded}%
    \renewcommand\color[2][]{}%
  }%
  \providecommand\transparent[1]{%
    \errmessage{(Inkscape) Transparency is used (non-zero) for the text in Inkscape, but the package 'transparent.sty' is not loaded}%
    \renewcommand\transparent[1]{}%
  }%
  \providecommand\rotatebox[2]{#2}%
  \newcommand*\fsize{\dimexpr\f@size pt\relax}%
  \newcommand*\lineheight[1]{\fontsize{\fsize}{#1\fsize}\selectfont}%
  \ifx\svgwidth\undefined%
    \setlength{\unitlength}{243.8397361bp}%
    \ifx\svgscale\undefined%
      \relax%
    \else%
      \setlength{\unitlength}{\unitlength * \real{\svgscale}}%
    \fi%
  \else%
    \setlength{\unitlength}{\svgwidth}%
  \fi%
  \global\let\svgwidth\undefined%
  \global\let\svgscale\undefined%
  \makeatother%
  \begin{picture}(1,0.46055012)%
    \lineheight{1}%
    \setlength\tabcolsep{0pt}%
    \put(0,0){\includegraphics[width=\unitlength,page=1]{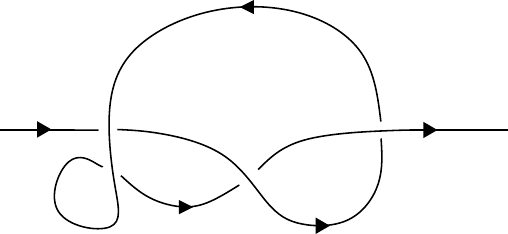}}%
    \put(0.1207542,-0.04609964){\color[rgb]{0,0,0}\makebox(0,0)[lt]{\begin{minipage}{0.05747982\unitlength}\raggedright \end{minipage}}}%
    \put(0.42644394,0.25457805){\color[rgb]{0,0,0}\makebox(0,0)[lt]{\begin{minipage}{0.10457688\unitlength}\end{minipage}}}%
    %\put(0,0){\includegraphics[width=\unitlength,page=2]{long-trefoil-nosigns2.pdf}}%
  \end{picture}%
\endgroup%
} & \scalebox{.7}{%% Creator: Inkscape 1.1-dev (48d18c81, 2020-07-08), www.inkscape.org
%% Accompanies image file 'long-trefoil-gd2-nolabels.pdf' (pdf, eps, ps)
\begingroup%
  \makeatletter%
  \providecommand\color[2][]{%
    \errmessage{(Inkscape) Color is used for the text in Inkscape, but the package 'color.sty' is not loaded}%
    \renewcommand\color[2][]{}%
  }%
  \providecommand\transparent[1]{%
    \errmessage{(Inkscape) Transparency is used (non-zero) for the text in Inkscape, but the package 'transparent.sty' is not loaded}%
    \renewcommand\transparent[1]{}%
  }%
  \providecommand\rotatebox[2]{#2}%
  \newcommand*\fsize{\dimexpr\f@size pt\relax}%
  \newcommand*\lineheight[1]{\fontsize{\fsize}{#1\fsize}\selectfont}%
  \ifx\svgwidth\undefined%
    \setlength{\unitlength}{241.83682583bp}%
    \ifx\svgscale\undefined%
      \relax%
    \else%
      \setlength{\unitlength}{\unitlength * \real{\svgscale}}%
    \fi%
  \else%
    \setlength{\unitlength}{\svgwidth}%
  \fi%
  \global\let\svgwidth\undefined%
  \global\let\svgscale\undefined%
  \makeatother%
  \begin{picture}(1,0.20032187)%
    \lineheight{1}%
    \setlength\tabcolsep{0pt}%
    \put(-0.26695444,-0.25539813){\color[rgb]{0,0,0}\makebox(0,0)[lt]{\begin{minipage}{0.05795588\unitlength}\raggedright \end{minipage}}}%
    \put(0.04126704,0.0477698){\color[rgb]{0,0,0}\makebox(0,0)[lt]{\begin{minipage}{0.105443\unitlength}\end{minipage}}}%
    \put(0,0){\includegraphics[width=\unitlength,page=1]{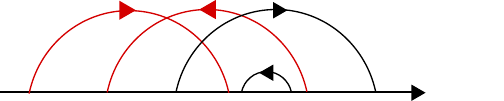}}%
  \end{picture}%
\endgroup%
}\\ 
			(c) &  \scalebox{.7}{%% Creator: Inkscape 1.1-dev (48d18c81, 2020-07-08), www.inkscape.org
%% Accompanies image file 'long-trefoil-nosigns3.pdf' (pdf, eps, ps)
\begingroup%
  \makeatletter%
  \providecommand\color[2][]{%
    \errmessage{(Inkscape) Color is used for the text in Inkscape, but the package 'color.sty' is not loaded}%
    \renewcommand\color[2][]{}%
  }%
  \providecommand\transparent[1]{%
    \errmessage{(Inkscape) Transparency is used (non-zero) for the text in Inkscape, but the package 'transparent.sty' is not loaded}%
    \renewcommand\transparent[1]{}%
  }%
  \providecommand\rotatebox[2]{#2}%
  \newcommand*\fsize{\dimexpr\f@size pt\relax}%
  \newcommand*\lineheight[1]{\fontsize{\fsize}{#1\fsize}\selectfont}%
  \ifx\svgwidth\undefined%
    \setlength{\unitlength}{243.8397361bp}%
    \ifx\svgscale\undefined%
      \relax%
    \else%
      \setlength{\unitlength}{\unitlength * \real{\svgscale}}%
    \fi%
  \else%
    \setlength{\unitlength}{\svgwidth}%
  \fi%
  \global\let\svgwidth\undefined%
  \global\let\svgscale\undefined%
  \makeatother%
  \begin{picture}(1,0.46055012)%
    \lineheight{1}%
    \setlength\tabcolsep{0pt}%
    \put(0,0){\includegraphics[width=\unitlength,page=1]{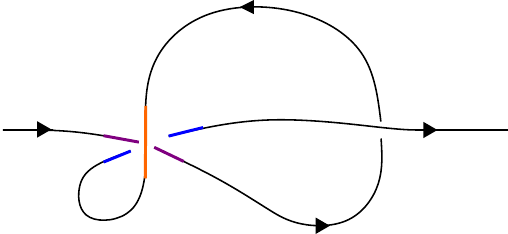}}%
    \put(0.1207542,-0.04609964){\color[rgb]{0,0,0}\makebox(0,0)[lt]{\begin{minipage}{0.05747982\unitlength}\raggedright \end{minipage}}}%
    \put(0.42644394,0.25457805){\color[rgb]{0,0,0}\makebox(0,0)[lt]{\begin{minipage}{0.10457688\unitlength}\end{minipage}}}%
    %\put(0,0){\includegraphics[width=\unitlength,page=2]{long-trefoil-nosigns3.pdf}}%
  \end{picture}%
\endgroup%
} & \scalebox{.7}{%% Creator: Inkscape 1.1-dev (48d18c81, 2020-07-08), www.inkscape.org
%% Accompanies image file 'long-trefoil-gd3-nolabels.pdf' (pdf, eps, ps)
\begingroup%
  \makeatletter%
  \providecommand\color[2][]{%
    \errmessage{(Inkscape) Color is used for the text in Inkscape, but the package 'color.sty' is not loaded}%
    \renewcommand\color[2][]{}%
  }%
  \providecommand\transparent[1]{%
    \errmessage{(Inkscape) Transparency is used (non-zero) for the text in Inkscape, but the package 'transparent.sty' is not loaded}%
    \renewcommand\transparent[1]{}%
  }%
  \providecommand\rotatebox[2]{#2}%
  \newcommand*\fsize{\dimexpr\f@size pt\relax}%
  \newcommand*\lineheight[1]{\fontsize{\fsize}{#1\fsize}\selectfont}%
  \ifx\svgwidth\undefined%
    \setlength{\unitlength}{241.83682583bp}%
    \ifx\svgscale\undefined%
      \relax%
    \else%
      \setlength{\unitlength}{\unitlength * \real{\svgscale}}%
    \fi%
  \else%
    \setlength{\unitlength}{\svgwidth}%
  \fi%
  \global\let\svgwidth\undefined%
  \global\let\svgscale\undefined%
  \makeatother%
  \begin{picture}(1,0.23742433)%
    \lineheight{1}%
    \setlength\tabcolsep{0pt}%
    \put(-0.26695444,-0.21947037){\color[rgb]{0,0,0}\makebox(0,0)[lt]{\begin{minipage}{0.05795588\unitlength}\raggedright \end{minipage}}}%
    \put(0.04126704,0.08369755){\color[rgb]{0,0,0}\makebox(0,0)[lt]{\begin{minipage}{0.105443\unitlength}\end{minipage}}}%
    \put(0,0){\includegraphics[width=\unitlength,page=1]{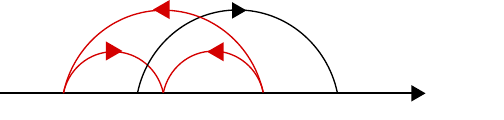}}%
    \put(0.11585672,0.00801767){\color[rgb]{0.50196078,0,0.50196078}\makebox(0,0)[lt]{\lineheight{1.25}\smash{\begin{tabular}[t]{l}$2$\end{tabular}}}}%
    \put(0.31433785,0.00801767){\color[rgb]{1,0.4,0}\makebox(0,0)[lt]{\lineheight{1.25}\smash{\begin{tabular}[t]{l}$1$\end{tabular}}}}%
    \put(0.51281826,0.00801767){\color[rgb]{0,0,1}\makebox(0,0)[lt]{\lineheight{1.25}\smash{\begin{tabular}[t]{l}$3$\end{tabular}}}}%
  \end{picture}%
\endgroup%
}\\ 
			(d) &    \scalebox{.7}{%% Creator: Inkscape 1.1-dev (48d18c81, 2020-07-08), www.inkscape.org
%% Accompanies image file 'long-trefoil-nosigns4.pdf' (pdf, eps, ps)
\begingroup%
  \makeatletter%
  \providecommand\color[2][]{%
    \errmessage{(Inkscape) Color is used for the text in Inkscape, but the package 'color.sty' is not loaded}%
    \renewcommand\color[2][]{}%
  }%
  \providecommand\transparent[1]{%
    \errmessage{(Inkscape) Transparency is used (non-zero) for the text in Inkscape, but the package 'transparent.sty' is not loaded}%
    \renewcommand\transparent[1]{}%
  }%
  \providecommand\rotatebox[2]{#2}%
  \newcommand*\fsize{\dimexpr\f@size pt\relax}%
  \newcommand*\lineheight[1]{\fontsize{\fsize}{#1\fsize}\selectfont}%
  \ifx\svgwidth\undefined%
    \setlength{\unitlength}{243.8397361bp}%
    \ifx\svgscale\undefined%
      \relax%
    \else%
      \setlength{\unitlength}{\unitlength * \real{\svgscale}}%
    \fi%
  \else%
    \setlength{\unitlength}{\svgwidth}%
  \fi%
  \global\let\svgwidth\undefined%
  \global\let\svgscale\undefined%
  \makeatother%
  \begin{picture}(1,0.46055012)%
    \lineheight{1}%
    \setlength\tabcolsep{0pt}%
    \put(0,0){\includegraphics[width=\unitlength,page=1]{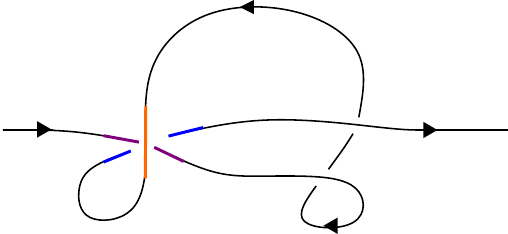}}%
    \put(0.1207542,-0.04609964){\color[rgb]{0,0,0}\makebox(0,0)[lt]{\begin{minipage}{0.05747982\unitlength}\raggedright \end{minipage}}}%
    \put(0.42644394,0.25457805){\color[rgb]{0,0,0}\makebox(0,0)[lt]{\begin{minipage}{0.10457688\unitlength}\end{minipage}}}%
   % \put(0,0){\includegraphics[width=\unitlength,page=2]{long-trefoil-nosigns4.pdf}}%
  \end{picture}%
\endgroup%
} & \scalebox{.7}{%% Creator: Inkscape 1.1-dev (48d18c81, 2020-07-08), www.inkscape.org
%% Accompanies image file 'long-trefoil-gd4-nolabels.pdf' (pdf, eps, ps)
\begingroup%
  \makeatletter%
  \providecommand\color[2][]{%
    \errmessage{(Inkscape) Color is used for the text in Inkscape, but the package 'color.sty' is not loaded}%
    \renewcommand\color[2][]{}%
  }%
  \providecommand\transparent[1]{%
    \errmessage{(Inkscape) Transparency is used (non-zero) for the text in Inkscape, but the package 'transparent.sty' is not loaded}%
    \renewcommand\transparent[1]{}%
  }%
  \providecommand\rotatebox[2]{#2}%
  \newcommand*\fsize{\dimexpr\f@size pt\relax}%
  \newcommand*\lineheight[1]{\fontsize{\fsize}{#1\fsize}\selectfont}%
  \ifx\svgwidth\undefined%
    \setlength{\unitlength}{241.83682583bp}%
    \ifx\svgscale\undefined%
      \relax%
    \else%
      \setlength{\unitlength}{\unitlength * \real{\svgscale}}%
    \fi%
  \else%
    \setlength{\unitlength}{\svgwidth}%
  \fi%
  \global\let\svgwidth\undefined%
  \global\let\svgscale\undefined%
  \makeatother%
  \begin{picture}(1,0.23742433)%
    \lineheight{1}%
    \setlength\tabcolsep{0pt}%
    \put(-0.26695444,-0.21947037){\color[rgb]{0,0,0}\makebox(0,0)[lt]{\begin{minipage}{0.05795588\unitlength}\raggedright \end{minipage}}}%
    \put(0.04126704,0.08369755){\color[rgb]{0,0,0}\makebox(0,0)[lt]{\begin{minipage}{0.105443\unitlength}\end{minipage}}}%
    \put(0,0){\includegraphics[width=\unitlength,page=1]{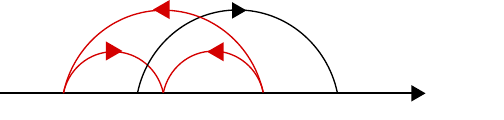}}%
    \put(0.11585672,0.00801767){\color[rgb]{0.50196078,0,0.50196078}\makebox(0,0)[lt]{\lineheight{1.25}\smash{\begin{tabular}[t]{l}$2$\end{tabular}}}}%
    \put(0.31433785,0.00801767){\color[rgb]{1,0.4,0}\makebox(0,0)[lt]{\lineheight{1.25}\smash{\begin{tabular}[t]{l}$1$\end{tabular}}}}%
    \put(0.51281826,0.00801767){\color[rgb]{0,0,1}\makebox(0,0)[lt]{\lineheight{1.25}\smash{\begin{tabular}[t]{l}$3$\end{tabular}}}}%
    \put(0,0){\includegraphics[width=\unitlength,page=2]{long-trefoil-gd4-nolabels.pdf}}%
  \end{picture}%
\endgroup%
}\\ 
			(e) &    \scalebox{.7}{%% Creator: Inkscape 1.1-dev (48d18c81, 2020-07-08), www.inkscape.org
%% Accompanies image file 'long-trefoil-nosigns6.pdf' (pdf, eps, ps)
\begingroup%
  \makeatletter%
  \providecommand\color[2][]{%
    \errmessage{(Inkscape) Color is used for the text in Inkscape, but the package 'color.sty' is not loaded}%
    \renewcommand\color[2][]{}%
  }%
  \providecommand\transparent[1]{%
    \errmessage{(Inkscape) Transparency is used (non-zero) for the text in Inkscape, but the package 'transparent.sty' is not loaded}%
    \renewcommand\transparent[1]{}%
  }%
  \providecommand\rotatebox[2]{#2}%
  \newcommand*\fsize{\dimexpr\f@size pt\relax}%
  \newcommand*\lineheight[1]{\fontsize{\fsize}{#1\fsize}\selectfont}%
  \ifx\svgwidth\undefined%
    \setlength{\unitlength}{243.8397361bp}%
    \ifx\svgscale\undefined%
      \relax%
    \else%
      \setlength{\unitlength}{\unitlength * \real{\svgscale}}%
    \fi%
  \else%
    \setlength{\unitlength}{\svgwidth}%
  \fi%
  \global\let\svgwidth\undefined%
  \global\let\svgscale\undefined%
  \makeatother%
  \begin{picture}(1,0.46055012)%
    \lineheight{1}%
    \setlength\tabcolsep{0pt}%
    \put(0,0){\includegraphics[width=\unitlength,page=1]{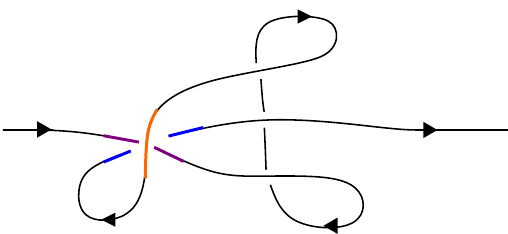}}%
    \put(0.1207542,-0.04609964){\color[rgb]{0,0,0}\makebox(0,0)[lt]{\begin{minipage}{0.05747982\unitlength}\raggedright \end{minipage}}}%
    \put(0.42644394,0.25457805){\color[rgb]{0,0,0}\makebox(0,0)[lt]{\begin{minipage}{0.10457688\unitlength}\end{minipage}}}%
    %\put(0,0){\includegraphics[width=\unitlength,page=2]{long-trefoil-nosigns6.pdf}}%
  \end{picture}%
\endgroup%
} & \scalebox{.7}{%% Creator: Inkscape 1.1-dev (48d18c81, 2020-07-08), www.inkscape.org
%% Accompanies image file 'long-trefoil-gd5-nolabels.pdf' (pdf, eps, ps)
\begingroup%
  \makeatletter%
  \providecommand\color[2][]{%
    \errmessage{(Inkscape) Color is used for the text in Inkscape, but the package 'color.sty' is not loaded}%
    \renewcommand\color[2][]{}%
  }%
  \providecommand\transparent[1]{%
    \errmessage{(Inkscape) Transparency is used (non-zero) for the text in Inkscape, but the package 'transparent.sty' is not loaded}%
    \renewcommand\transparent[1]{}%
  }%
  \providecommand\rotatebox[2]{#2}%
  \newcommand*\fsize{\dimexpr\f@size pt\relax}%
  \newcommand*\lineheight[1]{\fontsize{\fsize}{#1\fsize}\selectfont}%
  \ifx\svgwidth\undefined%
    \setlength{\unitlength}{241.83682583bp}%
    \ifx\svgscale\undefined%
      \relax%
    \else%
      \setlength{\unitlength}{\unitlength * \real{\svgscale}}%
    \fi%
  \else%
    \setlength{\unitlength}{\svgwidth}%
  \fi%
  \global\let\svgwidth\undefined%
  \global\let\svgscale\undefined%
  \makeatother%
  \begin{picture}(1,0.23742433)%
    \lineheight{1}%
    \setlength\tabcolsep{0pt}%
    \put(-0.26695444,-0.21947037){\color[rgb]{0,0,0}\makebox(0,0)[lt]{\begin{minipage}{0.05795588\unitlength}\raggedright \end{minipage}}}%
    \put(0.04126704,0.08369755){\color[rgb]{0,0,0}\makebox(0,0)[lt]{\begin{minipage}{0.105443\unitlength}\end{minipage}}}%
    \put(0,0){\includegraphics[width=\unitlength,page=1]{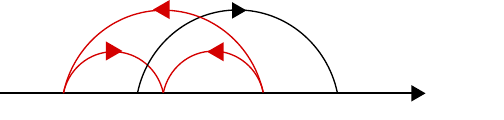}}%
    \put(0.11585672,0.00801767){\color[rgb]{0.50196078,0,0.50196078}\makebox(0,0)[lt]{\lineheight{1.25}\smash{\begin{tabular}[t]{l}$2$\end{tabular}}}}%
    \put(0.31433785,0.00801767){\color[rgb]{1,0.4,0}\makebox(0,0)[lt]{\lineheight{1.25}\smash{\begin{tabular}[t]{l}$1$\end{tabular}}}}%
    \put(0.51281826,0.00801767){\color[rgb]{0,0,1}\makebox(0,0)[lt]{\lineheight{1.25}\smash{\begin{tabular}[t]{l}$3$\end{tabular}}}}%
    \put(0,0){\includegraphics[width=\unitlength,page=2]{long-trefoil-gd5-nolabels.pdf}}%
    \put(0.40737596,0.05763723){\color[rgb]{0,0,0}\makebox(0,0)[lt]{\lineheight{1.25}\smash{\begin{tabular}[t]{l}$-$\end{tabular}}}}%
  \end{picture}%
\endgroup%
}\\ 
			(f) &      \scalebox{.7}{%% Creator: Inkscape 1.1-dev (48d18c81, 2020-07-08), www.inkscape.org
%% Accompanies image file 'long-trefoil-nosigns5.pdf' (pdf, eps, ps)
\begingroup%
  \makeatletter%
  \providecommand\color[2][]{%
    \errmessage{(Inkscape) Color is used for the text in Inkscape, but the package 'color.sty' is not loaded}%
    \renewcommand\color[2][]{}%
  }%
  \providecommand\transparent[1]{%
    \errmessage{(Inkscape) Transparency is used (non-zero) for the text in Inkscape, but the package 'transparent.sty' is not loaded}%
    \renewcommand\transparent[1]{}%
  }%
  \providecommand\rotatebox[2]{#2}%
  \newcommand*\fsize{\dimexpr\f@size pt\relax}%
  \newcommand*\lineheight[1]{\fontsize{\fsize}{#1\fsize}\selectfont}%
  \ifx\svgwidth\undefined%
    \setlength{\unitlength}{238.98047244bp}%
    \ifx\svgscale\undefined%
      \relax%
    \else%
      \setlength{\unitlength}{\unitlength * \real{\svgscale}}%
    \fi%
  \else%
    \setlength{\unitlength}{\svgwidth}%
  \fi%
  \global\let\svgwidth\undefined%
  \global\let\svgscale\undefined%
  \makeatother%
  \begin{picture}(1,0.46992325)%
    \lineheight{1}%
    \setlength\tabcolsep{0pt}%
    \put(0,0){\includegraphics[width=\unitlength,page=1]{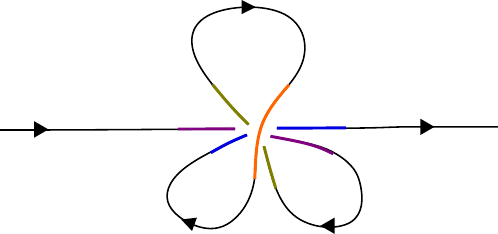}}%
    \put(0.11720302,-0.5491646){\color[rgb]{0,0,0}\makebox(0,0)[lt]{\begin{minipage}{0.05864858\unitlength}\raggedright \end{minipage}}}%
    \put(0.42910845,-0.24237314){\color[rgb]{0,0,0}\makebox(0,0)[lt]{\begin{minipage}{0.10670328\unitlength}\end{minipage}}}%
  \end{picture}%
\endgroup%
} & \scalebox{.7}{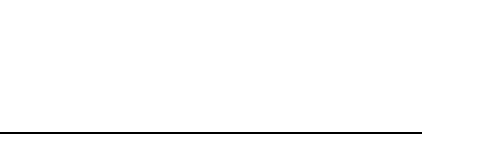}\\ \hline
		\end{tabular}
		\label{tab:petalexample}
	\end{table}
\end{center}

\begin{example}\label{eq:trefoil-computation}
	Table \ref{tab:petalexample} shows various projections of the long trefoil knot and the corresponding multiple--crossing Gauss diagrams: $G_a$, $G_b$, $G_c$, $G_d$, $G_e$, and $G_f$, which represent moves transforming a regular double--crossing projection into a multi-crossing projection.  Using the above definitions, we may calculate various pairings of arrow diagrams and Gauss diagrams. Since $G_a$ is a regular double crossing diagram \eqref{eq:gd_c2_dbl-crossings} yields
	\[
	c_2(4_1)=\langle \vvcenteredinclude{.3}{gd_c2_chord.pdf},G_a\rangle=\langle \vvcenteredinclude{.3}{gd_c2_chord_unmarked.pdf},G_a\rangle=1.
	\] 
	However, the formula in  \eqref{eq:gd_c2_dbl-crossings} does not extend to the multi-crossing Gauss diagrams, for instance
	\[
	\begin{split}
	 \langle \vvcenteredinclude{.3}{gd_c2_chord.pdf},G_c\rangle & =0,\quad \langle \vvcenteredinclude{.3}{gd_c2_chord_unmarked.pdf},G_c\rangle=2,\\
	 \langle \vvcenteredinclude{.3}{gd_c2_chord.pdf},G_f\rangle & =0,\quad \langle \vvcenteredinclude{.3}{gd_c2_chord_unmarked.pdf},G_f\rangle=2.
	\end{split}
	\]
	Formula \eqref{eq:gd_c2_mult-crossings} yields (for the multi-crossing Gauss diagram $G_e$)
	\[
		\begin{split}
			c_2(K) & =\langle \vvcenteredinclude{.3}{gd_c2_chord.pdf},G_f\rangle+ \tfrac{1}{2}\bigl(\langle \vvcenteredinclude{.3}{gd_c2_chord-left.pdf},G_f\rangle +\langle \vvcenteredinclude{.3}{gd_c2_chord-mid.pdf},G_f\rangle\\
			& +\langle \vvcenteredinclude{.3}{gd_c2_chord-right.pdf},G_f\rangle\bigr) = 0+\tfrac{1}{2}\bigl(1+(1-1)+1\bigr)=1.
		\end{split}
	\]
\end{example}

%%%%%%%%%%%%%%%%%%%%%%%%%%%%%%%%%%%%%%%%%%%%%%%%%%%%%%%%
\subsection{Connections to the existing work} In \cite{Polyak-Viro:2001}, Polyak and Viro  
equate the combinatorial formula \eqref{eq:gd_c2_dbl-crossings}, first proposed in \cite{Polyak-Viro:1994}, with a degree $\deg(h)$ of the map $h=h_{13}\times h_{42}\times\operatorname{id}$. In order to properly define this degree they extend the compactified domain\footnote{an $S^2$ factor is needed to make the boundary of $\Cnf[I;4]\times S^2$ compatible with the boundary of $\Cnf[I,\K_{S^1};3,1]$} $\Cnf[I;4]\times S^2$ of $h$ in \eqref{eq:gd_c2_dbl-crossings} by gluing $\Cnf[I,\K_{S^1};3,1]$ along the principal faces, and identifying hidden faces via Kontsevitch's involutions to obtain a $6$ dimensional cell complex $\mathcal{C}$. 
Since $\mathcal{C}$ has a well defined fundamental class, $\deg(h)$ can be defined homologically as the pullback of the fundametal class of the codomain $(S^2)^3$ 
\begin{equation}\label{eq:c_2=deg(h)}
c_2(K)=\deg(h)=\langle h_\ast([\mathcal{C}]),[(S^2)^3]\rangle.
\end{equation}
Whether we can perform this computation by integrating a differential form is a different issue because $\mathcal{C}$ is not a closed manifold. In \cite{Bott-Taubes:1994}, Bott and Taubes use the symmetric Gauss form $\omega$ on each $S^2$ factor to obtain the formula \eqref{eq:casson-integral-original}, we review the details of this calculation for long knots in Section \ref{sec:bott-taubes} and formula \eqref{eq:casson-integral-eta-pi_ast}.
 If $K$ has a regular double crossing projection (onto a plane orthogonal to $N$), $(N,N,N)$ is a regular value of $h$ and the local degree formula yields the arrow diagram count\footnote{for sufficiently small $\varepsilon$}: $c_2(K)=\langle \vvcenteredinclude{.25}{gd_c2_chord.pdf},G_K\rangle$, which we also obtain from
 our integral \eqref{eq:casson-integral-eta} in Section \ref{sec:arrow-counting-dbl-crossing} by showing that the contribution of the second term in \eqref{eq:casson-integral-eta} vanishes. 
 
In contrast to this ``cut and paste'' approach, we construct the localized\footnote{we use this terminology following \cite[Proposition 6.25]{Bott-Tu:1982}} integral \eqref{eq:casson-integral-eta} directly using the Bott--Taubes integration technique \cite{Bott-Taubes:1994, Volic:2007, Thurston:1999}, with the form $\eta_N$, \eqref{eq:eta_N}. Since $\eta_N$ is not antisymmetric, we cannot use the Kontsevitch's involutions for proving vanishing over the hidden and anomalous faces. As a consequence, we obtain the invariant defined over a subspace $\K_{N}$ \eqref{eq:K_N}, and showing the invariance over the whole knot space $\K$ reduces to the invariance under {\bf RI} move (Theorem \ref{thm:c2-integral} $(iii)$). Further, note that the local degree argument generally fails for knots which do not have regular double crossing projection in the direction of $N$, thus $\varint\limits_{\Cnf(I;4)} h^\ast_{13}\eta_N\wedge h^\ast_{42}\eta_N$ will not give a correct value for $c_2(K)$, 
while the formula \eqref{eq:casson-integral-eta} works for any knot in $\K_{N}$, in particular for the multicrossing knot diagrams it yields \eqref{eq:gd_c2_mult-crossings}.

The techniques outlined above can be regarded as the discretization of the Bott and Taubes  configuration space integral \eqref{eq:casson-integral-original}.
In \cite{Lin-Wang:1996}, Lin and Wang approached this discretization in a different way; rather than replacing the form $\omega$ with $\eta_N$, they adjust the parametrization of the knot $K$ to lie almost entirely in the plane orthogonal to $N$ except $\varepsilon$--small overpasses (assuming $K$ admits a regular projection onto that plane), denoted by $K_\varepsilon$. Then, they examined the limit of the integral \eqref{eq:casson-integral-original} over $K_\varepsilon$ as $\varepsilon\to 0$. Although technically similar, this limit does not yield the local degree expression
in \eqref{eq:gd_c2_dbl-crossings}. In particular the second term $\varint\limits_{\Cnf(I,K;3,1)} \omega_{1\mathbf{4}}\wedge\omega_{2\mathbf{4}}\wedge\omega_{3\mathbf{4}}$ (the ``tripod'' integral)
in \eqref{eq:casson-integral-original} does not vanish in the limit.  The first term $\frac{1}{4}\varint\limits_{\Cnf(I;4)} \omega_{13}\wedge\omega_{24}$ gives a signed crossing count in the limit, reminiscent\footnote{where the all arrow directions are allowed in the arrow diagram} of \eqref{eq:gd_c2_dbl-crossings}, plus a term proportional to the crossing number $\operatorname{Cr}(K)$ of $K$. We expect that computations of Lin and Wang can be also obtained in our setting, by using the averaged from $\eta_{Av}=\tfrac{1}{2}(\eta_N+\eta_S)$ in place of $\omega$. One advantage of $\eta_{Av}$ is that it is antisymmetric, thus we could use the involutions of \cite{Kontsevich:1993} to show vanishing over certain hidden faces, and apply our technique (Section \ref{sec:vanishing-faces-4}) to show vanishing over the anomalous face. On the other hand, nonvanishing of the tripod integral term in the limit prevents us from a straightforwad identification of the invariant with the general arrow diagram formula in \eqref{eq:<A,G>}.  One way to address this is to consider the more general setting of {\em tinkertoy diagrams} of Thurston \cite{Thurston:1999}, this however results in counting features\footnote{such as similar triangles in the knot projection} of the knot projection not tied entirely to the count of signed crossings, see e.g. \cite{Polyak-Viro:2001}. Other approaches to the calculation of $c_2(K)$ in terms certain intersection numbers have been developed in the work of Budney et al. \cite{Budney-Conant-Scannell:2005}.

In summary, the techniques investigated in this work, and presented in the case of $c_2(K)$,  show that some complications of the original Bott--Taubes integration procedure (such as vanishing over the anomalous face) can be simplified when working with localized forms, which may in turn provide a correct setting for the diagram cochain complex for knots  (c.f. \cite{Cattaneo-Cotta-Ramusino-Longoni:2002, Koytcheff-Munson-Volic:2013}) in the classical dimension $n=3$.
It may also be a proper setting to prove the extension of the Goussarov's result in \cite{Goussarov-Polyak-Viro:2000}, to unsigned arrow diagrams in \eqref{eq:<A,G>}. We plan to invesigate these questions in the future work.

%%%%%%%%%%%%%%%%%%%%%%%%%%%%%%%%%%%%%%%%%%%%%%%%%%%%%%%%
\subsection{Organization of the paper} The paper is organized as follows. In Section \ref{sec:bott-taubes} we review the Bott--Taubes technique and formulate the integral \eqref{eq:casson-integral-original} in terms of the pushforward operator. The proof of Theorem \ref{thm:c2-integral} is given in Section \ref{sec:proof-A}, where in the first part we follow the Bott--Taubes technique to prove invariance of the integral \eqref{eq:casson-integral-eta} and in the second part we show \eqref{eq:gd_c2_dbl-crossings} via the local degree formula in the style of \cite{Bott-Tu:1982}. Theorem \ref{thm:c2-multicrossing} is proven, via combinatorial techniques, in Section \ref{sec:arrow-counting-mltpl-crossing}. Finally, in Section \ref{sec:petal-diagrams}, we focus on petal projections and obtain Corollary \ref{cor:uber-crossing} as a consequence.
 
\subsection{Acknowledgments}  This paper builds on the results of the first author's doctoral thesis \cite{Brooks:2020}. The second author acknowledges the partial support by Louisiana Board of Regents
Targeted Enhancement Grant 090ENH-21. Both authors wish to thank the anonymous referee for pointing out critical issues with the first version of the paper and insightful comments in the review.

%%%%%%%%%%%%%%%%%%%%%%%%%%%%%%%%%%%%%%%%%%%%%%%%%%%%%%%%%
\section{Review of the Bott--Taubes integration}\label{sec:bott-taubes}

\subsection{Preliminaries on the Stokes theorem for fiber bundles and pushforwards} Consider a smooth fiber bundle $\pi:E\longrightarrow B$ of manifolds $E$ and $B$ of dimensions $m+n$ and $m$, with fiber $F$ compact manifold with boundary of dimension $n$. Let us review the definition of the {\em pushforward} of differential forms  in the style of \cite{Bott-Tu:1982},  
%%%%%%%%%%%%%%%%%%%%%%%%%%%%%%%%%%%%
\begin{equation}\label{eq:pi_ast}
 \pi_\ast:\Omega^\ast(E)\longrightarrow \Omega^{\ast-n}(B),\quad \text{also denoted by}\quad \varint_{F}=\pi_\ast.
\end{equation}
It suffices to define $\pi_\ast$ in a local trivialization $U\times W$ of $E$ where $U$ is a neighborhood in $B$, and $W=[0,\varepsilon)\times V\subset F$ is intersecting the boundary or $W=(0,\varepsilon)\times V$ is in the interior of $F$. 
Let $(\mathbf{x},\mathbf{t})=(x_1,\ldots,x_m,t_1,\ldots, t_n)$ be coordinates on $U\times W$ (where either $t_1\geq 0$ or $t_1>0$); a differential form $\beta\in \Omega^{k}(E)$ can be locally expressed as a sum of two types of terms, the type (I) forms which do not contain $d\mathbf{t}=dt_1\wedge\ldots\wedge dt_n$ and type (II) which contain it. The map $\pi_\ast$ is defined by 
\begin{equation}\label{eq:pi_ast-def}
 \pi_\ast \beta=\begin{cases}
 	\text{(I)}\ & \pi^\ast \phi \wedge f(\mathbf{x},\mathbf{t})\, dt_{i_1}\wedge dt_{i_2}\wedge\ldots\wedge dt_{i_r}  \longrightarrow 0, \quad i_1<i_2<\ldots<i_r,\ r<n; \\
 	\text{(II)}\ & \pi^\ast \phi \wedge f(\mathbf{x},\mathbf{t})\, d\mathbf{t}  \longrightarrow \phi \wedge \bigl(\varint_W f(\mathbf{x},\mathbf{t})\bigr)\, d\mathbf{t}, %\robyn{\phi \wedge\bigl( \varint_W f(\mathbf{x},\mathbf{t}) d\mathbf{t}\bigr)},
 \end{cases}
\end{equation}
where $\phi$ is a compactly supported form on $U\subset B$, $f$ has a compact support on $W$, and $d\mathbf{t}=dt_{1}\wedge dt_{2}\wedge\ldots\wedge dt_{n}$. 
We have the following version of Stokes' Theorem for fiber bundles\footnote{One recovers the classical Stokes' Theorem when $B$ is a point.},
\begin{equation}\label{eq:stokes-fiber}
 d \pi_\ast\beta = \pi_\ast d\beta + (-1)^{\operatorname{deg}(\phi)}(\partial \pi)_\ast \beta,\quad \operatorname{deg}(\phi)=k-n+1,
\end{equation}
where $(\partial \pi)_\ast \beta$ will be called the {\em boundary pushforward}.
For a type (II) form $\beta$, just as in  \cite[p. 62]{Bott-Tu:1982}, we have
\[
d\pi_\ast \beta = \pi_\ast d\beta.
\]
By \eqref{eq:pi_ast-def}, for a type (I) form $\beta=\pi^\ast\phi \wedge f(\mathbf{x},\mathbf{t})\, dt_{i_1}\wedge dt_{i_2}\wedge\ldots\wedge dt_{i_r}$, $\pi_\ast \beta=0$ and thus $d\pi_\ast \beta=0$. Further,
\[
\begin{split}
 \pi_\ast d\beta & =\pi_\ast\bigl((d\pi^\ast\phi) \wedge f(\mathbf{x},\mathbf{t})\, dt_{i_1}\wedge dt_{i_2}\wedge\ldots\wedge dt_{i_r}\bigr)\\
 & +(-1)^{\operatorname{deg}(\phi)}\sum_j \pi_\ast\bigl(\pi^\ast\phi \wedge \frac{\partial f}{\partial t_j}(\mathbf{x},\mathbf{t}) dt_j\wedge dt_{i_1}\wedge dt_{i_2}\wedge\ldots\wedge dt_{i_r}\bigr),
 \end{split}
\]
where the first term vanishes by \eqref{eq:pi_ast-def}, the second term is nonzero only for $r=n-1$, and can be written as
\[
\pi_\ast d\beta =(-1)^{\operatorname{deg}(\phi)}\sum^n_{j=1} \phi \wedge \bigl(\varint_{W}\frac{\partial f}{\partial t_j}(\mathbf{x},\mathbf{t}) dt_j\wedge dt_{1}\wedge \ldots\wedge \widehat{dt_{j}}\wedge\ldots\wedge dt_{n}\bigr).
\]
From the definition of $W$ we can show (similarly as in \cite[p. 62]{Bott-Tu:1982}) the right hand side vanishes for 
$W=(0,\varepsilon)\times V$, (since $\beta$ is compactly supported) and each term in the sum is zero for $j>1$, which leaves only $j=1$ and $W=[0,\varepsilon)\times V$, ($0\leq t_1<\varepsilon$) a possible nonzero contribution. Indeed, taking into account the orientation\footnote{consistent with the {\em outer normal first} convention} of $\partial W=\{0\}\times V$, by $-dt_2\wedge\ldots\wedge dt_{n}$ we obtain
\[
\begin{split}
\pi_\ast d\beta & =(-1)^{\operatorname{deg}(\phi)} \phi \wedge \bigl(\varint_{[0,\varepsilon)\times V}\frac{\partial f}{\partial t_1}(\mathbf{x},\mathbf{t}) dt_{1}\wedge \ldots\wedge dt_{n}\bigr)\\
& = (-1)^{\operatorname{deg}(\phi)} \phi \wedge \varint_{V} \bigl(\bigl(\lim_{t_1\to \varepsilon} f(\mathbf{x},\mathbf{t})\bigr)- f(\mathbf{x},\mathbf{t})\bigl|_{t_1=0}\bigr) dt_2\wedge\ldots\wedge dt_{n}
\end{split}
\]
\[
= -(-1)^{\operatorname{deg}(\phi)} \phi \wedge \varint_{\{0\}\times V} f(\mathbf{x},\mathbf{t}) dt_2\wedge\ldots\wedge dt_{n}=(-1)^{\operatorname{deg}(\phi)} \phi \wedge \varint_{\partial W} f(\mathbf{x},\mathbf{t}) dt_2\wedge\ldots\wedge dt_{n},
\]
(since $f$ is compactly supported on $W$ the limit of $t_1$ at $\varepsilon$ vanishes) and since $\operatorname{deg}(\phi)=\operatorname{deg}(\beta)-n$. Therefore, we obtain the formula \eqref{eq:stokes-fiber} in the Stokes' theorem, if $(\partial \pi_\ast)$ is defined on type (I) and (II) forms as follows $(\partial \pi_\ast)\beta=$
\begin{equation}\label{eq:partial-pi_ast-def}
	\begin{cases}
		\text{(I)} &   \pi^\ast \phi \wedge f(\mathbf{x},\mathbf{t})\, dt_{i_1}\wedge dt_{i_2}\wedge\ldots\wedge dt_{i_r}  \longrightarrow 0,\ r<n-1 \textrm{ or } i_j=1 \textrm{ for some } j; \\
&   \pi^\ast \phi \wedge f(\mathbf{x},\mathbf{t})\, dt_{2}\wedge dt_{3}\wedge\ldots\wedge dt_{n}  \longrightarrow \phi\wedge \varint_{\partial W} f(\mathbf{x},\mathbf{t}) dt_2\wedge\ldots\wedge dt_{n}; \\
	\text{(II)} 	&  \pi^\ast \phi \wedge f(\mathbf{x},\mathbf{t})\, dt_{1}\wedge dt_{2}\wedge\ldots\wedge dt_{n} \longrightarrow 0,
	\end{cases}
\end{equation}
where $\partial W\neq \varnothing$, just for $W=[0,\varepsilon)\times V$, and $\partial W=\{t_1=0\}=\{0\}\times V$. To express $(\partial \pi_\ast)$ in the coordinate free way, let $j:E\bigl|_{\partial F}\longrightarrow E$ be the inclusion of the sub-bundle of $E\bigl|_{\partial F}$ into $E$, whose fiber is $\partial F$ ($\partial F\longrightarrow E\bigl|_{\partial F}\xrightarrow{\pi_\partial} B$), then 
\begin{equation}\label{eq:partial-pi_ast-c-free}
 (\partial \pi_\ast)=(\pi_\partial)_\ast \circ j_\ast\ . 
\end{equation}
 In the calculations we encounter here, the boundary will be partitioned into pieces which possibly overlap as sets of zero measure, i.e. $\partial W=Q_1\cup\ldots\cup Q_m$, therefore restricting the integral in \eqref{eq:partial-pi_ast-def}   
to $Q_i$ yields $(\partial \pi_\ast)\beta\bigl|_{Q_i}$ and 
\begin{equation}\label{eq:bdry-pushforward-faces}
(\partial \pi_\ast)\beta=\sum^m_{i=1} (\partial \pi_\ast)\beta\bigl|_{Q_i}.
\end{equation}
%%%%%%%%%%%%%%%%%%%%%%%%%%%%%%%%%%%%%%%%%
\begin{figure}[h]
	\centering
	\includegraphics[width=0.9\textwidth]{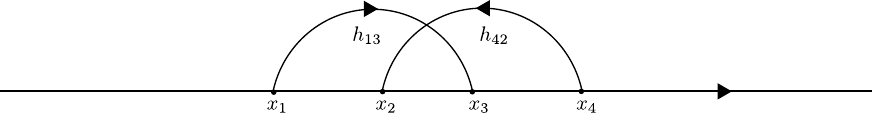}
	\caption{Chord diagram associated to the map $h_4$.}
	\label{fig:gaussdiagramc2}
\end{figure}
%%%%%%%%%%%%%%%%%%%%%%%%%%%%%%%%%%%%%%%%%%
%%%%%%%%%%%

\no We have the following basic result about boundary pushforwards which follows immediately from \eqref{eq:partial-pi_ast-c-free}.
\begin{prop}\label{prop:equal-along}
	Suppose $h,h':E,E'\longrightarrow M$ are smooth maps from the total spaces $E$, $E'$  of the bundles $\pi,\pi':E^{n+m}, (E')^{n+m}\longrightarrow B^m$ where a manifold $M$ is of dimension $m+n$. Assuming that the fibers $F$ and $F'$ of $E$ and $E'$ have the common boundary $\partial F'=\partial F$ and $h\bigl|_{\partial F}=h'\bigl|_{\partial F}$ for $\beta\in \Omega^k(M)$, one obtains
	\begin{equation}\label{eq:partial-pi-ast-equal}
	 (\partial \pi_\ast)h^\ast\beta=\pm(\partial \pi_\ast)h'^\ast\beta.
	\end{equation}
	where the sign is taking into account the orientation. 
\end{prop}

%%%%%%%%%%%%%%%%%%%%%%%%%%%%%%%%%
\subsection{Integrating over the compactified domains}
%%%%%%%%%%%%%%%%%%%%%%%%%%%%%%%%%
 In this section we show, in the style of \cite{Bott-Taubes:1994, Volic:2007}, how to express the integrals \eqref{eq:casson-integral-eta}, \eqref{eq:casson-integral-original}  in terms of the pushforward operator $\pi_\ast$.

Recall that we work with the subspace of long knots  $\K_N$ defined in \eqref{eq:K_N} and consider the trivial fiber bundle $\Cnf[\R;4]\times\K_N$ with projection map 
\begin{equation}\label{eq:pi-projection-C_4}
\pi:\Cnf[\R;4]\times \K_N\longrightarrow \K_N
\end{equation}
\noindent and fiber $\pi^{-1}(K)=\Cnf[\R;4]$, given $K\in\K_N$.
Define the evaluation map 
\begin{equation}\label{eq:ev_4}
	\operatorname{ev}_4:  \Cnf[\R;4]\times \K_N\longrightarrow \Cnf[\R^3;4]
\end{equation}
as an extension of the map
\[
 ((x_1,x_2,x_3,x_4),K)\longrightarrow (K(x_1),K(x_2),K(x_3),K(x_4)),
\]
from the open configuration spaces.
Also, let $h_{ij} :\Cnf[\R^3;4]\longrightarrow S^2$ be the smooth extension of the standard Gauss map to the associated compactifications
\begin{equation}\label{eq:h_ij}
		h_{ij} :\Cnf(\R^3;4)\longrightarrow S^2,\qquad
		(\bx_1,\bx_2,\bx_3,\bx_4)\longrightarrow \frac{\bx_j-\bx_i}{|\bx_j-
			\bx_i|}.
\end{equation}
Such extensions are uniquely defined as shown in \cite{Sinha:2004}. 
For the $2$--form $\eta_N$, given in \eqref{eq:eta_N}, $\eta_{N}\times\eta_{N}=p^\ast_1\eta_{N}\wedge p^\ast_2\eta_{N}$ is a $4$--form over $S^2\times S^2$. Let 
%%%%%%%%%%%%%%%%%%%%%%%%%%%%%%%%%%%%%%%%%
\begin{equation}\label{eq:beta}
\beta=h^\ast_4(\eta_{N}\times \eta_{N}),\quad h_4=(h_{13}\times h_{42})\circ \ev_4,
\end{equation}
be the pullback form over $\Cnf[\R;4]\times\K_N$. Note that the map $h_4$ in \eqref{eq:beta} can be formally encoded by the familiar arrow diagram 
$\vvcenteredinclude{.3}{gd_c2_chord_unmarked.pdf}$ of Section \ref{sec:arrow-diagrams}. The encoding is shown on Figure \ref{fig:gaussdiagramc2}, where the coordinates on $\Cnf[\R;4]$ and the Gauss maps are indicated. 

By restricting the map $\pi$ defined in \eqref{eq:pi-projection-C_4} to $\Cnf[\R;4]\times\K_N$, the pushforward of $\beta$ yields a function on $\K_N$ (i.e. the degree $0$--form on $\K_N$)
%%%%%%%%%%%%%%%%%%%%%%%%%%%%%%%%%%%%%%%%%
\begin{equation}\label{eq:pi-beta}
	\pi_\ast\beta(K) = \varint_{\Cnf[\R;4]} h^\ast_4(\eta_{N}\times \eta_N),\qquad K\in \K_N,
\end{equation}
which is understood in the sense of Appendix \ref{sec:diffeology}: The function $\pi_\ast\beta$ is defined over plots $\mathcal{K}:U\longrightarrow \K_N$, $U\subset \R^k$, where $\pi_\ast\beta=\{(\pi_\ast\beta)_{\mathcal{K}}\}_{\mathcal{K}:U\longrightarrow \K_N}$, and $(\pi_\ast\beta)_\mathcal{K}$ is given by the pushforward  of $(\beta)_\mathcal{K}$,in \eqref{eq:pi_ast-def}, over the trivial projection $\Cnf[\R;4]\times U\longrightarrow U$.
The question of whether this function is a knot invariant is a question of whether it is locally constant;  which  is true if $d \pi_\ast\beta =0$ (i.e. the $0$--form $\pi_\ast\beta$ is closed). Using \eqref{eq:stokes-fiber} and $d\beta=0$, we obtain 
\begin{equation}\label{eq:stokes-fiber2}
d\pi_\ast\beta=\pi_\ast d\beta - (\partial \pi)_\ast \beta=-(\partial \pi)_\ast \beta,
\end{equation}
where $(\partial \pi)_\ast \beta$ is the sum of pushforwards along the codimension one faces of $\Cnf[\R;4]$ as defined in \eqref{eq:bdry-pushforward-faces}.

Unfortunately, the $1$--form on right hand side of \eqref{eq:stokes-fiber2} vanishes for some, but not all, of the boundary faces of $\Cnf[\R;4]$; therefore a ``correction term'' is needed to obtain a knot invariant from $\pi_\ast\beta$ in \eqref{eq:pi-beta}. This correction term is constructed as an integral over a connected component of $\Cnf[\K_N,\R^3;3,1]$, which is also a manifold with corners, given as a pullback bundle over $p$ in the following diagram
\[
\xymatrix{
	\Cnf[\K_N,\R^3;3,1] \ar[rr] \ar^-{\ev_{3;1}}[rr] \ar@{->}_-{}[d]  & & \Cnf[\R^3;4] \ar^-{p}[d] \\
	\Cnf[\R;3]\times \K_N  \ar^-{\ev_3}_-{\ }[rr] & &  \Cnf[\R^3;3].
}
\]
The projection $p$ is the extension of the projection from $\Cnf(\R^3;4)$ to $\Cnf(\R^3;3)$ given by skipping the last coordinate. By definition 
\begin{equation*}\label{eq:Cnf_K}
	\Cnf[\K_N,\R^3;3,1]=\{(x, K;\mathbf{x})\in (\Cnf[\R;3]\times \K_N)\times \Cnf[\R^3;4]\ |\ \ev_3(x, K)= p(\mathbf{x})\},
\end{equation*}
 so that a fiber of $\Cnf[\K_N,\R^3;3,1]$ over $K\in \K_N$ is the subset of $\Cnf[\R^3;4]$ restricted to configurations $(\mathbf{x}_1,\mathbf{x}_2, \mathbf{x}_3, \mathbf{x}_4)$, where $\mathbf{x}_1,\mathbf{x}_2, \mathbf{x}_3$ ``sit'' on the knot $K$.  We refer to $\mathbf{x}_4$ as the point ``off the knot" and $(x_1,x_2,x_3)\in\Cnf[\R;3]$.  The top row evaluation map in the above diagram is given by
 \[
	\operatorname{ev}_{3;1}: \Cnf[\K_N,\R^3;3,1]\subset\Cnf[\R;3]\times \K_N\times \Cnf[\R^3;4] \longrightarrow \Cnf[\R^3;4],
\]
which is the extension of 
\[
	((x_1,x_2,x_3),K,(\mathbf{x}_1,\mathbf{x}_2, \mathbf{x}_3, \mathbf{x}_4))\longrightarrow (K(x_1),K(x_2),K(x_3),\mathbf{x}_4),
\]
the map $\ev_3$ is analogous to $\ev_4$ in \eqref{eq:ev_4}. 
It is then clear that the fiber of the obvious projection
\[
\pi':\Cnf[\K_N,\R^3;3,1]\longrightarrow \K_N
\]
is $6$ dimensional.
Let us define
\begin{equation}\label{eq:beta'-h'}
	\begin{split}
\beta' & =h^\ast_{3;1}(\eta_{N}\times \eta_N\times \eta_{N}),\qquad \text{where}\\
 h_{3;1} & :\Cnf[\R;3]\times \K_N\times \Cnf[\R^3;4]\longrightarrow S^2\times S^2\times S^2,\\
 h_{3;1} & =(h_{1\mathbf{4}}\times h_{\mathbf{4}2}\times h_{3\mathbf{4}})\circ \ev_{3;1},
 \end{split}
\end{equation}
this map is encoded by the {\em tripod} diagram of Figure \ref{fig:tripodc2}.
%%%%%%%%%%%%%%%%%%%%%%%%%%%%%%%%%%%%%%%%%%
\begin{figure}[h!]
	\centering
	\includegraphics{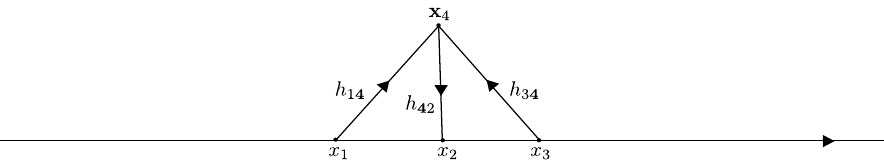}
	\caption{Trivalent tripod diagram associated to $h_{3;1}$.}
	\label{fig:tripodc2}
\end{figure}
%%%%%%%%%%%%%%%%%%%%%%%%%%%%%%%%%%%%%%%%%%
The required correction term for $\pi_\ast\beta$ is defined (over the plots of $\K_N$) as the pushforward $\pi'_\ast \beta'=\{(\pi'_\ast \beta')_{\mathcal{K}}\}_{\mathcal{K}:U\longrightarrow \K_N}$:
\[
\pi'_\ast \beta'=\varint_{\Cnf[\K_N,\R^3;3,1]} h^\ast_{3;1}(\eta_{N}\times \eta_N\times \eta_{N}).
\]
We may now rewrite the integral formula \eqref{eq:casson-integral-eta} of Theorem \ref{thm:c2-integral} as 
\begin{equation}\label{eq:casson-integral-eta-pi_ast}
	I(K)=\pi_\ast\beta(K)-\pi'_\ast\beta'(K).
\end{equation}
Clearly, it requires a proof 
that $\pi'_\ast\beta'$ provides a required correction for $\pi_\ast\beta$.
Such proof, as well as the  approach detailed above, was first provided for closed knots $\K_{S^1}$ in the work of Bott and Taubes in \cite{Bott-Taubes:1994}, with the standard rotationally invariant volume form $\omega$ on $S^2$ in place of $\eta_N$. The main result of \cite{Bott-Taubes:1994} yields the integral stated in \eqref{eq:casson-integral-original} over the compactified domains. Using the 
pushforwards, the integral in \eqref{eq:casson-integral-original} can be stated as\footnote{Note the rational coefficients, which appear due to a different count of principal faces.}
\[
 c_2(K)=\frac{1}{4}\pi_\ast \alpha(K)-\frac{1}{3}\pi'_\ast\alpha'(K),
\]
where $\alpha=h^\ast_4(\omega\times\omega)$ and $\alpha'=h^\ast_{3;1}(\omega\times\omega\times\omega)$.
The work in \cite{Bott-Taubes:1994} was further extended in \cite{Thurston:1999, Cattaneo-Cotta-Ramusino-Longoni:2002, Volic:2007, Koytcheff-Munson-Volic:2013}. In \cite{Cattaneo-Cotta-Ramusino-Longoni:2002} and \cite{Koytcheff-Munson-Volic:2013} the authors construct a diagram cochain complex which puts the correspondence between Figures \ref{fig:gaussdiagramc2}, \ref{fig:tripodc2} and maps: $h_4$, $h_{3;1}$ in the general setting of cohomology of long knots and links.

%%%%%%%%%%%%%%%%%%%%%%%%%%%%%%%%%%%%%%%%%%%%%%%%%%%%%%%%%%%%%%%%%%%
\section{Proof of Theorem \ref{thm:c2-integral}}\label{sec:proof-A}

A smooth path in $\K_N$: 
	\begin{equation}\label{eq:smooth-isotopy}
	K:\R\times [0,1] \longrightarrow \R^3,\qquad (\,\cdot\,,z)\longrightarrow K(\,\cdot\,,z)\in \K_N,
	\end{equation}
connecting $K_0$ and $K_1$, satisfies the following 
	\begin{enumerate}[label={(\it\roman*)}]
	\item  $K(\,\cdot\,,0)=K_0$, $K(\,\cdot\,,1)=K_1$, \label{eq:K-(i)}
	\item The unit vector\footnote{here $\dot{K}=\frac{\partial}{\partial x} K$ denotes a derivative along the knot}  $\frac{\dot{K}(x,z)}{|\dot{K}(x,z)|}$ never points along $N$, i.e.  there exists a $\delta$--disk: $D^2_\delta(N)$ around $N\in S^2$, such that $\frac{\dot{K}(x,z)}{|\dot{K}(x,z)|}\notin D^2_\delta(N)\subset S^2$ for all $x$ and $z$. Generally, we will require that $\varepsilon<\delta$ in the definition \eqref{eq:eta_N} of $\eta_N$. \label{eq:K-(ii)}
	\item $K(\,\cdot\,,z)$ is a long knot in $\K_N$, i.e. for each $z\in [0,1]$ and $x\not\in [-1,1]$ we have 
	$K(x,z)=x X$, where $X$ is a fixed vector perpendicular to $N$ (we choose the $x$--axis direction $X=(1,0,0)$ for that purpose). \label{eq:K-(iii)}
	\end{enumerate}
	With $I(K_i)$, $i=0,1$ defined in \eqref{eq:casson-integral-eta}  we need to show that $I(K_0)=I(K_1)$. Observe that $I(K(\,\cdot\, ,z))$ is a function of $z\in [0,1]$, and
	\[
	I(K_1)-I(K_0)=\varint^1_0 \frac{\partial}{\partial z} I(K(\,\cdot\, ,z))\, dz =\varint^1_0 d I(K(\,\cdot\, ,z)).
	\]
	 Therefore, it suffices to verify 
	 %%%%%%%%%%%%%%%%%%%%%%%%
	 \begin{equation}\label{eq:dI=0}
	  d I(K(\,\cdot\, ,z))=0.
	 \end{equation}
     Considering the path \eqref{eq:smooth-isotopy} as a plot in the knot space $\K_N$:
	   \begin{equation}\label{eq:plot-K}
	    \mathcal{K}:[0,1]\longrightarrow \K_N,\quad \mathcal{K}(z)=K(\,\cdot\,,z),
	   \end{equation}
	  in the differential form notation of Appendix \ref{sec:diffeology}, we obtain the $1$--form  
	 %%%%%%%%%%%%%%%%%%%%%%%%%%%%%%%%%%%%%%%%%%%
	 \begin{equation}\label{eq:dI}
	  d I(K(\,\cdot\, ,z)) =d \bigl((\pi_\ast\beta)_{\mathcal{K}} -(\pi'_\ast\beta')_{\mathcal{K}}\bigr)
	   =  -((\partial \pi)_\ast \beta)_{\mathcal{K}} +((\partial \pi')_\ast \beta')_{\mathcal{K}},
	 \end{equation}
	where the second identity follows from \eqref{eq:stokes-fiber2}.
	Recall from \eqref{eq:bdry-pushforward-faces}, $(\partial \pi)_\ast \beta$ and $(\partial \pi')_\ast \beta'$ can be written as the sum over the boundary faces (or strata) of $\Cnf[\R;4]$ and $\Cnf[\K_N,\R^3;3,1]$ respectively. In order to compute the sums it is sufficient to compute in local coordinates near the boundary faces of $\Cnf[\R;4]$ and $\Cnf[\K_N,\R^3;3,1]$ respectively.
	
	 We will follow the description of local coordinates near the boundary faces given in \cite{Bott-Taubes:1994, Volic:2007}, for equivalent approaches one may refer to  \cite{Cattaneo-Cotta-Ramusino-Longoni:2002} or \cite{Sinha:2004}.  

	For instance, every face (stratum) of\footnote{Similarly for $\Cnf[\K, \R^k;n,m]$.} $\Cnf[\R^k;n]$ denoted by $S_{\mathcal{A}}$, can be labeled by $|\mathcal{A}|=j$ nested subsets $\mathcal{A}=\{A_1,\ldots, A_j\}$, where  $A_k\subset\{1,2,\ldots,n\}$, $|A_k|\geq 2$ and for $A_p, A_q\in \mathcal{A}$:
	\[
	  A_p\cap A_q=\varnothing,\quad \text{or}\quad A_p\subseteq A_q.
	\]
	 As a direct consequence of coordinates description on $\Cnf[\R^k;n]$, one obtains that the codimension of $S_{\mathcal{A}}$ is $j$.  In particular, faces of $\Cnf[\R^k;n]$ of codimension $1$ correspond to a subset of the $n$ points either all colliding at the same time, or escaping to infinity at the same time. 
	
	%%%%%%%%%%%%%%%%%%%%
	\subsection{Faces of $\partial \Cnf[\R;4]$ appearing in $((\partial \pi)_\ast \beta)_{\mathcal{K}}$}\label{sec:comp-beta} Let us first consider $\Cnf[\R;4]$, the above description yields the following the codimension one faces of $\Cnf[\R;4]$ indexed by the following subsets of $\{1,2,3,4,\pm \infty\}$:
	
	\begin{equation}\label{eq:Cnf[R;4]-faces}
	\begin{split}
		\textit{ Principal Faces: }&\{1,2\},\{2,3\},\{3,4\}\\
		\textit{ Hidden Faces: }&\{1,2,3\},\{2,3,4\}\\
		\textit{ Anomalous Face: }&\{1,2,3,4\}\\
		\textit{ Faces at Infinity: }&\{1\to-\infty\},\{1,2\to-\infty\},\{1,2,3\to-\infty\},\{1,2,3,4\to-\infty\},\\
		&\{4\to+\infty\},\{3,4\to+\infty\}, \{2,3,4\to+\infty\},\{1,2,3,4\to+\infty\}.\\
	\end{split}
	\end{equation}	
	
	The pushforward $((\partial \pi)_\ast \beta)_{\mathcal{K}}$ will not vanish over the principal faces, i.e. faces indexed by $\{1,2\}$, $\{2,3\}$, $\{3,4\}$;  this non-vanishing is what requires the addition of the correction term $\pi_\ast\beta'$ in \eqref{eq:casson-integral-eta-pi_ast}.  On all other faces, the pushforward of the restriction of $\beta$ to that face is zero. The vanishing of the pushforwards of $\beta$ and $\beta'$ on non-principal faces will be shown in further sections. 
	
	 To give an idea of the local coordinates for these codimension $1$ faces, consider $A$ where $1\in A,\pm\infty\notin A$ and $|A|=a$. Then $S_A$ is the face where the first $a$ points on the knot collide away from $\infty$ in $\Cnf[\R;n]$.  Local coordinates for an open neighborhood of $S_A$ (which is itself obtained by setting $r=0$) are given by	
	\begin{equation}\label{eq:phi_A}
	\begin{split}
	&\phi_A:\Cnf(\R;a)\times\Cnf(\R;n-a+1)\times[0,\infty)\times [0,1] \longrightarrow\Cnf[\R;n]\times [0,1]\\
	&(u_1,\dots,u_a,x_1,\dots,x_{n-a+1},r, z)\longrightarrow  (x_1+ru_1,\dots,x_1+ru_a, x_2,\dots, x_{n-a+1}, z),
	\end{split}
	\end{equation}
	such that 	
	\[
	(1)\  x_i\neq x_j,\   i\neq j,\qquad
	(2)\ \displaystyle\sum_{i=1}^a u_i=0,\qquad 
	(3)\ \displaystyle\sum_{i=1}^a |u_i|^2=1.
	\] 
	With the help of the above coordinates, the pushforwards of $\beta$ along each boundary face $S_A$ are obtained by
	\begin{equation}\label{eq:partial-beta-restr}
	  (\partial \pi\bigl|_{S_A})_\ast \beta=(\partial \pi\bigl|_{r=0})_\ast (\phi^\ast_A\beta).
	\end{equation}
	Clearly, we have $(\partial \pi)_\ast \beta =\sum_A (\partial \pi\bigl|_{S_A})_\ast \beta$.
	%%%%%%%%%%%%%%%%%%%%%%%%%55
	\subsection{Computations for $((\partial \pi)_\ast \beta)_{\mathcal{K}}$ along the principal faces}\label{sec:principal-chords} For the principal faces: $A=\{1,2\}$, $\{2,3\}$, $\{3,4\}$, we may simplify the above coordinates. Additionally, for the purpose of computing the pushforwards $((\partial \pi\bigl|_{S_{i,j}})_\ast \beta)_{\mathcal{K}}$, we 
	add $[0,1]$ (the domain of the plot $\mathcal{K}:[0,1]\longrightarrow \K_N$ in \eqref{eq:plot-K}) and additional ``dummy'' $S^2$ factor 
	as follows\footnote{it is $\,\cdot\,+r$ (not $\,\cdot\,-r$) since $x_1<x_2<x_3<x_4$.}
	\[
	 	\hat{\phi}_{i,j}:\Cnf(\R;3)\times[0,\infty)\times S^2\times [0,1]  \longrightarrow\Cnf[\R;4]\times S^2\times [0,1],
	\]
	\begin{equation}\label{eq:phi_ij}
	\begin{split}
		\hat{\phi}_{1,2}:(x_1,x_2,x_3,r,\xi,z) & \longrightarrow  (x_1,x_1+ r,x_2,x_3,\xi,z),\\
		\hat{\phi}_{2,3}:(x_1,x_2,x_3,r,\xi,z) & \longrightarrow  (x_1,x_2,x_2+r,x_3,\xi,z),\\
		\hat{\phi}_{3,4}:(x_1,x_2,x_3,r,\xi,z) & \longrightarrow  (x_1,x_2,x_3,x_3+ r,\xi,z).
	\end{split}
	\end{equation}
	Calculating the pullback of the volume form\footnote{where $\omega$ is the Gauss form, as in \eqref{eq:lk-number-integral}.} $\sigma=dx_1\wedge dx_2\wedge dx_3\wedge dx_4\wedge\omega\wedge dz$ on $\Cnf[\R;4]\times S^2\times [0,1]$ over $\hat{\phi}_{i,j}$ shows that 
	\begin{equation}\label{eq:phi-orientations}
	 \hat{\phi}^\ast_{1,2}\sigma =  -\hat{\phi}^\ast_{2,3}\sigma =  \hat{\phi}^\ast_{3,4}\sigma,
	\end{equation}
	i.e. $\hat{\phi}^\ast_{1,2}$ and $\hat{\phi}^\ast_{3,4}$ are orientation preserving and $\hat{\phi}^\ast_{2,3}$ is orientation reversing.
	Computing the boundary extensions of Gauss maps\footnote{extended to the identity on the extra $S^2$ factor.} $(h\times \operatorname{id})\circ \hat{\phi}_{i,j}=(h_{13}\times h_{42}\times \operatorname{id})\circ \hat{\phi}_{i,j}$  along $S_{i,j}=\{r=0\}$ in $\hat{\phi}_{i,j}$ coordinates,  yields
	\begin{equation}\label{eq:h-bdry-extension}
	\begin{split}
		(h\times \operatorname{id}) & \circ\hat{\phi}_{1,2}(x_1,x_2,x_3,0,\xi,z)\\ 
		& =\displaystyle\lim_{r\to0}\left(\frac{K(x_2,z)-K(x_1,z)}{|K(x_2,z)-K(x_1,z)|},\frac{K(x_1+ r,z)-K(x_3,z)}{|K(x_1+ r,z)-K(x_3,z)|},\xi\right),
	\end{split}
	\end{equation}

	\begin{equation}\label{eq:h-bdry-extension-2}
	\begin{split}
	(h\times \operatorname{id}) & \circ\hat{\phi}_{2,3}(x_1,x_2,x_3,0,\xi,z)\\
	& = \left(\frac{K(x_2,z)-K(x_1,z)}{|K(x_2,z)-K(x_1,z)|},\frac{K(x_2,z)-K(x_3,z)}{|K(x_2,z)-K(x_3,z)|},\xi\right),
	\end{split}
	\end{equation}

	\begin{equation}\label{eq:h-bdry-extension-3}
	\begin{split}
	(h\times \operatorname{id}) & \circ\hat{\phi}_{3,4}(x_1,x_2,x_3,0,\xi,z)\\
	& = \left(\frac{K(x_3,z)-K(x_1,z)}{|K(x_3,z)-K(x_1,z)|},\frac{K(x_2,z)-K(x_3,z)}{|K(x_2,z)-K(x_3,z)|},\xi\right).
	\end{split}
	\end{equation}
	Adding the $S^2$ factor, does not change the evaluation of $((\partial \pi)_\ast \beta)_{\mathcal{K}}$. Indeed, using notation in \eqref{eq:partial-beta-restr}, yields for each $S_{i,j}$:
	\begin{equation}\label{eq:pi-beta-id-eval}
	\begin{split}
	((\partial \pi\bigl|_{S_{i,j}})_\ast \beta)_{\mathcal{K}} & =\bigl((\partial \pi\bigl|_{S_{i,j}})_\ast (h_{13}\times h_{42})^\ast (\eta_N\times\eta_N)\bigr)_{\mathcal{K}}\\
	 & =\bigl((\partial \pi\bigl|_{S_{i,j}\times S^2})_\ast (h_{13}\times h_{42}\times \operatorname{id})^\ast (\eta_N\times\eta_N\times\eta_N)\bigr)_{\mathcal{K}}\\
	 & =((\partial \pi\bigl|_{S_{i,j}\times S^2})_\ast (\beta\wedge \eta_N))_{\mathcal{K}},
	\end{split}
	\end{equation}
	where the second idenitity follows from the product integration and $\varint_{S^2} \eta_N=1$.	

	%%%%%%%%%%%%%%%%%%%%
	\subsection{Faces of $\partial \Cnf[\K_N,\R^3;3,1]$ appearing in $((\partial \pi')_\ast \beta')_{\mathcal{K}}$}
	By \eqref{eq:bdry-pushforward-faces}, $(\partial \pi')_\ast \beta'$is the sum of push forwards of $\beta'$ along the boundary faces of the fiber of 
	\[
	\pi':\Cnf[\K_N,\R^3;3,1]\longrightarrow \K_N,
	\] 
	This boundary has faces indexed by the following	
	\begin{equation}\label{eq:Cnf[K,R^3;3,1]-faces}
	\begin{split}
		\textit{ Principal Faces: }&\{1,2\},\{2,3\},\{1,\mathbf{4}\},\{2,\mathbf{4}\},\{3,\mathbf{4}\};\\
		\textit{ Hidden Faces: }&\{1,2,3\},\{1,2,\mathbf{4}\},\{2,3,\mathbf{4}\};\\
		\textit{ Anomalous Face: }&\{1,2,3,\mathbf{4}\};
	\end{split}
	\end{equation}
	\begin{equation}\label{eq:Cnf[K,R^3;3,1]-faces-2}
	\begin{split}
		\textit{ Faces at Infinity: }&\{1\to-\infty\},\{1,2\to-\infty\}\{1,2,3\to-\infty\},\{3\to\infty\},\{2,3\to\infty\},\\
		&\{1,2,3\to\infty\},\{1\to-\infty \ \& \ \mathbf{4}\to\infty\},\\
		&\{1,2\to-\infty\ \& \ \mathbf{4}\to\infty\},\{1,2,3\to-\infty\ \& \ \mathbf{4}\to\infty\},\\
		&\{3\to\infty\ \& \ \mathbf{4}\to\infty\},\{2,3\to\infty\ \& \ \mathbf{4}\to\infty\},\\
		&\{1,2,3\to\infty\ \& \ \mathbf{4}\to\infty\},\{\mathbf{4}\to\infty\},
	\end{split}
	\end{equation}
	where the indices are associated to the colliding coordinates. The interior of the fiber of $\Cnf[\K_N,\R^3;3,1]$, over a fixed $K\in \K_N$ is given as
	\[
	\begin{split}
	\Cnf(K,\R^3;3,1) & \cong\bigl\{(x_1,x_2,x_3,\mathbf{x}_4)\ |\ (x_1,x_2,x_3)\in \Cnf(\R;3),\\ 
	& \qquad\qquad\qquad\qquad\quad  \mathbf{x}_4\in \R^3-\{K(x_1), K(x_2), K(x_3)\} \bigr\}.
	\end{split}
	\]
	Note that $``\mathbf{4}\to\infty"$ means that the norm of $\mathbf{x}_4\in\R^3$ is tending to $\infty$. 
	
	%%%%%%%%%%%%%%%%%%%%%%%%%55
	\subsection{Computations for $((\partial \pi')_\ast \beta')_{\mathcal{K}}$ along the principal faces}  Similarly as in Section \ref{sec:comp-beta}, we first 
	calculate in local coordinates in the neighborhood of principal faces: $A=\{1,\mathbf{4}\}$, $\{2,\mathbf{4}\}$, $\{3,\mathbf{4}\}$.
	The local coordinates about 
	$S_{i,\mathbf{4}}=\{r=0\}$, amended by the factor $[0,1]$,  are given as
	\[
	\begin{split}
	\psi_{i,\mathbf{4}}:\Cnf(\R;3)\times [0,\infty)\times S^2\times [0,1] &\longrightarrow \Cnf[\R^3;4]\times [0,1]\\
	\psi_{i,\mathbf{4}}: (x_1,x_2,x_3,r,\xi, z) & \longrightarrow (K(x_1,z),K(x_2,z),K(x_3,z), K(x_i,z)+r\xi, z).
	\end{split}
	\]
	A routine calculation shows that $\psi_{i,\mathbf{4}}$'s do not change the orientation.
	 Computing the smooth boundary extensions of Gauss maps $h_{3;1}\circ \psi_{i,\mathbf{4}}$ along $S_{i,\mathbf{4}}$,
	yields:
	\begin{equation}\label{eq:h'-bdry-extension}
	\begin{split}
	(h_{3;1}\circ\psi_{1,\mathbf{4}}) & (x_1,x_2,x_3,0,\xi,z)=\displaystyle\lim_{r\to0}\bigg(\frac{K(x_1,z)+r\xi-K(x_1,z)}{|K(x_1,z)+r\xi-K(x_1,z)|},
	\end{split}
	\end{equation}
	\begin{equation}\label{eq:h'-bdry-extension-1}
	\begin{split}
	&\qquad\qquad\frac{K(x_2,z)-[K(x_1,z)+r\xi]}{|K(x_2,z)-[K(x_1,z)+r\xi]|},\frac{K(x_1,z)+r\xi-K(x_3,z)}{|K(x_1,z)+r\xi-K(x_3,z)|}\bigg)\\
	&\qquad\qquad=\left(\xi,\frac{K(x_2,z)-K(x_1,z)}{|K(x_2,z)-K(x_1,z)|},\frac{K(x_1,z)-K(x_3,z)}{|K(x_1,z)-K(x_3,z)|}\right),	
	\end{split}
	\end{equation}
	\begin{equation}\label{eq:h'-bdry-extension-2}
	\begin{split}
	(h_{3;1}\circ\psi_{2,\mathbf{4}}) & (x_1,x_2,x_3,0,\xi,z) = \left(\frac{K(x_2,z)-K(x_1,z)}{|K(x_2,z)-K(x_1,z)|},-\xi,\frac{K(x_2,z)-K(x_3,z)}{|K(x_2,z)-K(x_3,z)|}\right),
	\end{split}
	\end{equation}	
	\begin{equation}\label{eq:h'-bdry-extension-3}
	\begin{split}
	(h_{3;1}\circ\psi_{3,\mathbf{4}}) & (x_1,x_2,x_3,0,\xi,z) =\left(\frac{K(x_3,z)-K(x_1,z)}{|K(x_3,z)-K(x_1,z)|},\frac{K(x_2,z)-K(x_3,z)}{|K(x_2,z)-K(x_3,z)|},\xi\right).
	\end{split}
	\end{equation}
	
	\no Comparing, the smooth extensions of the Gauss maps in \eqref{eq:h-bdry-extension}, \eqref{eq:h-bdry-extension-2}, \eqref{eq:h-bdry-extension-3} and \eqref{eq:h'-bdry-extension}, \eqref{eq:h'-bdry-extension-2}, \eqref{eq:h'-bdry-extension-3} yields, up to a permutation of the $S^2$ factors\footnote{which do not change the orientation} and 
	antipodal symmetry $A$ in \eqref{eq:antipodal-eta_N} 
	\begin{equation}\label{eq:principal-faces-h-equal}
	 (h\times \operatorname{id})\circ\hat{\phi}_{i,i+1}=h_{3;1}\circ\psi_{i,\mathbf{4}},\qquad i=1,2,3.
	\end{equation}
	\subsection{Canceling of the pushforwards along the principal faces} Thanks to \eqref{eq:principal-faces-h-equal} and Proposition \ref{prop:equal-along}, we obtain  
	\begin{equation}\label{eq:principal-faces-cancel}
	 ((\partial \pi\bigl|_{S_{i,i+1}})_\ast \beta)_{\mathcal{K}}=((\partial \pi\bigl|_{S_{i,i+1}\times S^2})_\ast (\beta\wedge \eta_N))_{\mathcal{K}}=((\partial \pi'\bigl|_{S_{i,\mathbf{4}}})_\ast \beta')_{\mathcal{K}},\qquad i=1,2,3.
	\end{equation}
	 Note that for $i=2$ the sign in \eqref{eq:principal-faces-cancel}, which appears in \eqref{eq:h'-bdry-extension-2} is compensated with the orientation of the face $S_{2,3}$ in coordinates $\hat{\phi}_{2,3}$, see \eqref{eq:phi-orientations}.

%%%%%%%%%%%%%%%%%%%%%%%%%%%%%%%%%%%%%%%%%%%%%%%%%%%%%%%%%%%%%%%%%
\subsection{Vanishing over Hidden and Anomalous Faces of $\Cnf[\R;4]$}\label{sec:vanishing-faces-1}

Suppose that $S_A$ is a codimension one face of $\Cnf[\R;4]$ indexed by a subset $A\subseteq\{1,2,3,4\}$, $|A|>2$.  We claim that if $A$ contains $\{1,3\}$ or $\{2,4\}$, then 
%%%%%%%%%%%%%%%%%%%%%%%%%%%%%%%%%%%%%%
\begin{equation}\label{eq:phi-A-beta}
 \phi^\ast_A\beta\bigl|_{r=0}=0.
\end{equation}
 This applies to the hidden faces and the anomalous face listed in \eqref{eq:Cnf[R;4]-faces}. 

To see that this is true for $A$ containing $\{1,3\}$, consider the local parametrization $\phi_{A}$ in \eqref{eq:phi_A}, where the points are colliding at $K(x_1)$.  Composed with the map $h_{13}$, $\phi_{A}$ extends along $S_A$ as follows:
\[
h_{13}\circ\phi_{A}(x_1,\ldots,u_3,\ldots,r,z)=\lim_{r\to 0} \frac{K(x_1+r u_3,z)-K(x_1,z)}{|K(x_1+r u_3,z)-K(x_1,z)|} =\frac{K'(x_1,z)}{|K'(x_1,z)|},
\]
By \ref{eq:K-(ii)} on page \pageref{eq:K-(ii)}, $\frac{K'(x,z)}{|K'(x,z)|}$ is excluded from the support of $\eta_N$, thus $\phi^\ast_{A}(h_{13}^\ast \eta_N)\bigl|_{r=0}=0$, and therefore $\beta|_{S_{A}}=0$.  
Similarly, if $A$ contains $\{2,4\}$  then, the local parametrization $\phi_{A}$ composed with the map $h_{42}$ has the same property, thus $\phi^\ast_{A}(h^\ast_{42}\eta_N))=0$.

%%%%%%%%%%%%%%%%%%%%%%%%%%%%%%%%%%%%%%%%%%%%%%%%%%%%%%%%%%%%%%%%%
\subsection{Vanishing over Faces at Infinity  of $\Cnf[\R;4]$}\label{sec:vanishing-faces-2}
Let us perform a computation for $S_A$ when e.g. $4\to +\infty$ in $A$, and $2\notin A$, i.e. $K(x_2)$ is fixed on the knot. For such faces, the local coordinates in \eqref{eq:phi_A} differ only by replacing $ru_i$ with $\tfrac{1}{r} u_i$ in the definition. The map $h_{42}$ extends along $S_A$, as follows
\[
\begin{split}
h_{42}\circ\phi_{A}(\ldots,x_2,\ldots,x_4,\ldots,u_4,r,z) & =\lim_{r\to 0} \frac{K(x_2,z)-K(x_4+\tfrac{1}{r} u_4,z)}{|K(x_2,z)-K(x_4+\tfrac{1}{r} u_4,z)|}=\lim_{r\to 0} \frac{K(x_2,z)-(x_4+\tfrac{1}{r} u_4)X}{|K(x_2,z)-(x_4+\tfrac{1}{r} u_4)X|}\\
& =\lim_{r\to 0} \frac{r K(x_2,z)- (r x_4+u_4)X}{|r K(x_2,z)- (r x_4+u_4)X|}=-X
\end{split}
\]
where $X$ is the direction vector orthogonal to $N$, defined in the condition \ref{eq:K-(iii)} on page \pageref{eq:K-(ii)}. By assumption  \ref{eq:K-(ii)}, $-X$ is excluded from the support of $\eta_N$, thus $h_{42}^\ast(\eta_N)=0$, and therefore\footnote{For small enough $r$ the image of $h_{42}$ is outside of the support of $\eta_N$.} $\phi^\ast_{A}\beta|_{S_{A}}=\phi^\ast_{A}(h_{41}^\ast\eta_N\wedge h_{24}^\ast\eta_N\wedge h_{43}^\ast\eta_N)\bigr|_{r=0}=0$. Similarly, if $A$ contains $\{2\to +\infty,4\to +\infty\}$  then in the above limit, we replace $K(x_2,z)$ with $K(x_2+\tfrac{1}{r} u_2,z)$ and similar analysis yields the same result. The remaining cases 
are fully analogous.

%%%%%%%%%%%%%%%%%%%%%%%%%%%%%%%%%%%%%%%%%%%%%%%%%%%%%%%%%%%%%%%%%
\subsection{Vanishing over faces $\{1,2\}$, $\{2,3\}$, and $\{1,2,3\}$ of $\Cnf[\K_N,\R^3;3,1]$}\label{sec:vanishing-faces-3} 
The local coordinates around $S_{\{1,2\}}=\{r=0\}$ is given as
\[
\begin{split}
	\psi_{1,2}:\Cnf(\R;2)\times [0,\infty)\times (\R^3-\{K(x_1,z)\})\times [0,1] &\longrightarrow \Cnf[\R^3;4]\times [0,1]\\
	\psi_{1,2}: (x_1,x_2,r,\bx, z) \longrightarrow (K(x_1,z), K(x_1+r,z), & K(x_2,z),  \bx, z).
\end{split}
\]
\begin{equation}\label{eq:h'-bdry-extension-5}
	\begin{split}
		(h'\circ\psi_{1,2}) & (x_1,x_2,0,\bx, z)=\displaystyle\lim_{r\to0}\bigg(\frac{\bx-K(x_1,z)}{|\bx-K(x_1,z)|},\frac{K(x_1+r,z)-\bx}{|K(x_1+r,z)-\bx|},\frac{\bx-K(x_2,z)}{|\bx-K(x_2,z)|}\bigg)\\
		&\qquad\qquad=\left(\frac{\bx-K(x_1,z)}{|\bx-K(x_1,z)|},\frac{K(x_1,z)-\bx}{|K(x_1,z)-\bx|},\frac{\bx-K(x_2,z)}{|\bx-K(x_2,z)|}\right).
	\end{split}
\end{equation}
Clearly, the first two factors are equal after applying the antipodal map $A$, \eqref{eq:antipodal-eta_N},  ($h_{24}=A\circ h_{42}$). Therefore,
\[
\begin{split}
 \psi_{1,2}^\ast (h^\ast_{1\mathbf{4}}\eta_N)\bigl|_{r=0} =\psi_{1,2}^\ast (h^\ast_{\mathbf{4}2} A^\ast\eta_N)|_{r=0} & =-\psi_{1,2}^\ast (h^\ast_{\mathbf{4}2} \eta_S)|_{r=0},\qquad \text{and}\qquad\\
\psi_{1,2}^\ast\beta'|_{r=0}=\psi_{1,2}^\ast (h^\ast_{1\mathbf{4}}\eta_N\wedge h^\ast_{\mathbf{4}2} \eta_N\wedge h^\ast_{\mathbf{4}3} \eta_N)|_{r=0} & =-\psi_{1,2}^\ast (h^\ast_{\mathbf{4}2} (\eta_S\wedge \eta_N)\wedge h^\ast_{\mathbf{4}3} \eta_N)|_{r=0}=0,\\
 \end{split}
\]
where the vanishing in the last identity follows since $\eta_S\wedge \eta_N=0$ on $S^2$ as the degree $4$ form. For the face $S_{\{2,3\}}$ and $S_{\{1,2,3\}}$ the computation is analogous.

%%%%%%%%%%%%%%%%%%%%%%%%%%%%%%%%%%%%%%%%%%%%%%%%%%%%%%%%%%%%%%%%%
\subsection{Vanishing over the hidden and anomalous faces  of $\Cnf[\K_N,\R^3;3,1]$}\label{sec:vanishing-faces-4} 
All of the faces, under consideration, contain the indices $\{i,j,\mathbf{4}\}$. Without loss of generality, we present calculation for $A=\{1,2,\mathbf{4}\}$, the local coordinates around $S_{\{1,2,\mathbf{4}\}}=\{r=0\}$ are given as
\[
\begin{split}
\psi_A:\Cnf(\R;2)\times S^3\times [0,\infty)\times [0,1] &\longrightarrow \Cnf[\R^3;4]\times [0,1]\\
\psi_A: (x_1,x_2,u,\xi, r, z) \longrightarrow (K(x_1,z), K(x_1+r u,z), & K(x_2,z), K(x_1,z) +r\xi, z),
\end{split}
\]
(where $(u,\xi)\in S^3$, i.e. $u^2+|\xi|^2=1$, see \cite{Volic:2007,Bott-Taubes:1994}),
\begin{equation}\label{eq:h'-bdry-extension-4}
\begin{split}
(h'\circ\psi_A) & (x_1,x_2,u,\xi, 0, z)\\
& = \displaystyle\lim_{r\to 0}\bigg(\frac{(K(x_1,z) + r\xi)-K(x_1,z)}{|(K(x_1,z) +r\xi)-K(x_1,z)|},\\
&\qquad \frac{K(x_1+ r u,z)-(K(x_1,z) +r\xi)}{|K(x_1+r u,z)-(K(x_1,z) +r\xi)|}, \frac{(K(x_1,z) +r\xi)-K(x_2,z)}{|(K(x_1,z) +r\xi)-K(x_2,z)|}\bigg)\\
&=\left(\frac{\xi}{|\xi|},\frac{K'(x_1,z) u-\xi}{|K'(x_1,z) u-\xi|},\frac{K(x_1,z)-K(x_2,z)}{|K(x_1,z)-K(x_2,z)|}\right).
\end{split}
\end{equation}
Let us estimate how much the image of the above map intersects the support of $\eta_N\times \eta_N\times \eta_N$, which is a product of small $\varepsilon$--disks around $N\in S^2$ of each $S^2$ factor in the codomain of $h'$ in \eqref{eq:beta'-h'}. Thanks to \ref{eq:K-(ii)}, 
we obtain\footnote{where $\frac{1}{\varepsilon} O(\varepsilon)\longrightarrow \text{const}$ for $\varepsilon\longrightarrow 0$.}
\[
\frac{\xi}{|\xi|}=N+O(\varepsilon),\qquad \frac{K'(x_1,z) u-\xi}{|K'(x_1,z) u-\xi|}=N+O(\varepsilon),
\]
which yields
\[
\begin{split}
\xi & =|\xi| N+O(\varepsilon),\qquad K'(x_1,z) u-\xi=|K'(x_1,z) u-\xi| N+O(\varepsilon),\\
 K'(x_1,z) u & = (|\xi|+|K'(x_1,z) u-\xi|) N+O(\varepsilon).
\end{split}
\]
Taking the cross product $\times N$ of both sides of the last identity, we obtain
\[
(K'(x_1,z)\times N)  u  = O(\varepsilon),\quad \Rightarrow\quad u  = O(\varepsilon),
\]
since  $0<|K'(x_1,z)\times N|<\infty$, (for all $x_1$ and $z$) and $|u|\leq 1$.  In conclusion, we obtain for small $\varepsilon$
\[
 (h'\circ\psi_A)(x_1,x_2,u,\xi, 0, z)\approx \left(\frac{\xi}{|\xi|},-\frac{\xi}{|\xi|},\frac{K(x_1,z)-K(x_2,z)}{|K(x_1,z)-K(x_2,z)|}\right),
\]
thus the first two coordinates certainly have no intersection with the support of the first two factors of $\eta_N\times \eta_N\times \eta_N$ and therefore $ \psi^\ast_A\beta'\bigl|_{r=0}=0$ as claimed.

%%%%%%%%%%%%%%%%%%%%%%%%%%%%%%%%%%%%%%%%%%%%%%%%%%%%%%%%%%%%%%%%%
\subsection{Vanishing over the faces at infinity of $\Cnf[\K_N,\R^3;3,1]$}\label{sec:vanishing-faces-5} 
In the case only ${\bf 4}\to \infty$ in $A$, i.e. all $K(x_i)$, $1\leq i\leq 3$ are fixed on the knot, we have 
\[
\begin{split}
(h_{{\bf 4} 1}\times h_{2 {\bf 4}}\times h_{{\bf 4} 3}) & \circ\phi_{A}(x_1,x_2,x_3, 0,\xi,z)   = \displaystyle\lim_{r\to 0}\bigg(\frac{ \frac{1}{r}\xi-K(x_1,z)}{|\frac{1}{r}\xi-K(x_1,z)|},\frac{K(x_2,z)-\frac{1}{r}\xi}{|K(x_2,z)-\frac{1}{r}\xi|},\\
&\qquad  \frac{ \frac{1}{r}\xi-K(x_3,z)}{|\frac{1}{r}\xi-K(x_3,z)|}\bigg)=\big(\xi,-\xi,\xi\big),
\end{split}
\]
again the pullback $\phi^\ast_{A}((h_{{\bf 4} 1}\times h_{2 {\bf 4}}\times h_{{\bf 4} 3})^\ast(\eta_N\times \eta_N\times \eta_N)\bigl|_{r=0}=0$, because the above map factors through a space of lower dimension ($\dim=2$).

In the case $({\bf 4}\to \infty)\in A$ and $i\in A$, $1\leq i\leq 3$, i.e. $K(x_i)$ moves to $\pm\infty$ along the knot and $\bx_{4}\to \infty$ in $\R^3$, see \eqref{eq:Cnf[K,R^3;3,1]-faces-2}, the map $h_{{\bf 4} i}$ extends along $S_A$ as follows
\[
\begin{split}
h_{{\bf 4} i}\circ\phi_{A}(\ldots,x_i,\ldots,u_i,\ldots, 0,\xi,z) & =\lim_{r\to 0} \frac{\frac{1}{r}\xi-K(x_i+\frac{1}{r} u_i,z)}{|\frac{1}{r}\xi-K(x_i+\frac{1}{r} u_i,z)|}\\
=\lim_{r\to 0} \frac{\frac{1}{r}\xi-(x_i+\frac{1}{r} u_i) X}{|\frac{1}{r}\xi-(x_i+\frac{1}{r} u_i) X|}=\frac{\xi- u_i X}{|\xi- u_i X|}.
\end{split}
\]
where $X$ is the direction vector orthogonal to $N$, defined in the condition \ref{eq:K-(iii)}. Hence $h_{{\bf 4} i}\circ \phi_{A}$ is either equal $\pm \frac{\xi}{|\xi|}$ or $\frac{\xi- u_i X}{|\xi- u_i X|}$ the map $(h_{{\bf 4} 1}\times h_{2 {\bf 4}}\times h_{{\bf 4} 3})\circ\phi_{A}$ factors through a lower dimensional space as in the previous case.

The only case left is $\{ {\bf 4}\to \infty\} \not\in A$, and $\{i\to \pm \infty\}\in A$, when
\[
\begin{split}
h_{{\bf 4} i}\circ\phi_{A}(\ldots,x_i,\ldots,u_i,\ldots, 0,\bx,z) & =\lim_{r\to 0} \frac{\bx-(x_i\pm \frac{1}{r} u_i) X}{|\bx-(x_i\pm \frac{1}{r} u_i) X|}=\pm\frac{ u_i X}{| u_i X|}=\pm X.
\end{split}
\]
which analogous to the case already considered in Section \ref{sec:vanishing-faces-2}.

%%%%%%%%%%%%%%%%%%%%%%%%%%%%%%%%%%%%%%%%%%%%%%%%%%%%%%%%%%%%%%%%%
\subsection{Finishing the proof of invariance of $I(K)$ over $\K_N$}
Combining the results of Sections \ref{sec:principal-chords} -- \ref{sec:vanishing-faces-5} 
\[
d I(K(\,\cdot\, ,z)) =d \bigl((\pi_\ast\beta)_{\mathcal{K}} -(\pi'_\ast\beta')_{\mathcal{K}}\bigr)
=  -((\partial \pi)_\ast \beta)_{\mathcal{K}} +((\partial \pi')_\ast \beta')_{\mathcal{K}}
\]
proves invariance of the integral in \eqref{eq:casson-integral-eta} of Theorem \ref{thm:c2-integral}.

%%%%%%%%%%%%%%%%%%%%%%%%%%%%%%%%%%%%%%%%%%%%%%%%%%%%%%%%%%%%%%%%%
\subsection{Arrow diagram counting for double crossing projections}\label{sec:arrow-counting-dbl-crossing}
Recall the standard argument \cite[p. 41]{Bott-Tu:1982}, tells us that for any map $F:M^m\longrightarrow N^n$, $m=n$, between manifolds (with corners) and a regular value $y\in \operatorname{int}(N)$, where $F^{-1}(y)=\{x_1,\ldots,x_m\}\subset \operatorname{int}(M)$, the local degree \cite{Hatcher:2002} of $F$ at $y$ is given as 
\[
 \deg_y F = \sum_{i} \operatorname{sign} \bigl(\operatorname{det} DF(x_i)\bigr)=\int_M F^\ast\eta_y,
\]
where $\eta_y$ is a unit $n$--form ($\varint_N \eta_y=1$), supported on a sufficiently small neighborhood of $y$ in $N$. 
Clearly, this argument was used in Section \ref{sec:integrals-combinatorial} to obtain the standard formula for the linking number in \eqref{eq:lk-crossing-count}. For the long knot, which is near its regular planar projection we have
\begin{lemma}\label{lem:c_2-arrow-regular-crossing}
	Suppose $K$ projects on the plane orthogonal to $N$ in a regular double crossing knot diagram, then the integral $I(K)$ in \eqref{eq:casson-integral-eta} equals 
	\begin{equation}\label{eq:I(K)-gauss-dbl-crossing}
	I(K)=\varint\limits_{\Cnf[\R;4]} h^\ast_{13}\eta_N\wedge h^\ast_{42}\eta_N=\langle \vvcenteredinclude{.3}{gd_c2_chord.pdf},G_K\rangle.
	\end{equation}
\end{lemma}
Since $\langle \vvcenteredinclude{.3}{gd_c2_chord.pdf},G_K\rangle$ coincides with the arrow diagram formula of \cite{Polyak-Viro:2001} for the 
second coefficient of the Conway polynomial $c_2(K)$, and $I(K)$ is the knot invariant, we have
\begin{cor}\label{cor:I(K)=c_2(K)}
	$I(K)=c_2(K)$.
\end{cor}
\begin{proof}[Proof of Lemma \eqref{lem:c_2-arrow-regular-crossing}]
	Clearly, we may isotope $K$ so that it coincides with the image of its projection except at the double crossings, where we can create arbitrary ``small'' overpasses. 
	
	For the Gauss map
	\[
	 h =h_{13}\times h_{42}:\Cnf(\R;4)\longrightarrow S^2\times S^2,
	\]
	the inverse image of $(N,N)$ under $h$ is given as
	\[
	 h^{-1}(N,N)=\{(x_1,x_2,x_3,x_4)\ |\ \tfrac{K(x_3)-K(x_1)}{|K(x_3)-K(x_1)|}=N, \tfrac{K(x_2)-K(x_4)}{|K(x_2)-K(x_4)|}=N\}.
	\]
	Since the crossings are transverse double crossings, $(N,N)$ is the regular value and the local degree formula, reviewed at the beginning of this section yields
	\[
	\varint\limits_{\Cnf[\R;4]} h^\ast_{13}\eta_N\wedge h^\ast_{42}\eta_N=\langle \vvcenteredinclude{.3}{gd_c2_chord.pdf},G_K\rangle.
	\]
	\begin{rem}\label{rem:single-cross-contribution}
		Note that near the double crossings, configurations $(x_1,x_2,x_3,x_4)$  can easily produce arrows parallel to $N$, however, by the definition of  $h=h_{13}\times h_{42}$ these arrows can never both point in the direction of $N$, Figure \ref{fig:eta-support-double}. This is in contrast with the work in \cite[Proposition 4.3]{Lin-Wang:1996}, where such pairs of arrows are counted in the formula for $c_2(K)$.
	\end{rem}
	
	To finish the proof of \eqref{eq:I(K)-gauss-dbl-crossing}, we need to show 
	\begin{equation}\label{eq:tripod-vanish}
	 \varint\limits_{\Cnf[\K_N,\R^3;3,1]} h^\ast_{1{\bf 4}}\eta_N\wedge h^\ast_{{\bf 4}2}\eta_N\wedge h^\ast_{3{\bf 4}}\eta_N=0.
	\end{equation}
Since the support of $\eta_N$ is a small $\varepsilon$--ball around $N$; $x_i$, $\mathbf{x}_4$, $i=1,2,3$ need to satisfy\footnote{$\approx_\varepsilon$ stands for $\varepsilon$--close in the standard round metric on $S^2$}
\begin{equation}\label{eq:tripod-support-condition}
 \tfrac{\mathbf{x}_4-K(x_1)}{|\mathbf{x}_4-K(x_1)|}\approx_\varepsilon N,\qquad \tfrac{K(x_2)-\mathbf{x}_4}{|K(x_2)-\mathbf{x}_4|}\approx_\varepsilon N,\qquad \tfrac{\mathbf{x}_4-K(x_3)}{|\mathbf{x}_4-K(x_3)|}\approx_\varepsilon N,
\end{equation}
and therefore, the only possible configuration for $K(x_i)$, $\mathbf{x}_4$, $i=1$,$2$,$3$ occurs in the neighborhood of the double crossing as illustrated on Figure \ref{fig:eta-support-double}, where the arrows are all pointing into directions $\varepsilon$--close to $N$. However, for any $(x_1, x_2, x_3, \mathbf{x}_4) \in \Cnf(\K_N,\R^3;3,1)$ it forces $x_1$ to be close to $x_3$ and since $x_1 < x_2 < x_3$ no such configurations are possible. \qedhere
\end{proof}

%%%%%%%%%%%%%%%%%%%%%%%%%%%%%%%%%%%%%%%%%
\begin{figure}[h!]
	\centering
	%% Creator: Inkscape 1.3 (0e150ed, 2023-07-21), www.inkscape.org
%% Accompanies image file 'double-cross-parallel.pdf' (pdf, eps, ps)
\begingroup%
  \makeatletter%
  \providecommand\color[2][]{%
    \errmessage{(Inkscape) Color is used for the text in Inkscape, but the package 'color.sty' is not loaded}%
    \renewcommand\color[2][]{}%
  }%
  \providecommand\transparent[1]{%
    \errmessage{(Inkscape) Transparency is used (non-zero) for the text in Inkscape, but the package 'transparent.sty' is not loaded}%
    \renewcommand\transparent[1]{}%
  }%
  \providecommand\rotatebox[2]{#2}%
  \newcommand*\fsize{\dimexpr\f@size pt\relax}%
  \newcommand*\lineheight[1]{\fontsize{\fsize}{#1\fsize}\selectfont}%
  \ifx\svgwidth\undefined%
    \setlength{\unitlength}{114.49733824bp}%
    \ifx\svgscale\undefined%
      \relax%
    \else%
      \setlength{\unitlength}{\unitlength * \real{\svgscale}}%
    \fi%
  \else%
    \setlength{\unitlength}{\svgwidth}%
  \fi%
  \global\let\svgwidth\undefined%
  \global\let\svgscale\undefined%
  \makeatother%
  \begin{picture}(1,0.78238477)%
    \lineheight{1}%
    \setlength\tabcolsep{0pt}%
    \put(0,0){\includegraphics[width=\unitlength,page=1]{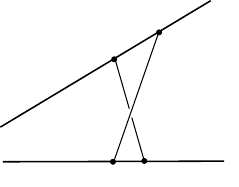}}%
    \put(0.23836678,0.01405602){\makebox(0,0)[lt]{\lineheight{1.25}\smash{\begin{tabular}[t]{l}$K(x_2)$\end{tabular}}}}%
    \put(0.1514191,0.560478){\makebox(0,0)[lt]{\lineheight{1.25}\smash{\begin{tabular}[t]{l}$K(x_3)$\end{tabular}}}}%
    \put(0.67394876,0.58420627){\makebox(0,0)[lt]{\lineheight{1.25}\smash{\begin{tabular}[t]{l}$K(x_4)$\end{tabular}}}}%
    \put(0.60122436,0.02067553){\makebox(0,0)[lt]{\lineheight{1.25}\smash{\begin{tabular}[t]{l}$K(x_1)$\end{tabular}}}}%
    \put(0,0){\includegraphics[width=\unitlength,page=2]{double-cross-parallel.pdf}}%
  \end{picture}%
\endgroup%
\qquad\qquad\qquad
	%% Creator: Inkscape 1.3 (0e150ed, 2023-07-21), www.inkscape.org
%% Accompanies image file 'double-cross-trip-2.pdf' (pdf, eps, ps)
\begingroup%
  \makeatletter%
  \providecommand\color[2][]{%
    \errmessage{(Inkscape) Color is used for the text in Inkscape, but the package 'color.sty' is not loaded}%
    \renewcommand\color[2][]{}%
  }%
  \providecommand\transparent[1]{%
    \errmessage{(Inkscape) Transparency is used (non-zero) for the text in Inkscape, but the package 'transparent.sty' is not loaded}%
    \renewcommand\transparent[1]{}%
  }%
  \providecommand\rotatebox[2]{#2}%
  \newcommand*\fsize{\dimexpr\f@size pt\relax}%
  \newcommand*\lineheight[1]{\fontsize{\fsize}{#1\fsize}\selectfont}%
  \ifx\svgwidth\undefined%
    \setlength{\unitlength}{151.7245976bp}%
    \ifx\svgscale\undefined%
      \relax%
    \else%
      \setlength{\unitlength}{\unitlength * \real{\svgscale}}%
    \fi%
  \else%
    \setlength{\unitlength}{\svgwidth}%
  \fi%
  \global\let\svgwidth\undefined%
  \global\let\svgscale\undefined%
  \makeatother%
  \begin{picture}(1,0.59307002)%
    \lineheight{1}%
    \setlength\tabcolsep{0pt}%
    \put(0,0){\includegraphics[width=\unitlength,page=1]{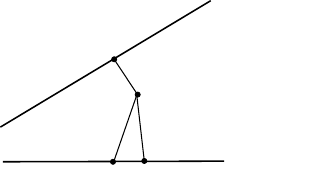}}%
    \put(0.18357211,0.4535075){\makebox(0,0)[lt]{\lineheight{1.25}\smash{\begin{tabular}[t]{l}$K(x_2)$\end{tabular}}}}%
    \put(0.17988093,0.01325898){\makebox(0,0)[lt]{\lineheight{1.25}\smash{\begin{tabular}[t]{l}$K(x_3)$\end{tabular}}}}%
    \put(0.45863978,0.01353658){\makebox(0,0)[lt]{\lineheight{1.25}\smash{\begin{tabular}[t]{l}$K(x_1)$\end{tabular}}}}%
    \put(0,0){\includegraphics[width=\unitlength,page=2]{double-cross-trip-2.pdf}}%
    \put(0.44875286,0.30612567){\color[rgb]{0,0,0}\makebox(0,0)[lt]{\lineheight{1.25}\smash{\begin{tabular}[t]{l}$\mathbf{x_4}$\end{tabular}}}}%
  \end{picture}%
\endgroup%

	\caption{(left) Configurations of chord arrows near a double crossing, clearly both arrows cannot be $\varepsilon$--close to  $N$. (right) Configurations of arrows near a double crossin $\varepsilon$--close to 
		$(N,N,N)$, clearly $K(x_3)$ cannot follow $K(x_1)$.}
	\label{fig:eta-support-double}
\end{figure}

 To justify $(iii)$ of Theorem \ref{thm:c2-integral}, consider a local regular knot isotopy 	$K:\R\times [0,1] \longrightarrow \R^3$, $(\,\cdot\,,z)\longrightarrow K(\,\cdot\,,z)\in \K$ realizing the {\bf RI} move in the projection onto the plane orthogonal to $N$. 
 Such isotopy is not in $\K_N$ but in $\K$, however the integral \eqref{eq:casson-integral-eta} can still be evaluated along such isotopy. First, the crossing in {\bf RI} move does not contribute to the value of $I(K)$ as already noted in Remark \ref{rem:single-cross-contribution}.
 Next, observe that the first term of \eqref{eq:casson-integral-eta} can only be nonzero for $x_1<x_2<x_3<x_4$ and $z$ satisfying
\begin{equation}\label{eq:RI-c2-requirement}
\tfrac{K(x_3,z)-K(x_1,z)}{|K(x_3,z)-K(x_1,z)|}\approx_\varepsilon N,\qquad \tfrac{K(x_2,z)-K(x_4,z)}{|K(x_2,z)-K(x_4,z)|}\approx_\varepsilon N.
\end{equation}
It is clear that values of $(x_1,x_2,x_3,x_4,z)$ associated to the pairing of the crossing in {\bf RI} move with any other crossing cannot satisfy both of these approximate identities together with the inequalities $x_1<x_2<x_3<x_4$. Indeed, the crossing in {\bf RI} move implies $\tfrac{K(x_i,z)-K(x_{i+1},z)}{|K(x_i,z)-K(x_{i+1},z)|}\approx_\varepsilon N$ or 
$\tfrac{K(x_{i+1},z)-K(x_{i},z)}{|K(x_{i+1},z)-K(x_{i},z)|}\approx_\varepsilon N$,
for $x_i<x_{i+1}$ which is inconsistent with  \eqref{eq:RI-c2-requirement}.
We conclude that the move does not contribute to the value of $I(K)$. \qed

%%%%%%%%%%%%%%%%%%%%%%%%%%%%%%%%%%%%%%%%%%%%%%%%%%%%%%%%%%%%%%%%%
\section{Proof of Theorem \ref{thm:c2-multicrossing}}\label{sec:arrow-counting-mltpl-crossing}

The remaining part is an extension of the formula \eqref{eq:I(K)-gauss-dbl-crossing} to knot diagrams with multiple crossings given in \eqref{eq:gd_c2_mult-crossings}:
\begin{equation}\label{eq:gd_c2_mult-crossings-2}
	\begin{split}
		c_2(K) & =\langle \vvcenteredinclude{.3}{gd_c2_chord.pdf},G_K\rangle+ \tfrac{1}{2}\bigl(\langle \vvcenteredinclude{.3}{gd_c2_chord-left.pdf},G_K\rangle\\
		&\qquad +\langle \vvcenteredinclude{.3}{gd_c2_chord-mid.pdf},G_K\rangle+\langle \vvcenteredinclude{.3}{gd_c2_chord-right.pdf},G_K\rangle\bigr).
	\end{split}
\end{equation}

	\begin{rem}
The second term in the above formula counts the contribution of the {\em tripod integral}:\\ $\varint\limits_{\Cnf[\R,\K_N;3,1]} h^\ast_{1\mathbf{4}}\eta_N\wedge h^\ast_{\mathbf{4}2}\eta_N\wedge h^\ast_{3\mathbf{4}}\eta_N$. Heuristically, the tripod integral counts configurations 
	$(x_1,x_2,x_3,\mathbf{x}_4)$ near triple crossings satisfying \eqref{eq:tripod-support-condition} (pictured\footnote{where the arrows point approximately in the direction of $N$.} in Figure \ref{fig:eta-support-triple}). Rather than counting such tripods, we can count corresponding pairs of crossings obtained by collapsing the arrow $\overrightarrow{\mathbf{x}_4\, K(x_{3\,\text{or}\,1})}$ or $\overrightarrow{K(x_2)\,\mathbf{x}_4}$. Collapsing either of these arrows results in exactly two of the three diagrams given in the second term; putting these together, in terms of the arrow diagrams, this yields the expression $\tfrac{1}{2}\bigl(\langle \vvcenteredinclude{.3}{gd_c2_chord-left.pdf},G_K\rangle$
	 $+\langle \vvcenteredinclude{.3}{gd_c2_chord-mid.pdf},G_K\rangle$ $+\langle \vvcenteredinclude{.3}{gd_c2_chord-right.pdf},G_K\rangle\bigr)$ and explains  $\tfrac{1}{2}$ factor in the formula.
	%%%%%%%%%%%%%%%%%%%%%%%%%%%%%%%%%%%%%%%%%
	\begin{figure}[h!]
		\centering
		\includegraphics[width=0.25\textwidth]{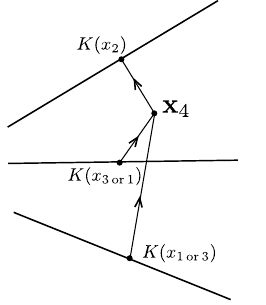}
		\caption{Configuration of points near the triple crossing in $\Cnf[\K_N,\R^3;3,1]$, whose image under $h'$ is $\varepsilon$--close to 
			$(N,N,N)$.}
		\label{fig:eta-support-triple}
	\end{figure}	
Since the exact evaluation of the tripod integral appears too technical, in the following, we provide a combinatorial argument justifying \eqref{eq:gd_c2_mult-crossings-2}. 
\end{rem} 

%%%%%%%%%%%%%%%%%%%%%%%%%%%%%%%%%%%%%%%%%%
\begin{figure}[h!]
	\centering
	\includegraphics[width=\textwidth]{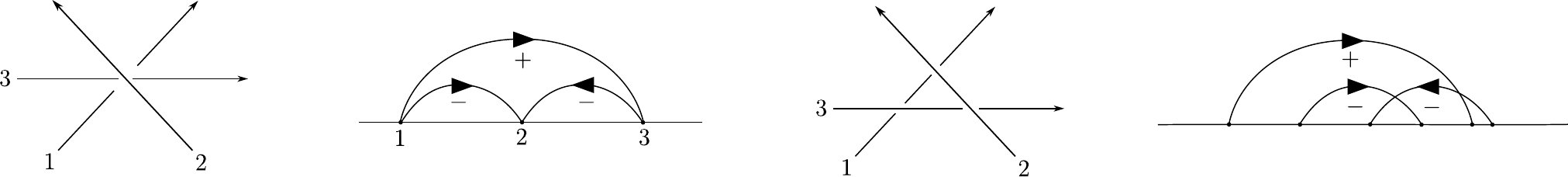}
	\caption{A triple crossing perturbation and corresponding fragments of the Gauss diagram.}
	\label{fig:triple-cross-perturbation}
\end{figure}
%%%%%%%%%%%%%%%%%%%%%%%%%%%%%%%%%%%%%%%%%%
First, for simplicity we think of a long knot $K$ and its planar diagram as the same object (i.e. the parametrization of $K$ is almost entirely in the plane of the projection except for small overpasses). If $K$ is a (transverse) multi-crossing knot, we denote by $K'$ its local perturbation. Specifically, $K'$ has a regular double crossing diagram which agrees with $K$ outside small balls containing the multiple crossings of $K$, and is isotopic to $K$, Figure \ref{fig:triple-cross-perturbation}.
For that reason, we trivially observe that 
\begin{equation}\label{eq:c_2-via-K'}
 c_2(K)=\langle \vvcenteredinclude{.3}{gd_c2_chord.pdf},G_{K'}\rangle=\sum_{\phi':\vvcenteredinclude{.15}{gd_c2_chord.pdf}\longrightarrow G_{K'}} \operatorname{sign}(\phi').
\end{equation}
\no Clearly, the goal is to replace the sum over embeddings  
\begin{equation}\label{eq:embeddings_in_G_K'}
	\mathcal{E}(G_{K'}) =\{\phi':\vvcenteredinclude{.2}{gd_c2_chord.pdf}\longrightarrow G_{K'}\}
\end{equation} 
with a sum over embeddings 
into $G_K$. Denote by $e(G_K)$ and $e(G_{K'})$ the set of oriented and signed chords of $G_K$ and $G'_K$ respectively, each embedding $\phi'\in \mathcal{E}(G_{K'})$ can be identified with a choice of the pair of distinct chords from $e(G_{K'})$, i.e. $\phi'=\{\alpha',\beta'\}$, $\alpha',\beta'\in e(G_{K'})$, satisfying the obvious condition on their endpoints. Therefore, the sum in \eqref{eq:c_2-via-K'} can be written as a sum over unordered pairs of chords 
\[
\mathcal{E}^2(G_{K'})  =\{\{\alpha',\beta'\}\ |\ \alpha',\beta'\in e(G_{K'}), \alpha'\neq\beta'\},\\
\]
as follows 
\begin{equation}\label{eq:c_2-via-K'-chords}
\langle \vvcenteredinclude{.3}{gd_c2_chord.pdf},G_{K'}\rangle=\sum_{\{\alpha',\beta'\}\in \mathcal{E}^2(G_{K'})} f_{\vvcenteredinclude{.1}{gd_c2_chord.pdf}}(\alpha',\beta'),
\end{equation}
where $f_{\vvcenteredinclude{.1}{gd_c2_chord.pdf}}$ is an obvious function (valued in $\{0,\pm 1\}$).
Next, observe that, by construction of $K'$, we have a bijection 
\[
 F: e(G_{K'})\longrightarrow e(G_K),
\]
which preserves the orientation of chords and their signs. Further, for each $\phi'=\{\alpha',\beta'\}\in \mathcal{E}(G_{K'})$, we will use notation 
\[
\phi=\{\alpha,\beta\}=F\circ \phi'=\{F(\alpha'),F(\beta')\},
\]
 under the above identification of pairs of chords and embeddings. 

 There are two cases of interest for such $\phi$: first $\alpha$ and $\beta$ do not share any endpoints then $\phi$ embeds $\vvcenteredinclude{.2}{gd_c2_chord.pdf}$ into $G_K$ (i.e. $\phi$ is a {\em regular} embedding), second $\alpha$ and $\beta$ share an endpoint, then $\phi$ embeds one of  
\begin{equation}\label{eq:diag-common-ends}
\vvcenteredinclude{.3}{gd_c2_chord-left.pdf},\quad \vvcenteredinclude{.3}{gd_c2_chord-mid.pdf},\quad
\vvcenteredinclude{.3}{gd_c2_chord-right.pdf}
\end{equation}
 into $G_K$. In the view, of the above correspondence between the embeddings and the pairs of chords, let us define the following subsets of pairs of chords partitioning $e(G_{K'})^2$:
 \begin{equation}\label{eq:E^2-bullet}
 \begin{split}
  \mathcal{E}^2_\bullet(G_{K'}) & =\{\{\alpha',\beta'\}\in \mathcal{E}^2(G_{K'})\ |\ \alpha\ \text{and}\ \beta\ \text{share an endpoint}\},\\
  \mathcal{E}^2_\circ(G_{K'}) & = \{\{\alpha',\beta'\}\in \mathcal{E}^2(G_{K'})\ |\ \alpha\ \text{and}\ \beta\ \text{do not share an endpoint}\},
 \end{split}
 \end{equation}
 the sum in \eqref{eq:c_2-via-K'-chords} splits accordingly as 
 \begin{equation}\label{eq:c_2-via-K'-chords-2}
 	\langle \vvcenteredinclude{.3}{gd_c2_chord.pdf},G_{K'}\rangle=\sum_{\{\alpha',\beta'\}\in \mathcal{E}^2_\circ(G_{K'})} f_{\vvcenteredinclude{.1}{gd_c2_chord.pdf}}(\alpha',\beta')+\sum_{\{\alpha',\beta'\}\in \mathcal{E}^2_\bullet(G_{K'})} f_{\vvcenteredinclude{.1}{gd_c2_chord.pdf}}(\alpha',\beta').
 \end{equation}
 Note that the first term in this sum equals (by definition)
 \begin{equation}\label{eq:sum-E^2-circ}
 \sum_{\{\alpha',\beta'\}\in \mathcal{E}^2_\circ(G_{K'})} f_{\vvcenteredinclude{.1}{gd_c2_chord.pdf}}(\alpha',\beta')=\langle \vvcenteredinclude{.3}{gd_c2_chord.pdf},G_{K}\rangle.
 \end{equation}
 \begin{rem} 
 	The expression $\langle \vvcenteredinclude{.25}{gd_c2_chord.pdf}+  \vvcenteredinclude{.25}{gd_c2_chord-left.pdf}+ \vvcenteredinclude{.25}{gd_c2_chord-mid.pdf}+\vvcenteredinclude{.25}{gd_c2_chord-right.pdf},G_K\rangle$ does not yield a correct count for $c_2(K)$, 
 because there is no one--to--one correspondence between the embeddings of  diagrams with common endpoints listed in \eqref{eq:diag-common-ends} in $G_K$, with a subset of $\mathcal{E}(G_{K'})$ in \eqref{eq:embeddings_in_G_K'}, preserving signs and orientations, see Example \eqref{eq:trefoil-computation}. 
 \end{rem}
 \begin{table}[ht]
	\caption{Verifying \eqref{eq:triple-cross-count} on triple--crossings.}
	\begin{tabular}{|c|c|c|c|}
		
		\hline
		Triple Crossing & Arrow diagrams contributing to \eqref{eq:triple-cross-count} & Perturbed Crossing & \eqref{eq:c_2-via-K'} \\ \hline
		
		\scalebox{0.3}{\includegraphics{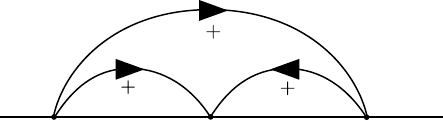}} & \scalebox{0.3}{%% Creator: Inkscape 1.1-dev (48d18c81, 2020-07-08), www.inkscape.org
%% Accompanies image file 'gd_c2_chord-mid-++.pdf' (pdf, eps, ps)
\begingroup%
  \makeatletter%
  \providecommand\color[2][]{%
    \errmessage{(Inkscape) Color is used for the text in Inkscape, but the package 'color.sty' is not loaded}%
    \renewcommand\color[2][]{}%
  }%
  \providecommand\transparent[1]{%
    \errmessage{(Inkscape) Transparency is used (non-zero) for the text in Inkscape, but the package 'transparent.sty' is not loaded}%
    \renewcommand\transparent[1]{}%
  }%
  \providecommand\rotatebox[2]{#2}%
  \newcommand*\fsize{\dimexpr\f@size pt\relax}%
  \newcommand*\lineheight[1]{\fontsize{\fsize}{#1\fsize}\selectfont}%
  \ifx\svgwidth\undefined%
    \setlength{\unitlength}{212.5984252bp}%
    \ifx\svgscale\undefined%
      \relax%
    \else%
      \setlength{\unitlength}{\unitlength * \real{\svgscale}}%
    \fi%
  \else%
    \setlength{\unitlength}{\svgwidth}%
  \fi%
  \global\let\svgwidth\undefined%
  \global\let\svgscale\undefined%
  \makeatother%
  \begin{picture}(1,0.27697906)%
    \lineheight{1}%
    \setlength\tabcolsep{0pt}%
    \put(-0.25958258,-0.28324667){\color[rgb]{0,0,0}\makebox(0,0)[lt]{\begin{minipage}{0.06592648\unitlength}\raggedright \end{minipage}}}%
    \put(0.09102822,0.06161557){\color[rgb]{0,0,0}\makebox(0,0)[lt]{\begin{minipage}{0.11994444\unitlength}\end{minipage}}}%
    \put(0.23035288,0.37492507){\color[rgb]{0,0,0}\makebox(0,0)[lt]{\begin{minipage}{0.21085618\unitlength}\end{minipage}}}%
    \put(0,0){\includegraphics[width=\unitlength,page=1]{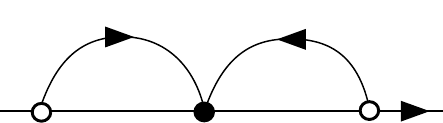}}%
    \put(0.23539374,0.11297076){\color[rgb]{0,0,0}\makebox(0,0)[lt]{\lineheight{1.25}\smash{\begin{tabular}[t]{l}$+$\end{tabular}}}}%
    \put(0.63912726,0.11297076){\color[rgb]{0,0,0}\makebox(0,0)[lt]{\lineheight{1.25}\smash{\begin{tabular}[t]{l}$+$\end{tabular}}}}%
  \end{picture}%
\endgroup%
} $+$ \scalebox{0.3}{%% Creator: Inkscape 1.1-dev (48d18c81, 2020-07-08), www.inkscape.org
%% Accompanies image file 'gd_c2_chord-left-++.pdf' (pdf, eps, ps)
\begingroup%
  \makeatletter%
  \providecommand\color[2][]{%
    \errmessage{(Inkscape) Color is used for the text in Inkscape, but the package 'color.sty' is not loaded}%
    \renewcommand\color[2][]{}%
  }%
  \providecommand\transparent[1]{%
    \errmessage{(Inkscape) Transparency is used (non-zero) for the text in Inkscape, but the package 'transparent.sty' is not loaded}%
    \renewcommand\transparent[1]{}%
  }%
  \providecommand\rotatebox[2]{#2}%
  \newcommand*\fsize{\dimexpr\f@size pt\relax}%
  \newcommand*\lineheight[1]{\fontsize{\fsize}{#1\fsize}\selectfont}%
  \ifx\svgwidth\undefined%
    \setlength{\unitlength}{212.5984252bp}%
    \ifx\svgscale\undefined%
      \relax%
    \else%
      \setlength{\unitlength}{\unitlength * \real{\svgscale}}%
    \fi%
  \else%
    \setlength{\unitlength}{\svgwidth}%
  \fi%
  \global\let\svgwidth\undefined%
  \global\let\svgscale\undefined%
  \makeatother%
  \begin{picture}(1,0.27697906)%
    \lineheight{1}%
    \setlength\tabcolsep{0pt}%
    \put(0,0){\includegraphics[width=\unitlength,page=1]{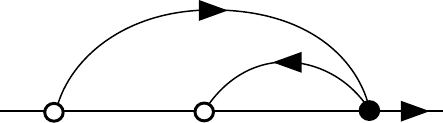}}%
    \put(-0.25958258,-0.28324667){\color[rgb]{0,0,0}\makebox(0,0)[lt]{\begin{minipage}{0.06592648\unitlength}\raggedright \end{minipage}}}%
    \put(0.45413611,0.17226906){\color[rgb]{0,0,0}\makebox(0,0)[lt]{\lineheight{1.25}\smash{\begin{tabular}[t]{l}$+$\end{tabular}}}}%
    \put(0.6234689,0.05938169){\color[rgb]{0,0,0}\makebox(0,0)[lt]{\lineheight{1.25}\smash{\begin{tabular}[t]{l}$+$\end{tabular}}}}%
  \end{picture}%
\endgroup%
} $\rightsquigarrow 2$ & \scalebox{0.3}{\includegraphics{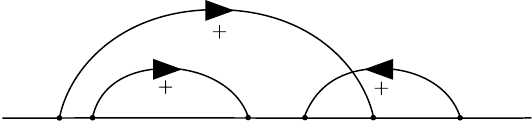}}  \textrm{or} \scalebox{0.3}{\includegraphics{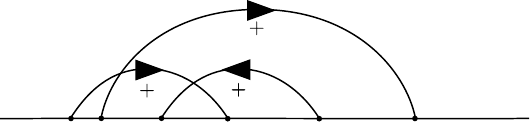}}  &   1  \\ 
		
		\scalebox{0.3}{\includegraphics{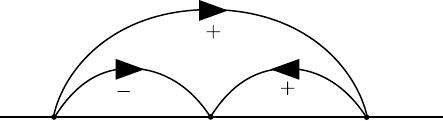}} & \scalebox{0.3}{%% Creator: Inkscape 1.1-dev (48d18c81, 2020-07-08), www.inkscape.org
%% Accompanies image file 'gd_c2_chord-mid--+.pdf' (pdf, eps, ps)
\begingroup%
  \makeatletter%
  \providecommand\color[2][]{%
    \errmessage{(Inkscape) Color is used for the text in Inkscape, but the package 'color.sty' is not loaded}%
    \renewcommand\color[2][]{}%
  }%
  \providecommand\transparent[1]{%
    \errmessage{(Inkscape) Transparency is used (non-zero) for the text in Inkscape, but the package 'transparent.sty' is not loaded}%
    \renewcommand\transparent[1]{}%
  }%
  \providecommand\rotatebox[2]{#2}%
  \newcommand*\fsize{\dimexpr\f@size pt\relax}%
  \newcommand*\lineheight[1]{\fontsize{\fsize}{#1\fsize}\selectfont}%
  \ifx\svgwidth\undefined%
    \setlength{\unitlength}{212.5984252bp}%
    \ifx\svgscale\undefined%
      \relax%
    \else%
      \setlength{\unitlength}{\unitlength * \real{\svgscale}}%
    \fi%
  \else%
    \setlength{\unitlength}{\svgwidth}%
  \fi%
  \global\let\svgwidth\undefined%
  \global\let\svgscale\undefined%
  \makeatother%
  \begin{picture}(1,0.27697906)%
    \lineheight{1}%
    \setlength\tabcolsep{0pt}%
    \put(-0.25958258,-0.28324667){\color[rgb]{0,0,0}\makebox(0,0)[lt]{\begin{minipage}{0.06592648\unitlength}\raggedright \end{minipage}}}%
    \put(0.09102822,0.06161557){\color[rgb]{0,0,0}\makebox(0,0)[lt]{\begin{minipage}{0.11994444\unitlength}\end{minipage}}}%
    \put(0.23035288,0.37492507){\color[rgb]{0,0,0}\makebox(0,0)[lt]{\begin{minipage}{0.21085618\unitlength}\end{minipage}}}%
    \put(0,0){\includegraphics[width=\unitlength,page=1]{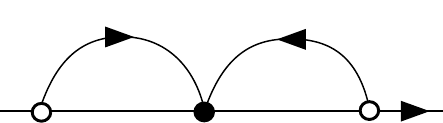}}%
    \put(0.23539374,0.11297076){\color[rgb]{0,0,0}\makebox(0,0)[lt]{\lineheight{1.25}\smash{\begin{tabular}[t]{l}$-$\end{tabular}}}}%
    \put(0.63912726,0.11297076){\color[rgb]{0,0,0}\makebox(0,0)[lt]{\lineheight{1.25}\smash{\begin{tabular}[t]{l}$+$\end{tabular}}}}%
  \end{picture}%
\endgroup%
} $+$ \scalebox{0.3}{ } $\rightsquigarrow 0$ & \scalebox{0.3}{\includegraphics{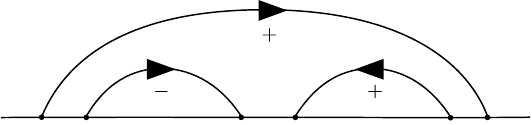}}  \textrm{or} \scalebox{0.3}{\includegraphics{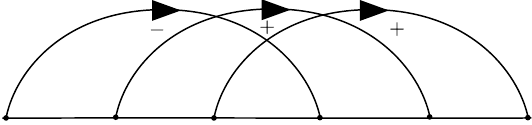}} &  0 \\ 
		
		\scalebox{0.3}{\includegraphics{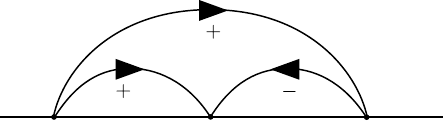}} & \scalebox{0.3}{%% Creator: Inkscape 1.1-dev (48d18c81, 2020-07-08), www.inkscape.org
%% Accompanies image file 'gd_c2_chord-mid-+-.pdf' (pdf, eps, ps)
\begingroup%
  \makeatletter%
  \providecommand\color[2][]{%
    \errmessage{(Inkscape) Color is used for the text in Inkscape, but the package 'color.sty' is not loaded}%
    \renewcommand\color[2][]{}%
  }%
  \providecommand\transparent[1]{%
    \errmessage{(Inkscape) Transparency is used (non-zero) for the text in Inkscape, but the package 'transparent.sty' is not loaded}%
    \renewcommand\transparent[1]{}%
  }%
  \providecommand\rotatebox[2]{#2}%
  \newcommand*\fsize{\dimexpr\f@size pt\relax}%
  \newcommand*\lineheight[1]{\fontsize{\fsize}{#1\fsize}\selectfont}%
  \ifx\svgwidth\undefined%
    \setlength{\unitlength}{212.5984252bp}%
    \ifx\svgscale\undefined%
      \relax%
    \else%
      \setlength{\unitlength}{\unitlength * \real{\svgscale}}%
    \fi%
  \else%
    \setlength{\unitlength}{\svgwidth}%
  \fi%
  \global\let\svgwidth\undefined%
  \global\let\svgscale\undefined%
  \makeatother%
  \begin{picture}(1,0.27697906)%
    \lineheight{1}%
    \setlength\tabcolsep{0pt}%
    \put(-0.25958258,-0.28324667){\color[rgb]{0,0,0}\makebox(0,0)[lt]{\begin{minipage}{0.06592648\unitlength}\raggedright \end{minipage}}}%
    \put(0.09102822,0.06161557){\color[rgb]{0,0,0}\makebox(0,0)[lt]{\begin{minipage}{0.11994444\unitlength}\end{minipage}}}%
    \put(0.23035288,0.37492507){\color[rgb]{0,0,0}\makebox(0,0)[lt]{\begin{minipage}{0.21085618\unitlength}\end{minipage}}}%
    \put(0,0){\includegraphics[width=\unitlength,page=1]{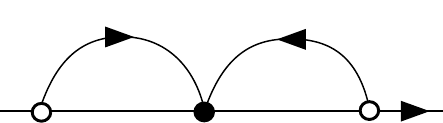}}%
    \put(0.23539374,0.11297076){\color[rgb]{0,0,0}\makebox(0,0)[lt]{\lineheight{1.25}\smash{\begin{tabular}[t]{l}$+$\end{tabular}}}}%
    \put(0.63912726,0.11297076){\color[rgb]{0,0,0}\makebox(0,0)[lt]{\lineheight{1.25}\smash{\begin{tabular}[t]{l}$-$\end{tabular}}}}%
  \end{picture}%
\endgroup%
} $+$ \scalebox{0.3}{%% Creator: Inkscape 1.1-dev (48d18c81, 2020-07-08), www.inkscape.org
%% Accompanies image file 'gd_c2_chord-left-+-.pdf' (pdf, eps, ps)
\begingroup%
  \makeatletter%
  \providecommand\color[2][]{%
    \errmessage{(Inkscape) Color is used for the text in Inkscape, but the package 'color.sty' is not loaded}%
    \renewcommand\color[2][]{}%
  }%
  \providecommand\transparent[1]{%
    \errmessage{(Inkscape) Transparency is used (non-zero) for the text in Inkscape, but the package 'transparent.sty' is not loaded}%
    \renewcommand\transparent[1]{}%
  }%
  \providecommand\rotatebox[2]{#2}%
  \newcommand*\fsize{\dimexpr\f@size pt\relax}%
  \newcommand*\lineheight[1]{\fontsize{\fsize}{#1\fsize}\selectfont}%
  \ifx\svgwidth\undefined%
    \setlength{\unitlength}{212.5984252bp}%
    \ifx\svgscale\undefined%
      \relax%
    \else%
      \setlength{\unitlength}{\unitlength * \real{\svgscale}}%
    \fi%
  \else%
    \setlength{\unitlength}{\svgwidth}%
  \fi%
  \global\let\svgwidth\undefined%
  \global\let\svgscale\undefined%
  \makeatother%
  \begin{picture}(1,0.27697906)%
    \lineheight{1}%
    \setlength\tabcolsep{0pt}%
    \put(0,0){\includegraphics[width=\unitlength,page=1]{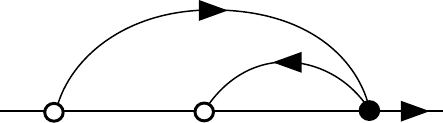}}%
    \put(-0.25958258,-0.28324667){\color[rgb]{0,0,0}\makebox(0,0)[lt]{\begin{minipage}{0.06592648\unitlength}\raggedright \end{minipage}}}%
    \put(0.45413611,0.17226906){\color[rgb]{0,0,0}\makebox(0,0)[lt]{\lineheight{1.25}\smash{\begin{tabular}[t]{l}$+$\end{tabular}}}}%
    \put(0.6234689,0.05938169){\color[rgb]{0,0,0}\makebox(0,0)[lt]{\lineheight{1.25}\smash{\begin{tabular}[t]{l}$-$\end{tabular}}}}%
  \end{picture}%
\endgroup%
 } $\rightsquigarrow -2$ & \scalebox{0.3}{\includegraphics{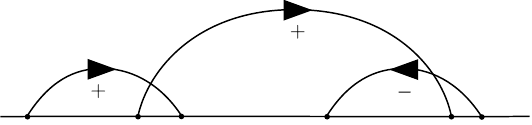}}  \textrm{or}  \scalebox{0.3}{\includegraphics{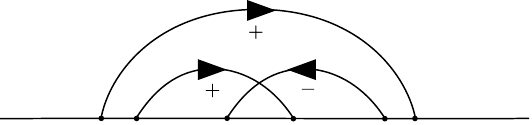}} &  -1 \\ 
		
		\scalebox{0.3}{\includegraphics{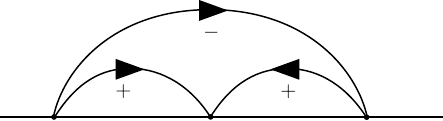}} & \scalebox{0.3}{} $+$ \scalebox{0.3}{%% Creator: Inkscape 1.1-dev (48d18c81, 2020-07-08), www.inkscape.org
%% Accompanies image file 'gd_c2_chord-left--+.pdf' (pdf, eps, ps)
\begingroup%
  \makeatletter%
  \providecommand\color[2][]{%
    \errmessage{(Inkscape) Color is used for the text in Inkscape, but the package 'color.sty' is not loaded}%
    \renewcommand\color[2][]{}%
  }%
  \providecommand\transparent[1]{%
    \errmessage{(Inkscape) Transparency is used (non-zero) for the text in Inkscape, but the package 'transparent.sty' is not loaded}%
    \renewcommand\transparent[1]{}%
  }%
  \providecommand\rotatebox[2]{#2}%
  \newcommand*\fsize{\dimexpr\f@size pt\relax}%
  \newcommand*\lineheight[1]{\fontsize{\fsize}{#1\fsize}\selectfont}%
  \ifx\svgwidth\undefined%
    \setlength{\unitlength}{212.5984252bp}%
    \ifx\svgscale\undefined%
      \relax%
    \else%
      \setlength{\unitlength}{\unitlength * \real{\svgscale}}%
    \fi%
  \else%
    \setlength{\unitlength}{\svgwidth}%
  \fi%
  \global\let\svgwidth\undefined%
  \global\let\svgscale\undefined%
  \makeatother%
  \begin{picture}(1,0.27697906)%
    \lineheight{1}%
    \setlength\tabcolsep{0pt}%
    \put(0,0){\includegraphics[width=\unitlength,page=1]{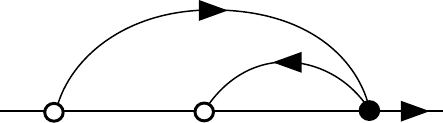}}%
    \put(-0.25958258,-0.28324667){\color[rgb]{0,0,0}\makebox(0,0)[lt]{\begin{minipage}{0.06592648\unitlength}\raggedright \end{minipage}}}%
    \put(0.45413611,0.17226906){\color[rgb]{0,0,0}\makebox(0,0)[lt]{\lineheight{1.25}\smash{\begin{tabular}[t]{l}$-$\end{tabular}}}}%
    \put(0.6234689,0.05938169){\color[rgb]{0,0,0}\makebox(0,0)[lt]{\lineheight{1.25}\smash{\begin{tabular}[t]{l}$+$\end{tabular}}}}%
  \end{picture}%
\endgroup%
 } $\rightsquigarrow 0$ & \scalebox{0.3}{\includegraphics{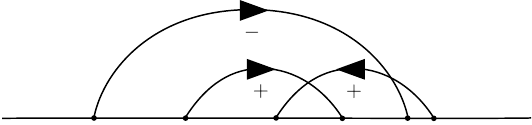}}  \textrm{or}  \scalebox{0.3}{\includegraphics{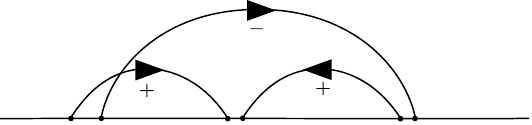}} &0 \\ 
		
		\scalebox{0.3}{\includegraphics{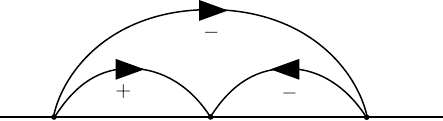}} & \scalebox{0.3}{} $+$ \scalebox{0.3}{%% Creator: Inkscape 1.1-dev (48d18c81, 2020-07-08), www.inkscape.org
%% Accompanies image file 'gd_c2_chord-left---.pdf' (pdf, eps, ps)
\begingroup%
  \makeatletter%
  \providecommand\color[2][]{%
    \errmessage{(Inkscape) Color is used for the text in Inkscape, but the package 'color.sty' is not loaded}%
    \renewcommand\color[2][]{}%
  }%
  \providecommand\transparent[1]{%
    \errmessage{(Inkscape) Transparency is used (non-zero) for the text in Inkscape, but the package 'transparent.sty' is not loaded}%
    \renewcommand\transparent[1]{}%
  }%
  \providecommand\rotatebox[2]{#2}%
  \newcommand*\fsize{\dimexpr\f@size pt\relax}%
  \newcommand*\lineheight[1]{\fontsize{\fsize}{#1\fsize}\selectfont}%
  \ifx\svgwidth\undefined%
    \setlength{\unitlength}{212.5984252bp}%
    \ifx\svgscale\undefined%
      \relax%
    \else%
      \setlength{\unitlength}{\unitlength * \real{\svgscale}}%
    \fi%
  \else%
    \setlength{\unitlength}{\svgwidth}%
  \fi%
  \global\let\svgwidth\undefined%
  \global\let\svgscale\undefined%
  \makeatother%
  \begin{picture}(1,0.27697906)%
    \lineheight{1}%
    \setlength\tabcolsep{0pt}%
    \put(0,0){\includegraphics[width=\unitlength,page=1]{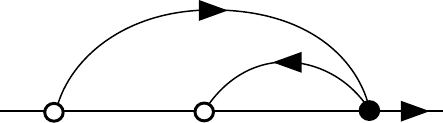}}%
    \put(-0.25958258,-0.28324667){\color[rgb]{0,0,0}\makebox(0,0)[lt]{\begin{minipage}{0.06592648\unitlength}\raggedright \end{minipage}}}%
    \put(0.45413611,0.17226906){\color[rgb]{0,0,0}\makebox(0,0)[lt]{\lineheight{1.25}\smash{\begin{tabular}[t]{l}$-$\end{tabular}}}}%
    \put(0.6234689,0.05938169){\color[rgb]{0,0,0}\makebox(0,0)[lt]{\lineheight{1.25}\smash{\begin{tabular}[t]{l}$-$\end{tabular}}}}%
  \end{picture}%
\endgroup%
 } $\rightsquigarrow 0$ & \scalebox{0.3}{\includegraphics{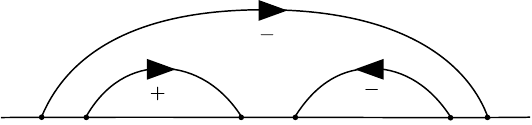}}  \textrm{or} \scalebox{0.3}{\includegraphics{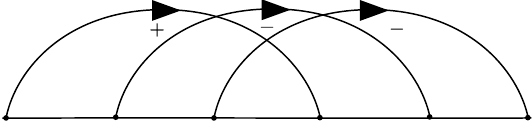}} &  0 \\ 
		
		\scalebox{0.3}{\includegraphics{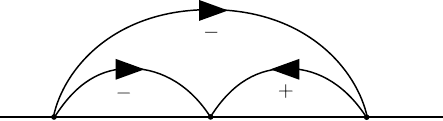}} & \scalebox{0.3}{} $+$ \scalebox{0.3}{ } $\rightsquigarrow -2$ & \scalebox{0.3}{\includegraphics{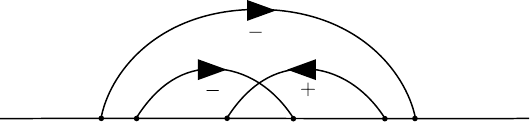}}  \textrm{or} \scalebox{0.3}{\includegraphics{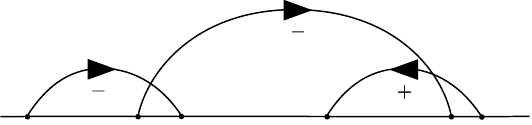}} & -1 \\ 
		
		\scalebox{0.3}{\includegraphics{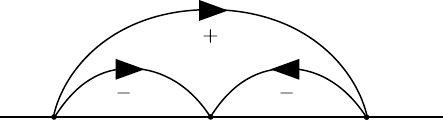}} & \scalebox{0.3}{%% Creator: Inkscape 1.1-dev (48d18c81, 2020-07-08), www.inkscape.org
%% Accompanies image file 'gd_c2_chord-mid---.pdf' (pdf, eps, ps)
\begingroup%
  \makeatletter%
  \providecommand\color[2][]{%
    \errmessage{(Inkscape) Color is used for the text in Inkscape, but the package 'color.sty' is not loaded}%
    \renewcommand\color[2][]{}%
  }%
  \providecommand\transparent[1]{%
    \errmessage{(Inkscape) Transparency is used (non-zero) for the text in Inkscape, but the package 'transparent.sty' is not loaded}%
    \renewcommand\transparent[1]{}%
  }%
  \providecommand\rotatebox[2]{#2}%
  \newcommand*\fsize{\dimexpr\f@size pt\relax}%
  \newcommand*\lineheight[1]{\fontsize{\fsize}{#1\fsize}\selectfont}%
  \ifx\svgwidth\undefined%
    \setlength{\unitlength}{212.5984252bp}%
    \ifx\svgscale\undefined%
      \relax%
    \else%
      \setlength{\unitlength}{\unitlength * \real{\svgscale}}%
    \fi%
  \else%
    \setlength{\unitlength}{\svgwidth}%
  \fi%
  \global\let\svgwidth\undefined%
  \global\let\svgscale\undefined%
  \makeatother%
  \begin{picture}(1,0.27697906)%
    \lineheight{1}%
    \setlength\tabcolsep{0pt}%
    \put(-0.25958258,-0.28324667){\color[rgb]{0,0,0}\makebox(0,0)[lt]{\begin{minipage}{0.06592648\unitlength}\raggedright \end{minipage}}}%
    \put(0.09102822,0.06161557){\color[rgb]{0,0,0}\makebox(0,0)[lt]{\begin{minipage}{0.11994444\unitlength}\end{minipage}}}%
    \put(0.23035288,0.37492507){\color[rgb]{0,0,0}\makebox(0,0)[lt]{\begin{minipage}{0.21085618\unitlength}\end{minipage}}}%
    \put(0,0){\includegraphics[width=\unitlength,page=1]{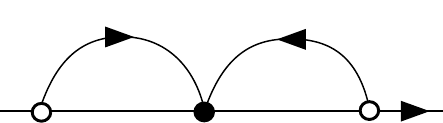}}%
    \put(0.23539374,0.11297076){\color[rgb]{0,0,0}\makebox(0,0)[lt]{\lineheight{1.25}\smash{\begin{tabular}[t]{l}$-$\end{tabular}}}}%
    \put(0.63912726,0.11297076){\color[rgb]{0,0,0}\makebox(0,0)[lt]{\lineheight{1.25}\smash{\begin{tabular}[t]{l}$-$\end{tabular}}}}%
  \end{picture}%
\endgroup%
} $+$ \scalebox{0.3}{ } $\rightsquigarrow 0$ & \scalebox{0.3}{\includegraphics{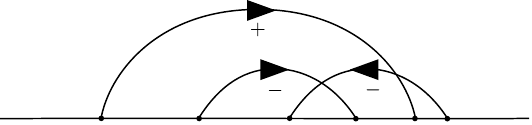}}  \textrm{or}  \scalebox{0.3}{\includegraphics{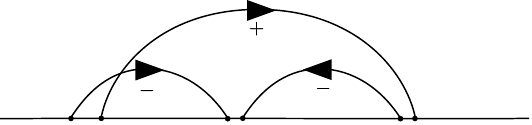}} & 0 \\ 
		
		\scalebox{0.3}{\includegraphics{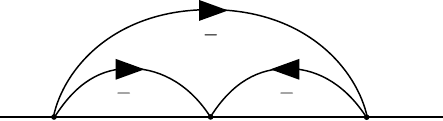}} & \scalebox{0.3}{} $+$ \scalebox{0.3}{ } $\rightsquigarrow 2$ & \scalebox{0.3}{\includegraphics{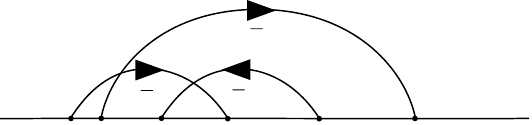}}  \textrm{or}  \scalebox{0.3}{\includegraphics{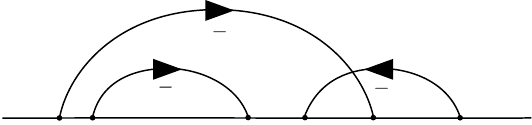}} & 1  \\ \hline
	\end{tabular}\label{tb:triplecrossing} 	 
	
\end{table}
 In order to find an expression for the second term in \eqref{eq:c_2-via-K'-chords-2}, note that for every triple of distinct chords $\alpha',\beta',\gamma'\in e(G_{K'})$, either $\{\alpha,\beta,\gamma\}=F\{\alpha',\beta',\gamma'\}$ form a triple crossing\footnote{where the labels $\alpha$, $\beta$, $\gamma$ can be permuted and arrow directions with signs assigned freely.} $T=\vvcenteredinclude{.3}{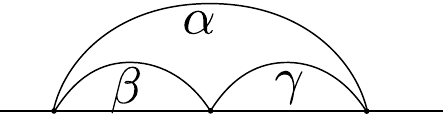}$ in $G_K$ or not. For each triple crossing the set of pairs $\{\alpha,\beta\}$, $\{\alpha,\gamma\}$, $\{\gamma,\beta\}$ uniquely determine the triple--crossings $T$, i.e. none of the pairs can be in a different triple--crossing. Also every pair $\{\alpha,\beta\}$ in $\mathcal{E}^2_\bullet(G_{K'})$ is a part of a unique triple--crossing. As a result the set $\mathcal{E}^2_\bullet(G_{K'})$ partitions into three elements subsets $\{\alpha',\beta'\}$, $\{\alpha',\gamma'\}$, $\{\gamma',\beta'\}$ which are pairs such that the corresponding $\{\alpha,\beta\}=F\{\alpha',\beta'\}$, $\{\alpha,\gamma\}=F\{\alpha',\gamma'\}$, $\{\gamma,\beta\}=F\{\gamma',\beta'\}$ constitute a triple--crossing. Denoting the set of ``triple crossings'' by $\mathcal{T}(G_{K'})$, we obtain
 \begin{equation}\label{eq:partitioned-sum}
 \sum_{\{\alpha',\beta'\}\in \mathcal{E}^2_\bullet(G_{K'})} f_{\vvcenteredinclude{.1}{gd_c2_chord.pdf}}(\alpha',\beta')= \sum_{\{\alpha',\beta',\gamma'\}\in \mathcal{T}(G_{K'})} \bigl(f_{\vvcenteredinclude{.1}{gd_c2_chord.pdf}}(\alpha',\beta')+f_{\vvcenteredinclude{.1}{gd_c2_chord.pdf}}(\alpha',\gamma')+f_{\vvcenteredinclude{.1}{gd_c2_chord.pdf}}(\gamma',\beta')\bigr).
 \end{equation}
 Table \ref{tb:triplecrossing} shows that for each $\{\alpha',\beta',\gamma'\}\in \mathcal{T}(G_{K'})$ the corresponding term in the above sum equals:
 \begin{equation}\label{eq:triple-cross-count}
 \begin{split}
   & f_{\vvcenteredinclude{.1}{gd_c2_chord.pdf}}(\alpha',\beta')+f_{\vvcenteredinclude{.1}{gd_c2_chord.pdf}}(\alpha',\gamma')+f_{\vvcenteredinclude{.1}{gd_c2_chord.pdf}}(\gamma',\beta')=\\
   & \qquad\tfrac{1}{2}\bigl\langle \vvcenteredinclude{.3}{gd_c2_chord-left.pdf}+ \vvcenteredinclude{.3}{gd_c2_chord-mid.pdf}+\vvcenteredinclude{.3}{gd_c2_chord-right.pdf}, G_K\bigr\rangle.%\cdots \vvcenteredinclude{.3}{triangle-gd-abg.pdf}\cdots \bigr\rangle.
  \end{split}
 \end{equation} 

 The table shows all cases with the top arrow pointing right $\vvcenteredinclude{.3}{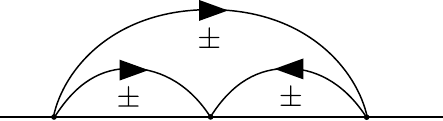}$, all cases with the top arrow pointing left $\vvcenteredinclude{.3}{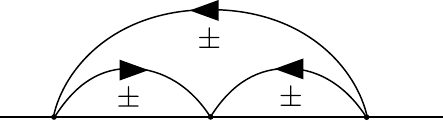}$ is obtained analogously by replacing diagrams $\vvcenteredinclude{.3}{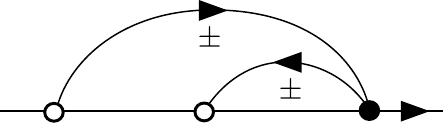}$ in the second column with $\vvcenteredinclude{.3}{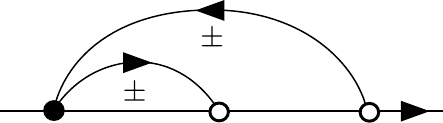}$. Note that   
  $\vvcenteredinclude{.3}{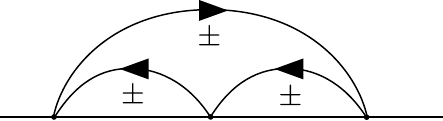}$ and $\vvcenteredinclude{.3}{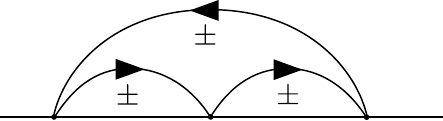}$ are invalid subdiagrams of $G_K$, thus they do not need consideration.   For each triple crossing listed in the table, there are only 2 possible crossing diagrams that can be obtained by perturbation.  Although these perturbed crossing diagrams may differ for a given triple crossing, the contribution to the count in \eqref{eq:c_2-via-K'} (listed in the last column of Table \ref{tb:triplecrossing}) by either perturbed crossing is the same.  
 
 Further, the entire sum in \eqref{eq:partitioned-sum} can be replaced with 
 \[
  \tfrac{1}{2}\bigl\langle \vvcenteredinclude{.3}{gd_c2_chord-left.pdf}+ \vvcenteredinclude{.3}{gd_c2_chord-mid.pdf}+\vvcenteredinclude{.3}{gd_c2_chord-right.pdf}, G_K\bigr\rangle,
 \]
 because each pair $\{\rho',\xi'\}\in \mathcal{E}^2_\bullet(G_{K'})$ is a part of a unique triple--crossing $\{\rho,\xi,\psi\}$ in $G_K$. From \eqref{eq:c_2-via-K'-chords}, \eqref{eq:c_2-via-K'-chords-2},  \eqref{eq:sum-E^2-circ} and  \eqref{eq:partitioned-sum} the formula \eqref{eq:gd_c2_mult-crossings-2} follows.
\qedsymbol{}

%%%%%%%%%%%%%%%%%%%%%%%%%%%%%%%%%%%%%%%%%%%%%

\section{Computing $c_2$ from a petal diagram}\label{sec:petal-diagrams}

Any knot $K$ has an {\em {\"u}bercrossing projection}, a projection which only has one crossing (called the  {\"u}bercrossing) and for which each strand bisects this crossing. The concept was coined in \cite{Adams:2012} by Adams et al, showing that every knot has a {\em petal diagram} - an {\"u}bercrossing projection which no loop is nested in the interior of another, see also \cite{Adams:2013,Adams:2015, Adams:2019, Colton:2019, Even-Zohar:2018}.

Given a petal diagram $P$ of a oriented knot $K$ with $n$ loops, follow $P$ in the direction of the orientation, starting with the top strand of the  {\"u}bercrossing.  By recording the order in which the levels of the  {\"u}bercrossing are traversed, one obtains a permutation of $1,\dots,n$ which begins with 1.  Together, with the orientation, this permutation completely identifies $K$.  Although the permutation is not necessarily unique, we may use it to compute the multiple-crossing Gauss diagram $G_P$ for $P$, and from $G_P$ compute $c_2(K)$.  We describe this process below for an oriented knot whose petal diagram is oriented counterclockwise; in the case that the petal diagram is oriented clockwise, small changes to the process of determining the associated long knot and determining the signs of the chords in the multi-crossing yield a similar result.

Let $K$ be an oriented knot $K$, whose petal diagram $P$ has $n$ loops, is oriented counterclockwise, and has associated permutation $(1, a_2, a_3,\dots, a_n)$.  
Consider the projection $P$ to sit within the cube $-1\leq x,y,z\leq1$ in $\R^3$.  $P$ can be transformed into a long knot diagram with one multi-crossing by cutting the petal labeled with $1$ and $a_2$.   Deform\footnote{If $P$ is oriented clockwise, reverse this. The associated multiple crossing Gauss diagram will have $n$ common endpoints, labeled, in order from left to right, as $1,a_2,a_{3},\dots,a_n$} $P$ so that the severed endpoint of strand $1$ is sent to $(\infty,0,0)$ and the severed end point of strand $a_2$ is sent to  $(-\infty,0,0)$.  A multiple crossing Gauss diagram for this long knot has $n$ common endpoints, labeled, in order from left to right, as $a_2,a_{3},\dots,a_n,1$.

\begin{figure}[ht]
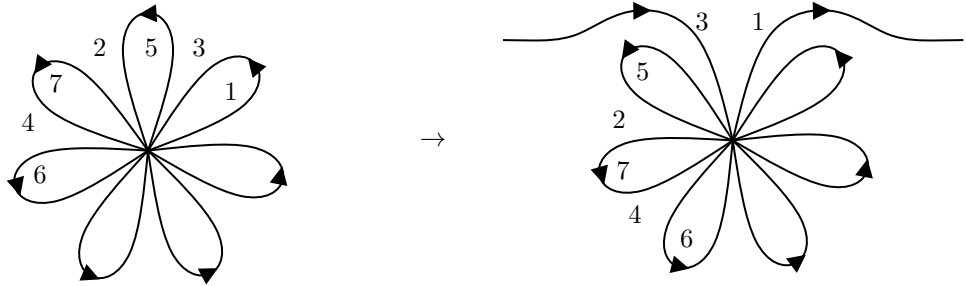

	\centering
	\begin{minipage}{0.45\textwidth}
       		 \centering
		\input{figs/petaldiagram3.pdf_tex}
	  \end{minipage}$\rightarrow$
    \begin{minipage}{0.45\textwidth}
           		 \centering
		\input{figs/petaldiagram3-long.pdf_tex}
	  \end{minipage}
\caption{The following oriented copy of $4_1$ is associated to the permutation $(1, 3, 5, 2, 7, 4, 6)$.}
\label{fig:petalexample}
	  \end{figure}

Every pair $a_i,a_j$, $1\leq i<j\leq n$ shares an arrow, $\alpha_{a_i,a_j}$, where the arrow points from max$\{a_i,a_j\}$ to min$\{a_i,a_j\}$, i.e. pointing from the lower level strand to the higher level strand in the crossing.  The sign for an arrow is shown in Figure \ref{fig:petal-crossing-signs}, and given as 
\[
sgn(\alpha_{a_i,a_j})=\begin{cases}(-1)^{j-i}& \textrm{ if }a_i<a_j,\\
(-1)^{j-i+1}& \textrm{ if }a_j<a_i.\end{cases}
\]

 The value of $j-i$ indicates how many strands are in between strand $a_i$ and $a_j$ in the petal projection.  In a petal diagram oriented counterclockwise, if $j-i$ is odd, then the strands,  in the projection, are traversed in opposite directions, and if this number is even, then they are traversed in the same direction. If $a_i$ is less than $a_j$, then the $i^{th}$ strand is projected on top of the $j^{th}$, and vice versa. This leads to the four possible cases in Figure ~\ref{fig:petal-crossing-signs}, which leads to the above formula\footnote{If $P$ is oriented clockwise, then these conditions are reversed.} for $sgn(\alpha_{a_i,a_j})$.   For example, in Figure \ref{fig:petalexample}, $a_1=1$ and $a_2=3$, $2-1=1$; we see that strand $1$ is traversed pointing upwards, and strand $3$ is traversed going downwards.  This leads to a negative crossing between strand $1$ and strand $3$.   Likewise, when computing the sign of the crossing involving strand $1$ and strand $5$, $a_3=5$, and $3-1=2$; we see that both strand $1$ and strand $5$ are traversed pointing upwards.  This leads to a positive crossing.
 
 %%%%%%%%%%%%%%%%%%%%%%%%%%%%%%%%%%%%%%%%%
\begin{figure}[h!]
		\centering
	\begin{minipage}{0.45\textwidth}
       		 \centering
	%% Creator: Inkscape 1.1-dev (48d18c81, 2020-07-08), www.inkscape.org
%% Accompanies image file 'pedal-crossing-odd-case1.pdf' (pdf, eps, ps)
\begingroup%
  \makeatletter%
  \providecommand\color[2][]{%
    \errmessage{(Inkscape) Color is used for the text in Inkscape, but the package 'color.sty' is not loaded}%
    \renewcommand\color[2][]{}%
  }%
  \providecommand\transparent[1]{%
    \errmessage{(Inkscape) Transparency is used (non-zero) for the text in Inkscape, but the package 'transparent.sty' is not loaded}%
    \renewcommand\transparent[1]{}%
  }%
  \providecommand\rotatebox[2]{#2}%
  \newcommand*\fsize{\dimexpr\f@size pt\relax}%
  \newcommand*\lineheight[1]{\fontsize{\fsize}{#1\fsize}\selectfont}%
  \ifx\svgwidth\undefined%
    \setlength{\unitlength}{126.11411642bp}%
    \ifx\svgscale\undefined%
      \relax%
    \else%
      \setlength{\unitlength}{\unitlength * \real{\svgscale}}%
    \fi%
  \else%
    \setlength{\unitlength}{\svgwidth}%
  \fi%
  \global\let\svgwidth\undefined%
  \global\let\svgscale\undefined%
  \makeatother%
  \begin{picture}(1,0.50783274)%
    \lineheight{1}%
    \setlength\tabcolsep{0pt}%
    \put(0,0){\includegraphics[width=\unitlength,page=1]{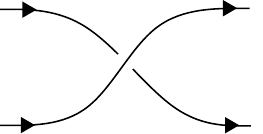}}%
    \put(0.76303369,0.35911849){\color[rgb]{0,0,0}\makebox(0,0)[lt]{\lineheight{1.25}\smash{\begin{tabular}[t]{l}$a_i$\end{tabular}}}}%
    \put(0.14454563,0.35911849){\color[rgb]{0,0,0}\makebox(0,0)[lt]{\lineheight{1.25}\smash{\begin{tabular}[t]{l}$a_j$\end{tabular}}}}%
    \put(0.58591615,0.25454735){\color[rgb]{0,0,0}\makebox(0,0)[lt]{\lineheight{1.25}\smash{\begin{tabular}[t]{l}$-$\end{tabular}}}}%
  \end{picture}%
\endgroup%

		  \end{minipage}
    \begin{minipage}{0.45\textwidth}
           		 \centering
	 %% Creator: Inkscape 1.1-dev (48d18c81, 2020-07-08), www.inkscape.org
%% Accompanies image file 'pedal-crossing-even-case1.pdf' (pdf, eps, ps)
\begingroup%
  \makeatletter%
  \providecommand\color[2][]{%
    \errmessage{(Inkscape) Color is used for the text in Inkscape, but the package 'color.sty' is not loaded}%
    \renewcommand\color[2][]{}%
  }%
  \providecommand\transparent[1]{%
    \errmessage{(Inkscape) Transparency is used (non-zero) for the text in Inkscape, but the package 'transparent.sty' is not loaded}%
    \renewcommand\transparent[1]{}%
  }%
  \providecommand\rotatebox[2]{#2}%
  \newcommand*\fsize{\dimexpr\f@size pt\relax}%
  \newcommand*\lineheight[1]{\fontsize{\fsize}{#1\fsize}\selectfont}%
  \ifx\svgwidth\undefined%
    \setlength{\unitlength}{126.11411642bp}%
    \ifx\svgscale\undefined%
      \relax%
    \else%
      \setlength{\unitlength}{\unitlength * \real{\svgscale}}%
    \fi%
  \else%
    \setlength{\unitlength}{\svgwidth}%
  \fi%
  \global\let\svgwidth\undefined%
  \global\let\svgscale\undefined%
  \makeatother%
  \begin{picture}(1,0.50783274)%
    \lineheight{1}%
    \setlength\tabcolsep{0pt}%
    \put(0,0){\includegraphics[width=\unitlength,page=1]{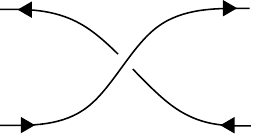}}%
    \put(0.76303369,0.35911849){\color[rgb]{0,0,0}\makebox(0,0)[lt]{\lineheight{1.25}\smash{\begin{tabular}[t]{l}$a_i$\end{tabular}}}}%
    \put(0.14454563,0.35911849){\color[rgb]{0,0,0}\makebox(0,0)[lt]{\lineheight{1.25}\smash{\begin{tabular}[t]{l}$a_j$\end{tabular}}}}%
    \put(0.58591615,0.25454735){\color[rgb]{0,0,0}\makebox(0,0)[lt]{\lineheight{1.25}\smash{\begin{tabular}[t]{l}$+$\end{tabular}}}}%
  \end{picture}%
\endgroup%

	  \end{minipage}
	  
	    \vspace{1cm}
	    
	\begin{minipage}{0.45\textwidth}
       		 \centering
	%% Creator: Inkscape 1.1-dev (48d18c81, 2020-07-08), www.inkscape.org
%% Accompanies image file 'pedal-crossing-odd-case2.pdf' (pdf, eps, ps)
\begingroup%
  \makeatletter%
  \providecommand\color[2][]{%
    \errmessage{(Inkscape) Color is used for the text in Inkscape, but the package 'color.sty' is not loaded}%
    \renewcommand\color[2][]{}%
  }%
  \providecommand\transparent[1]{%
    \errmessage{(Inkscape) Transparency is used (non-zero) for the text in Inkscape, but the package 'transparent.sty' is not loaded}%
    \renewcommand\transparent[1]{}%
  }%
  \providecommand\rotatebox[2]{#2}%
  \newcommand*\fsize{\dimexpr\f@size pt\relax}%
  \newcommand*\lineheight[1]{\fontsize{\fsize}{#1\fsize}\selectfont}%
  \ifx\svgwidth\undefined%
    \setlength{\unitlength}{126.12047908bp}%
    \ifx\svgscale\undefined%
      \relax%
    \else%
      \setlength{\unitlength}{\unitlength * \real{\svgscale}}%
    \fi%
  \else%
    \setlength{\unitlength}{\svgwidth}%
  \fi%
  \global\let\svgwidth\undefined%
  \global\let\svgscale\undefined%
  \makeatother%
  \begin{picture}(1,0.50780712)%
    \lineheight{1}%
    \setlength\tabcolsep{0pt}%
    \put(0,0){\includegraphics[width=\unitlength,page=1]{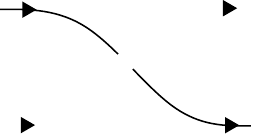}}%
    \put(0.76304564,0.35910037){\color[rgb]{0,0,0}\makebox(0,0)[lt]{\lineheight{1.25}\smash{\begin{tabular}[t]{l}$a_i$\end{tabular}}}}%
    \put(0.14458879,0.35910037){\color[rgb]{0,0,0}\makebox(0,0)[lt]{\lineheight{1.25}\smash{\begin{tabular}[t]{l}$a_j$\end{tabular}}}}%
    \put(0.58593704,0.2545345){\color[rgb]{0,0,0}\makebox(0,0)[lt]{\lineheight{1.25}\smash{\begin{tabular}[t]{l}$+$\end{tabular}}}}%
    \put(0,0){\includegraphics[width=\unitlength,page=2]{pedal-crossing-odd-case2.pdf}}%
  \end{picture}%
\endgroup%

		  \end{minipage}
    \begin{minipage}{0.45\textwidth}
           		 \centering
	 %% Creator: Inkscape 1.1-dev (48d18c81, 2020-07-08), www.inkscape.org
%% Accompanies image file 'pedal-crossing-even-case2.pdf' (pdf, eps, ps)
\begingroup%
  \makeatletter%
  \providecommand\color[2][]{%
    \errmessage{(Inkscape) Color is used for the text in Inkscape, but the package 'color.sty' is not loaded}%
    \renewcommand\color[2][]{}%
  }%
  \providecommand\transparent[1]{%
    \errmessage{(Inkscape) Transparency is used (non-zero) for the text in Inkscape, but the package 'transparent.sty' is not loaded}%
    \renewcommand\transparent[1]{}%
  }%
  \providecommand\rotatebox[2]{#2}%
  \newcommand*\fsize{\dimexpr\f@size pt\relax}%
  \newcommand*\lineheight[1]{\fontsize{\fsize}{#1\fsize}\selectfont}%
  \ifx\svgwidth\undefined%
    \setlength{\unitlength}{126.12047908bp}%
    \ifx\svgscale\undefined%
      \relax%
    \else%
      \setlength{\unitlength}{\unitlength * \real{\svgscale}}%
    \fi%
  \else%
    \setlength{\unitlength}{\svgwidth}%
  \fi%
  \global\let\svgwidth\undefined%
  \global\let\svgscale\undefined%
  \makeatother%
  \begin{picture}(1,0.50780712)%
    \lineheight{1}%
    \setlength\tabcolsep{0pt}%
    \put(0,0){\includegraphics[width=\unitlength,page=1]{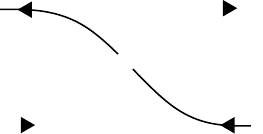}}%
    \put(0.76304564,0.35910037){\color[rgb]{0,0,0}\makebox(0,0)[lt]{\lineheight{1.25}\smash{\begin{tabular}[t]{l}$a_i$\end{tabular}}}}%
    \put(0.14458879,0.35910037){\color[rgb]{0,0,0}\makebox(0,0)[lt]{\lineheight{1.25}\smash{\begin{tabular}[t]{l}$a_j$\end{tabular}}}}%
    \put(0.58593704,0.2545345){\color[rgb]{0,0,0}\makebox(0,0)[lt]{\lineheight{1.25}\smash{\begin{tabular}[t]{l}$-$\end{tabular}}}}%
    \put(0,0){\includegraphics[width=\unitlength,page=2]{pedal-crossing-even-case2.pdf}}%
  \end{picture}%
\endgroup%

	  \end{minipage}
	  \caption{Possible configurations for strands within a multicrossing. Top row has $a_i<a_j$, so that $a_i$ is the over strand and the bottom row has $a_j<a_i$.  The left column has $j-i$ is odd, and the right column has $j-i$ is even.}
	    \label{fig:petal-crossing-signs}
\end{figure}
%%%%%%%%%%%%%%%%%%%%%%%%%%%%%%%%%%%%%%%%%

Once we have the multiple crossing Gauss Diagram, we count the relevant subdiagrams. A subdiagram is relevant if it represents a crossing that will give a non-zero contribution to the sum that determines the value of the Casson invariant.  If we count the subdiagrams algorithmically, we can track the value of the sum by adding the contribution of each subdiagram to a ``Casson count"; once we have iterated through all subdiagrams, this count will be equal to the value of the Casson invariant for that Gauss Diagram.  

 Let $1\leq i<j<k\leq n$;  the subdiagram on $a_i,a_j$ and $a_k$ is equal to one of the triple crossing diagrams $\vvcenteredinclude{.3}{triangle-tab-right-cases.pdf}$ or $\vvcenteredinclude{.3}{triangle-tab-left-cases.pdf}$ found in Table ~\ref{tb:triplecrossing} only if $a_j<a_i$ and $a_j<a_k$.  Therefore, one may iterate through the ordered common endpoints of the Gauss diagram $\{a_2, \dots, a_n, 1\}$ and find all such triples.  

For each such triple:
\begin{itemize}
\item If $sgn(\alpha_{a_i,a_j})=sgn(\alpha_{a_i,a_k})=sgn(\alpha_{a_j,a_k})$, then add $1$ to the Casson count.
\item Else if $a_k<a_i$ and $sgn(\alpha_{a_i,a_j})=sgn(\alpha_{a_i,a_k})$ and $sgn(\alpha_{a_i,a_k})=-sgn(\alpha_{a_j,a_k})$, then add $-1$ to the Casson count.
\item Else if $a_k>a_i$ and $sgn(\alpha_{a_j,a_k})=sgn(\alpha_{a_i,a_k})$ and $sgn(\alpha_{a_i,a_k})=-sgn(\alpha_{a_i,a_j})$, then add $-1$ to the Casson count.
\item Else add $0$ to the count.
\end{itemize}

In addition to the contributions from triple crossing diagrams, there are also contributions from quadruples of strands which represent embeddings of $\vvcenteredinclude{.3}{gd_c2_chord.pdf}$ in the multi-crossing Gauss diagram.  Let $1\leq i<j<k<l\leq n$; the subdiagram on $a_i,a_j, a_k$, and $a_l$ only if $a_k<a_i$ and $a_j<a_l$.  Again, for all such quadruples, if $sgn(\alpha_{a_i,a_k})=sgn(\alpha_{a_j,a_l})$, then $1$ is added to the Casson count, and if $sgn(\alpha_{a_i,a_k})=-sgn(\alpha_{a_j,a_l})$, then $-1$ is added to the Casson count.

For the petal diagram shown on Figure \ref{fig:petalexample}, the associated Gauss diagram will have common endpoints labeled in order from left to right as $3,5,2,7,4,6,1$.  The relevant triples crossing diagrams in this Gauss diagram will be given by strands $\{3,2,7\}$, $\{3,2,4\}$, $\{3,2,6\}$, $\{5,2,7\}$, $\{5,2,4\}$, $\{5,2,6\}$, $\{5,4,6\}$, and $\{7,4,6\}$.  The triples $\{3,2,7\}$, $\{3,2,6\}$, and $\{5,2,4\}$ contribute 1 each, the triple  $\{3,2,4\}$ contributes $-1$, and the other $4$ contribute $0$.  Additionally, we find that the quadruples $\{3,5,2,7\}$, $\{3,5,2,6\}$, and $\{5,2,4,6\}$ each contribute $-1$.  So, the final Casson count is $-1$, as desired. The above techniques yield a simple lower bound for both the {\em petal} and the {\em {\"u}bercrossing number} as stated in Corollary \ref{cor:uber-crossing},
	\[
	 |c_2(K)|\leq {\textit{\"u}(K) -1 \choose 3}+{ \textit{\"u}(K) -1 \choose 4}.
	\]
Indeed, the triple crossings contribute at most  ${\textit{\"u}(K) -1 \choose 3}$ to $c_2(K)$ and the quadruples at most ${ \textit{\"u}(K) -1 \choose 4}$.  Since the last endpoint on the multiple crossing Gauss diagram is always labeled as $1$, no triple or quadruple subset of endpoints which contains $1$ can satisfy the necessary equalities on the endpoints for that subset to contribute to the Casson count.  So, there are at most ${\textit{\"u}(K) -1 \choose 3}$ triples and ${ \textit{\"u}(K) -1 \choose 4}$ quadruples that can contribute.

%\bibliography{casson}
%\bibliographystyle{plain}  

\appendix

\section{Diffeology and differential forms on $\K$}\label{sec:diffeology}

A {\em diffeology} on a non-empty set $X$ is a set of parameterizations $\mathcal{D}=\{f:U\longrightarrow X\ |\ U\ \text{open in}\ \R^n,\ n=n(f)\}$ of $X$, called \emph {plots}, such that three axioms are satisfied:

\begin{itemize}
	\item[$(i)$] The set $\mathcal{D}$ contains all constant parameterizations $r\to x$ defined on $\R^n$, for all $x\in X$ and $n\geq 0$.
	
	\item[$(ii)$] Let $f : U \to X$ be a parameterization.  If for every point $r\in U$ there exists an open neighborhood $V$ of $r$ such that $f|_V$ belongs to $\mathcal{D}$, then $f$ belongs to $\mathcal{D}.$  
	
	\item[$(iii)$] For every plot $f : U \to X$ of $\mathcal{D}$, every real domain $V$, and every $g\in C^\infty(V,U),$ $f\circ g$ belongs to $\mathcal{D}$.
\end{itemize}

A \emph{long curve} is a smooth curve $\R\to\R^3$ which is the inclusion of the first coordinate axis except on a compact subset of $\R^3$. The \emph{space of long curves} will be denoted by $\textrm{Map}_c(\R,\R^3)$, and the long knots $\K$ is a subspace of $\textrm{Map}_c(\R,\R^3)$.  
A diffeology on the space of long knots $\K$ is induced from the canonical diffeology on $\textrm{Map}_c(\R,\R^3)$: a map $f:U\to \textrm{Map}_c(\R,\R^3)$ from an open $U\subset\R^l$ is a plot, if and only if the composite
\[
U\times \R\xrightarrow{\ f\times\textrm{id}\ } \textrm{Map}_c(\R,\R^3)\times \R\xrightarrow{\ \operatorname{ev}\ } \R^3
\]
 is smooth.  In \cite[Lemma~A.1.7]{Waldorf:2012}, it is shown that this diffeology on $\textrm{Map}_c(\R,\R^3)$ is equivalent to the \emph{smooth diffeology} on $\textrm{Map}_c(\R,\R^3)$, where plots are all smooth maps. 
A differential form $\omega\in \Omega^j(\K)$ on $\K$ is an assignment of a form $(\omega)_f$ on $U$ for each plot $f:U\longrightarrow \K$, such that, given a smooth map $h:V\longrightarrow U$ which induces a plot $f\circ h: V\longrightarrow \R^n$, we have: $(\omega)_{f\circ h}=h^\ast((\omega)_f).$

\end{document}